\documentclass[12pt]{article}

\usepackage{fullpage}
\usepackage{tabularx}
\usepackage{booktabs}

\usepackage{graphicx}
\graphicspath{ {./img/} }

\usepackage{caption}
\usepackage{float}

\usepackage{etoc}

\usepackage{natbib}
\usepackage{amsmath}
\usepackage{amssymb}
\usepackage{amsthm}
\usepackage{mathtools}
\mathtoolsset{showonlyrefs}
\usepackage{thmtools,thm-restate}

\usepackage{zref-clever}
\zcsetup{cap=true,nameinlink=true}
\usepackage[pdfencoding=auto, psdextra]{hyperref}
\hypersetup{
	colorlinks,
	citecolor=blue,
	filecolor=black,
	linkcolor=black,
	urlcolor=black
}

\usepackage{bookmark}
\bookmarksetup{depth=2}  

\newtheorem{lemma}{Lemma}
\newtheorem{remark}{Remark}
\newtheorem{corollary}{Corollary}
\newtheorem{proposition}{Proposition}

\DeclareMathOperator{\iid}{\stackrel{iid}{\sim}}
\DeclareMathOperator{\st}{\text{ s.t. }}
\DeclareMathOperator{\period}{\text{.}}
\DeclareMathOperator{\comma}{\text{, }}

\DeclareMathOperator{\since}{\text{since }}

\DeclareMathOperator{\vs}{\text{ versus }}
\DeclareMathOperator{\where}{\text{ where }}
\DeclareMathOperator{\by}{\text{by }}
\DeclareMathOperator{\textif}{\text{if }}
\DeclareMathOperator{\otherwise}{\text{otherwise }}
\DeclareMathOperator{\textand}{\text{ and }}
\DeclareMathOperator{\textor}{\text{ or }}
\DeclareMathOperator{\for}{\text{ for }}
\DeclareMathOperator{\wpa}{\text{ with prob. }}

\DeclareMathOperator{\supp}{supp}

\DeclareMathOperator*{\argmin}{arg\,min}
\DeclareMathOperator*{\argmax}{arg\,max}
\newcommand{\norm}[1]{\left\lVert #1 \right\rVert}

\DeclareMathOperator{\bin}{Bin}
\DeclareMathOperator{\Bin}{\bin}
\DeclareMathOperator{\N}{\mathbb{N}}
\DeclareMathOperator{\Lip}{Lip}

\DeclareMathOperator{\toD}{\overset{d}{\to}}

\let\tilde\widetilde
\let\hat\widehat

\begin{document}

\begin{center}

{\bf{\Large{Testing Random Effects for Binomial Data 
  }}}

\vspace*{.2in}

{{
\begin{tabular}{ccc}
Lucas Kania$^{\dagger}$  
& Larry Wasserman$^{\dagger, \ddagger}$
& Sivaraman Balakrishnan$^{\dagger,\ddagger}$ 
\end{tabular}
}}

\vspace{.15in}

\begin{tabular}{c}
	$^\dagger$Department of Statistics 
    and Data Science, Carnegie Mellon University\\
	$^\ddagger$Machine Learning Department, Carnegie Mellon University \\[0.12in]
\end{tabular}
\begin{tabular}{cc}
    \texttt{lucaskania@cmu.edu}, 
     \texttt{\{larry,siva\}@stat.cmu.edu} 
\end{tabular}

\vspace{.15in}

\today

\end{center}

\begin{abstract}
In modern scientific research, small-scale studies with limited participants are increasingly common. However, interpreting individual outcomes can be challenging, making it standard practice to combine data across studies using random effects to draw broader scientific conclusions. In this work, we introduce an optimal methodology for assessing the goodness-of-fit of a reference distribution for the random effects arising from binomial counts. For meta-analyses, we also derive optimal tests to evaluate whether multiple studies are in agreement before pooling the data. In all cases, we prove that the proposed tests optimally distinguish null and alternative hypotheses separated in the 1-Wasserstein distance.
\end{abstract}

\etocdepthtag.toc{mtchapter}
\etocsettagdepth{mtchapter}{section}
\etocsettagdepth{mtappendix}{none}
\etocsettagdepth{mtreferences}{section}
{
\normalsize
\parskip=0em
\renewcommand{\contentsname}{\normalsize Table of contents}
\tableofcontents
}

\pagebreak

\section{Introduction}

In performance evaluation, \citet{lordStrongTruescoreTheory1965,lordEmpiricalBayesProcedure1975} and \citet{grilliBinomialMixtureModeling2015} study the problem of estimating an individual's \textit{true} performance in a task based on a limited number of evaluations. They consider a scenario where $n$ individuals answer $t$ questions, and the number of correct responses is recorded. When the number of questions is small compared to the number of participants, making meaningful inferences about any single participant becomes challenging. However, aggregating data across all participants enables drawing conclusions about the population.

To facilitate such data pooling, the authors propose a binomial mixture model, where the probability of a correct response for each individual is drawn from an underlying distribution \begin{equation}\label{eq:binomial_mixture}
    X_i \mid p_i \sim \text{Bin}(t, p_i), \quad p_i \sim \pi, \quad 1 \leq i \leq n,
\end{equation} where $X_i$ represents the observed scores, $p_i$ are the \textit{true} scores of the candidates, and $\pi$ is the \textit{mixing distribution}, which captures the variability in individual abilities. The research primarily focuses on analyzing the properties of the mixing distribution, which may belong to a parametric or nonparametric family. For instance, while \citet{lordEstimatingTruescoreDistributions1969} allowed $\pi$ to be any distribution supported on $[0,1]$, \citet{thomasBinomialMixtureModel1989} and \citet{grilliBinomialMixtureModeling2015} restricted their analysis to finite binomial mixtures to better account for known individual differences in task performance.

Binomial mixtures also play an important role in other fields. They are used to account for variation in mouse mortality rates \citep{brooksFiniteMixtureModels1997}, word frequencies \citep{loweBetabinominalMixtureModel1999}, welfare program participation \citep{melkerssonWelfareParticipationWelfare2004}, genetic heterogeneity \citep{zhouBinomialMixtureModelbased2009}, genomic dependencies \citep{snipenMicrobialComparativePangenomics2009,hoggCharacterizationModelingHaemophilus2007}, and RNA composition \citep{jurgesDissectingNewlyTranscribed2018,linWellTEMPseqMicrowellbasedStrategy2023}. In all cases, assessing the compatibility between a reference mixing distribution and the data is crucial.

A central statistical question is how accurately the mixing distribution can be recovered. The main challenge is that a binomial mixture preserves only the first $t$ moments of the underlying distribution. Consequently, any two mixing distributions sharing these moments produce statistically indistinguishable observations. Research on this problem falls into two categories: methods that reliably recover the mixing distribution but require strong conditions on the number of trials and methods that avoid such conditions but offer weaker statistical guarantees.

\citet{teicherIdentifiabilityFiniteMixtures1963} studied the identifiability of finite mixing distributions. Building on his work, \citet{dedeckerMinimaxRatesConvergence2013,dedeckerImprovedRatesWasserstein2015,nguyenConvergenceLatentMixing2013,hoStrongIdentifiabilityConvergence2016} and \citet{heinrichStrongIdentifiabilityOptimal2018} analyzed the convergence of the method of moments (MOM) and the maximum likelihood estimator (MLE) under strong identifiability conditions, which require smoothness assumptions on the mixing distribution. In particular, \citet{manoleEstimatingNumberComponents2021c} investigated binomial mixtures under these conditions. More generally, \citet{yeBinomialMixtureModel2021a,tian2017} and \citet{vinayakMaximumLikelihoodEstimation2019} study the estimation of the mixing distribution without requiring identifiability conditions using the plug-in, MOM, and MLE estimators. Their findings indicate that, without stronger assumptions, standard estimators of the mixing distribution may be unreliable when the number of trials is much smaller than the number of studies.

In this work, we focus on testing rather than estimation. Instead of approximating the mixing distribution, we assess its proximity to a reference distribution. To compare them, we use the 1-Wasserstein distance, denoted by $W_1$ and defined in \zcref[S]{sec:problem}, since it does not impose restrictions on the support of the compared distributions.

We study goodness-of-fit testing, also known as the identity testing problem, under the binomial mixture model. Given an arbitrary reference mixing distribution $\pi_0$, called the null distribution, we aim to test whether the mixing distribution underlying the observed data in \eqref{eq:binomial_mixture} equals $\pi_0$ or deviates significantly from it as measured by the 1-Wasserstein distance: \begin{equation}
H_0: \pi = \pi_0 \vs H_1: W_1(\pi,\pi_0) \geq \epsilon \period
\end{equation}The parameter $\epsilon$ represents the separation between the hypotheses. When $\epsilon=0$, the hypotheses overlap, making them indistinguishable. Our goal is to determine the smallest $\epsilon$ for which a test can successfully differentiate the hypotheses while controlling the probability of making a mistake. To characterize the minimum separation, we adopt the non-parametric minimax framework for hypothesis testing, which can be traced back to the foundational work of \citet{mannChoiceNumberClass1942}, \citet{ingsterMinimaxNonparametricDetection1982,ingsterAsymptoticallyMinimaxHypothesisI1993}, \citet{ermakovAsymptoticallyMinimaxTests1990,ermakovMinimaxDetectionSignal1991}, and \citet{lepskiMinimaxNonparametricHypothesis1999}.

Intuitively, some null distributions are easier to test because fewer observations fall within the statistical fluctuations allowed under the null hypothesis. This phenomenon, known as locality, has been observed in binomial, Multinomial, and Poisson distributions, and smooth densities in the context of fixed effects testing \citep{valiantAutomaticInequalityProver2014,balakrishnanHypothesisTestingDensities2019,chhor2021}.

A key application of locality arises in statistical meta-analysis of treatment effectiveness. Suppose $n$ studies are conducted, each applying the treatment to $t$ participants. To obtain a precise estimate of treatment effectiveness, scientists pool the data across studies. However, before combining results, it is crucial to verify whether the treatment effect is homogeneous across studies. This requires testing if the mixing distribution is concentrated around a single point: \begin{equation}
H_0: \pi= \delta_{p_0} \vs H_1: W_1(\pi,\delta_{p_0})\geq \epsilon
\end{equation} where $p_0 \in (0,1)$, called the reference effect, is typically unknown. In clinical trials, $p_0$ can be very small, making commonly used asymptotically valid tests, such as Pearson's chi-squared test \citep{pearsonCriterionThatGiven1900} and Cochran's chi-squared test \citep{cochranMethodsStrengtheningCommon1954}, unreliable since their asymptotic approximations do not hold \citep{parkTestingHomogeneityProportions2019}.

\pagebreak

\subsection{Summary of main contributions}
In this work, we address the following questions:
\begin{itemize}
\item[(i)] How difficult is goodness-of-fit testing for an arbitrarily complex null hypothesis?
\item[(ii)] How difficult is homogeneity testing?
\end{itemize}

For goodness-of-fit testing, the worst case is as difficult as estimation. When the null hypothesis is arbitrarily complex, no consistent test exists for a fixed number of trials. The optimal testing algorithm combines a plug-in estimator of the Wasserstein distance with a debiased Pearson's chi-squared test. Our main contribution is proving its optimality using Kravchuk polynomials \citep{krawtchoukGeneralisationPolynomesHermite1929}, which are orthogonal under the binomial distribution. They enable unbiased estimation of the mixing distribution's moments and link probabilistic distances between marginal distributions to moment differences between the mixing distributions. Analogously to the Hermite and Charlier polynomials for the Normal and Poisson distributions \citep{wuPolynomialMethodsStatistical2020}, they significantly simplify constructing worst-case scenarios where all testing algorithms fail.

For homogeneity testing with respect to a reference effect, the optimal test combines two Kravchuk polynomials. The problem's difficulty depends on the null distribution's location. When there are few trials, the data sparsity induced by the null distribution significantly affects the hardness of the data, meaning that mixing distributions that generate non-sparse data are harder to test. Nevertheless, it remains possible to consistently test whether the mixing distribution concentrates around the reference effect, even with as few as two trials. This result extends to homogeneity testing without a reference effect for which an optimal test can be constructed by debiasing Cochran's chi-squared test. Both methods can be used to construct tighter confidence intervals and enhance the power of p-values in clinical meta-analyses.

\textbf{Outline} The remainder of the paper is organized as follows: \zcref[S]{sec:problem} introduces the minimax framework, and \zcref[S]{sec:general_testing} studies goodness-of-fit testing.  \zcref[S]{sec:homogeneity_testing} introduces the local minimax framework and applies it to the homogeneity testing with simple null distributions. \zcref[S]{sec:homogeneity_testing_unkown} extends the analysis to homogeneity testing with a composite null distribution. Finally, \zcref[S]{sec:applications} applies the methodologies to problems in the medical and political sciences.

\textbf{Notation} Henceforth, we write $a_n \lesssim b_n$ if there exists a positive constant $C$ such that $a_n \leq C \cdot b_n$ for all $n$ large enough. Analogously, $a_n \asymp b_n$ denotes that $a_n \lesssim b_n$ and $b_n \lesssim a_n$. Furthermore, given a distribution $\pi$, we use $m_l(\pi)=E_{p\sim \pi}[p^l]$ to denote its $l$-th moment, and $V(\pi)=m_2(\pi)-m_1^2(\pi)$ to denote its variance. Let $\delta_A(a)$ equal $1$ if $a \in A$, and $0$ otherwise. Alternatively, we use the following notation $I(\text{condition})$ to denote the indicator function that evaluates to $1$ whenever the condition is true and returns $0$ otherwise. Finally, let $a\wedge b = \min(a,b)$ and $a\lor b = \max(a,b)$.

\section{The minimax framework}\label{sec:problem}

We develop the minimax framework to evaluate test performance. Consider $n$ binomial observations, each consisting of $t$ trials. The random effects model is given by \eqref{eq:binomial_mixture}. The probability mass function of the binomial distribution is \begin{equation}
P(A) = \text{Bin}(t,p)(A) =
\sum_{j\in A}  \binom{t}{j} p^j (1-p)^{t-j} =
\sum_{j=0}^t\ \delta_A(j) \cdot B_{j,t}(p),
\end{equation} where $B_{j,t}(p) = \binom{t}{j} p^j (1-p)^{t-j}$ is the $j$-th element of the $t$-order
Bernstein basis. The marginal measure of the data follows a mixture of binomial distributions: \begin{equation}\label{eq:marginal_measures}
P_{\pi}(A) = \int_0^1 \text{Bin}(t,p)(A)\ d\pi(p) = \sum_{j=1}^t\ \delta_{A}(j) \cdot b_{j,t}(\pi) \for A \subseteq \{0,\dots,t\}
\end{equation} where $b_{j,t}(\pi) = E_{p \sim \pi}[B_{j,t}(p)]$ for $1\leq j \leq t$ are called the expected fingerprint under $\pi$, representing the probability of observing $j$ successes under the binomial mixture. Thus, \eqref{eq:binomial_mixture} is equivalent to stating that we have $n$ observations from the marginal measure: $
X_i \sim P_\pi$ for $1\leq i \leq n$.

A key challenge is that the marginal measure of the data $P_\pi$ preserves only the first $t$ moments of the mixing distribution. This can be easily seen by noting \eqref{eq:marginal_measures} depends on $\pi$ only through the expected fingerprint and rewriting them as a function of the first $t$ moments \begin{equation}\label{eq:expected_fingerprint}
b_{j,t}(\pi) = \sum_{l=j}^t \binom{t}{l}\binom{l}{v}(-1)^{l-j}\cdot m_l(\pi) \where m_l(\pi) = E_{p\sim\pi}[p^l] \period
\end{equation} Hence, tests can estimate only the first $t$ moments of the mixing distribution, making mixing distributions that share these moments indistinguishable.

Given a null distribution $\pi_0$, we wish to test whether a mixing distribution equals the null or differs from it: $
H_0: \pi=\pi_0$ versus $H_1: \pi \not= \pi_0$. However, under this alternative hypothesis, the mixing distributions may be arbitrarily close. If we want to distinguish them while providing statistical guarantees, we can only do so asymptotically. To provide a finite sample analysis, some separation between the hypotheses is required. In this work, we measure their separation under the 1-Wasserstein distance: \begin{equation}\label{eq:w1_testing}
H_0: \pi = \pi_0 \vs H_1: W_1(\pi,\pi_0) \geq \epsilon.\end{equation} The 1-Wasserstein distance quantifies the cost of transporting mass from $\pi$ to $\pi_0$ \citep{villaniOptimalTransport2009,santambrogioOptimalTransportApplied2015}. Let $D$ denote the set of all distributions supported on $[0,1]$ and $\Gamma$ be the set of distributions supported on $[0,1]^2$ whose marginals are $\pi$ and $\pi_0$, then $W_1$ is defined as: $W_1(\pi,\pi_0) = \inf_{\gamma \in \Gamma}E_{(p,q)\sim \gamma} \left|p-q\right|$. Moreover, it admits a dual representation as an integrated probability metric \citep{kantorovich1958space}: \begin{equation}\label{eq:w1_dual}
W_1(\pi,\pi_0) = \sup_{f \in \Lip_1[0,1]} E_{p\sim\pi}\left[f(p)\right] - E_{p\sim\pi_0}\left[f(p)\right]
\end{equation} where $\Lip_1[0,1]$ is the set of all 1-Lipschitz functions supported on $[0,1]$, and the function achieving the supremum is called the witness function.

Our goal is to understand how small $\epsilon$ in \eqref{eq:w1_testing} can be such that there exists a test that consistently distinguishes the hypotheses. A test $\psi$ maps the sample space to $\{0,1\}$, returning $0$ if it considers that the data supports the null hypothesis and $1$ otherwise. Its type I error is measured as the probability of choosing the alternative hypothesis when the null hypothesis is true. Let $\Psi(\pi_0)$ be the set of all tests that control the type I error by $\alpha \in (0,1)$ for the null distribution $\pi_0$, called valid tests: $
\Psi(\pi_0)=\left\{\psi : P_{\pi_0}^n(\psi(X) = 1)\leq \alpha \right\}$ where $X=(X_1,\dots,X_n)$ is a vector containing $n$ observations. The risk of a valid test is given by its maximum type II error, i.e., the probability of choosing the null hypothesis when the alternative is true: \begin{equation}
R(\psi,\pi_0,\epsilon) = \sup_{W_1(\pi,\pi_0)\geq \epsilon} P_\pi^n(\psi(X)=0).
\end{equation} We call any valid test that controls the risk by $\beta \in (0,1-\alpha)$ powerful. The minimax risk quantifies the best performance over all valid tests under the worst-case null distribution \begin{equation}\label{eq:global_and_local_risk}
R_*(\epsilon) = \sup_{\pi_0\in D}R_*(\epsilon,\pi_0) \where R_*(\epsilon,\pi_0)=\inf_{\psi \in \Psi(\pi_0)}R(\psi,\pi_0,\epsilon) \period
\end{equation} A valid test is called minimax or optimal if its risk equals the minimax risk. For a fixed sample size $n$ and number of trials $t$, there might not exist any powerful valid test if $\epsilon$ is too small. Thus, we define the critical separation as the smallest $\epsilon$ such that there exists a powerful valid test: \begin{equation}\label{eq:critical_separation}
\epsilon_*(n,t) = \inf \left\{\epsilon : R_*(\epsilon) \leq \beta \right\}
\end{equation}For any $\epsilon < \epsilon_*$, no valid test can control the type II error by $\beta$. Furthermore, if $\liminf_n \epsilon _*(n,t)=0$, we say that there exists a sequence of tests that consistently distinguishes the hypotheses. That is, there exists a sequence of tests that detects any deviation from the null distribution in the limit of infinite observations.

To characterize the critical separation, we derive upper and lower bounds that match up to constants. Given any valid test $\psi$ that controls the risk by $\beta$, its risk function provides an upper bound on the critical separation: $$
\epsilon_*(n,t) \leq \inf\{\epsilon: \sup_{\pi_0 \in D}R(\psi,\pi_0,\epsilon) \leq \beta\}.$$ In the simplest case, a lower bound follows by identifying two mixing distributions, $\pi_0$ and $\pi_1$, that are far when measured by the Wasserstein distance but induce statistically indistinguishable marginal measures, e.g., $P_{\pi_0} {\buildrel d \over =} P_{\pi_1}$. Consequently, no test can distinguish them and the critical separation is lower-bounded by $\epsilon_*(n,t) \geq W_1(\pi_1,\pi_0)$. Our goal is to find upper and lower bounds of the critical separation that match up to constants.

\section{Goodness-of-fit testing}\label{sec:general_testing}

We derive the critical separation for goodness-of-fit testing under arbitrary null distributions. \zcref[S]{sec:plugin_test,sec:debiased_pearson_test} analyze the performance of the plugin and debiased Pearson chi-squared tests. \zcref[S]{sec:lower_bounds_gof} shows that the plugin test is optimal when $t > n$, while the debiased Pearson chi-squared test is optimal when $t < n$. 

\subsection{The plugin test}\label{sec:plugin_test}

Since we aim to detect deviations under the Wasserstein distance, a natural approach is to use the plug-in estimator. Consider the Wasserstein distance between the empirical and null distributions $\hat{W}_1(X) = W_1(\hat{\pi},\pi_0)$ where $\hat{\pi} = n^{-1}\sum_{i=1}^n\delta_{X_i/t} \period$ We reject the null hypothesis when $\hat{W}_1(X)$ exceeds its $1-\alpha$ quantile under the null distribution, denoted by $q_{\alpha}(P_{\pi_0}, \hat{W}_1)$, ensuring type I error control:
\begin{equation}\label{eq:w1_plugin_test}
\psi^\alpha_{\hat{W}_1}(X) = I\left( \hat{W}_1(X) \geq q_{\alpha}(P_{\pi_0},\hat{W}_1)\right) \period
\end{equation} Intuitively, the test works whenever $W_1(\hat{\pi},\pi_0)$ is a good approximation of $W_1(\pi,\pi_0)$. To understand the quality of the approximation, it is useful to separate the uncertainty due to the mixing distribution and the binomial sampling. Let $p_i$ be the random effect associated with $X_i$ in \eqref{eq:binomial_mixture} and define the unobserved empirical measure of random effects $
\tilde{\pi}=n^{-1}\sum_{i=1}^n\delta_{p_i}$. By the triangle inequality, it holds that \begin{equation}\label{eq:trig_ineq}
|W_1(\pi,\pi_0) - W_1(\hat{\pi},\pi_0)| \leq   W_1(\pi,\tilde{\pi}) + W_1(\tilde{\pi},\hat{\pi})\period
\end{equation} Hence, whenever the right-hand side is small, the plug-in statistic is a good approximation of $W_1(\pi,\pi_0)$. Since $\tilde{\pi}$ is an empirical measure sampled from $\pi$, Theorem 3.2 of \citet{bobkovOnedimensionalEmpiricalMeasures2019} tells us that the first term should not be large on average: $
E\left[W_1(\pi,\tilde{\pi})\right] \leq J(\pi)/\sqrt{n}$ where $J(\pi)=\int_{0}^1 \sqrt{F_\pi(x)(1-F_\pi(x))}\ dx$ and $F_{\pi}(x) = \int_0^x d\pi$ is the cumulative distribution function of $\pi$.

For the second term, conditioning on $\tilde{\pi}$, the concentration of $\hat{\pi}$ around $\tilde{\pi}$ can be measured by a Bernstein bound \citep{yeBinomialMixtureModel2021a}. Using this strategy and localizing the result around the null distribution, we derive the separation rates for the plug-in test. The proof is in \zcref[S]{sec:plugin_test_proof}.

\begin{restatable}{theorem}{PlugInTestRates}\label{lemma:plugin_test}
For hypotheses \eqref{eq:w1_testing}, the plug-in test \eqref{eq:w1_plugin_test} controls the type I error by $\alpha$. Moreover, there exists a universal positive constant $C$ such that the test controls the type II error by $\beta$ whenever
\begin{equation}
\epsilon \geq C \cdot \left[\frac{J(\pi_0)}{\sqrt{n}} + \sqrt{\frac{E_{p\sim\pi_0}\left[p(1-p)\right]}{t}} + \frac{1}{n} + \frac{1}{t}\right] \period
\end{equation}
\end{restatable}

The first term indicates that the plug-in test adapts to the concentration of the null distribution, vanishing when $\pi_0$ is a point mass. The second term accounts for the adaptation to low binomial variance, meaning that the location of $\pi_0$ influences detection. When $\pi_0$ concentrates near the endpoints of $[0,1]$, the test detects smaller deviations from the null distribution. In the worst case, when $\pi_0$ does not concentrate, the plug-in test requires the separation between hypotheses to satisfy $\epsilon \gtrsim n^{-1/2} + t^{-1/2}$.

For $ n \geq t$, the test achieves a parametric rate. However, when $t$ is smaller, the term $t^{-1/2}$ dominates the separation. Intuitively, this occurs because $\hat{\pi}$ approximates the mixing distribution by coarsening it into a mixture of $t+1$ point masses. We can gain some insight by using approximation theory. Let $n_j$ denote the $j$th observed fingerprint: \begin{equation}\label{eq:observed_fingerprint}
n_j =  \sum_{i=1}^n I(X_i = j) \quad \text{for } j \in \{0,\dots,t\}\period
\end{equation} The empirical distribution can be written as $\hat{\pi}=\sum_{j=0}^t \delta_{j/t} \cdot n_j/n$. Taking the limit as $n\to\infty$, so that the only remaining uncertainty is due to the number of trials, we obtain $\tilde{\pi}\toD\pi$ and $\hat{\pi}\toD\sum_{j=0}^t \delta_{j/t} \cdot b_{j,t}(\pi)$. Thus, by \eqref{eq:trig_ineq} the error of the plug-in estimate is dominated by $W_1(\hat{\pi},\pi)$. Exploiting the duality of $W_1$ \eqref{eq:w1_dual}, we obtain the following representation of the error \begin{equation}
W_1(\hat{\pi},\pi) = \sup_{f\in \Lip[0,1]}E_{p\sim \pi}\left[f(p) - B_{t}(f,p) \right] \where B_{t}(f,p)=\sum_{j=0}^t f\left(j/t\right) \cdot B_{j,t}(p)
\end{equation}is the $t$-order Bernstein polynomial approximation of $f$. Thus, the plug-in approach is constrained by the error of approximating Lipschitz functions by Bernstein polynomials, which is of order $t^{-1/2}$ \citep{bustamanteBernsteinOperators2017}.

This error is known to be suboptimal. If we consider all polynomials of degree $t$, then the best uniform approximation of a Lipschitz function has an error of order $t^{-1}$ \citep{plaskotaApproximationComplexity2021}. Therefore, the plug-in test can be improved by modifying the polynomial approximation of the witness function in \eqref{eq:w1_dual}. In the following section, we show that directly comparing the observed fingerprints to their expected values under the null hypothesis can be related to approximating the witness function optimally.

\subsection{The debiased Pearson's chi-squared test}\label{sec:debiased_pearson_test}

Since $P_\pi$ is a $t$-dimensional Multinomial, a natural approach is to use a chi-squared test to compare the observed and expected fingerprints. \citet{indyk2003fast, weedSharpAsymptoticFinitesample2019} show that the $W_1$ distance is closely tied to the $\ell_1$ distance, suggesting that a chi-squared test that is powerful under $\ell_1$ separation can detect deviations under $W_1$. Formally, we want to reduce problem \eqref{eq:w1_testing} to observing $\left(n_1,\dots,n_t\right) \sim \text{Multinomial}(n,b_t(\pi))$ and testing \begin{equation}\label{eq:fingerprint_testing}
H_0: b_t(\pi) = b_t(\pi_0) \quad \text{v.s.} \quad  H_1: \norm{b_t(\pi)-b_t(\pi_0)}_1 \geq \epsilon\period
\end{equation} To relate the Wasserstein distance between the mixing distributions to the $\ell_1$ distance between the fingerprints, we use the duality of the Wasserstein distance \eqref{eq:w1_dual}. Let $\mathcal{P}_k$ be the space of polynomials supported on $[0,1]$ of degree at most $k$. Then, approximating the witness function by the best $k$-order polynomial gives the bound \begin{equation}\label{eq:up1}
W_1(\pi,\pi_0) \leq 2 \cdot \inf_{p_k \in \mathcal{P}_k}\norm{f-p_k}_\infty + \left|\ E_{q\sim\pi}\left[p_k(q)\right] - E_{q\sim\pi_0}\left[p_k(q)\right]\ \right| \period
\end{equation} where $\norm{f-p_k}_\infty=\sup_{x\in[0,1]}|f(x)-p_k(x)|$. The second term compares the first $k$ moments of $\pi$ and $\pi_0$. Since $P_\pi$ contains only information about the first $t$ moments, we consider only polynomials of degree at most $t$. To compare fingerprints instead of moments, we express $p_k$ in the $t$-order Bernstein basis: $
p_k(x) = \sum_{j=0}^t c_{j,k} \cdot B_{j,t}(x)$ for $k\leq t$. Recalling that $E_{q\sim \pi}[B_{j,t}(q)]=b_{j,t}(\pi)$ and propagating the expectation in \eqref{eq:up1}, the second term becomes bounded by the fingerprint difference under $\ell_1$ \begin{align}\label{eq:w1_to_l1}
W_1(\pi,\pi_0) \leq 2 \cdot \inf_{p \in \mathcal{P}_k}\norm{f-p}_\infty + \norm{c_k}_\infty \cdot \norm{b_t(\pi)-b_t(\pi_0)}_1 \period
\end{align} Thus, optimally reducing testing under $W_1$ \eqref{eq:w1_testing} to fingerprint testing under $\ell_1$ \eqref{eq:fingerprint_testing} requires balancing the uniform approximation error in the first term of \eqref{eq:w1_to_l1} with the fingerprint separation under $\ell_1$ in the second term of \eqref{eq:w1_to_l1}.

Testing Multinomials under the $\ell_1$ distance has been previously studied by \citet{batuTestingThatDistributions2000}, \citet{paninskiCoincidenceBasedTestUniformity2008,valiantAutomaticInequalityProver2014,chanOptimalAlgorithmsTesting2014}, \citet{balakrishnanHypothesisTestingDensities2019} and  \citet{chhor2021}, among others. We
revisit the minimax test for \eqref{eq:fingerprint_testing} using the machinery of the Kravchuk polynomials, which are central to both upper and lower bounds in this work. 

Let $m\in\{0,\dots,t\}$ and $p\in(0,1)$, the $m$-th normalized Kravchuk polynomial \citep{krawtchoukGeneralisationPolynomesHermite1929,szegoOrthogonalPolynomials1975,dominiciAsymptoticAnalysisKrawtchouk2005} is given by \begin{equation}
\tilde{K}_m(x,p,t) = \binom{t}{m}^{-1}\sum_{v=0}^m (-1)^{m-v} \binom{t-x}{m-v}\binom{x}{v}p^{m-v}(1-p)^v\period
\end{equation} They are the orthogonal with respect to the binomial distribution \citep{szegoOrthogonalPolynomials1975}, and provide unbiased estimates of the first $t$ moments of the mixing distribution; we defer the reader to \zcref{sec:kravchuk} for further details. 

To solve the fingerprint testing problem \eqref{eq:fingerprint_testing}, we construct an unbiased estimator of the $\ell_2$ distance between the observed and expected fingerprints using the second Kravchuk polynomial, and reject whenever the estimate is large \begin{equation}\label{eq:fingerprint_test}
\psi_{\hat{\ell}_2}^\alpha = I\left(\ \hat{\ell}_2(X) > q_\alpha(P_{\pi_0},\hat{\ell}_2) \right) \where \hat{\ell}_2(X) = \frac{1}{t+1}\sum_{j=0}^t \frac{\tilde{K}_2(n_j,b_{j,t}(\pi_0),n)}{\max\left(\frac{1}{t+1},\mu_{b_{j,t}(\pi_0)}\right)}
\end{equation} where $\mu_{b_{j,t}(\pi_0)}=b_{j,t}(\pi_0)\cdot (1-b_{j,t}(\pi_0))$. We call the corresponding test the debiased Pearson's chi-squared test, since expanding the definition of $\tilde{K}_2$, it is clear that it is a centered chi-squared statistic \begin{equation}
\tilde{K}_2(n_j,b_{j,t}(\pi_0),n) =  \left(\frac{n_j}{n}-b_{j,t}(\pi_0)\right)^2-\frac{\mu_{n_j/n}}{n-1} \period
\end{equation} where $\mu_{n_j/n}=(n_j/n)\cdot (1-n_j/n)$. Unlike the usual chi-squared statistic, the centering is random and debiases the estimator \citep{valiantAutomaticInequalityProver2014}. \citet{balakrishnanHypothesisTestingDensities2019} proved that this test is minimax optimal for testing under the $\ell_1$ distance.

\begin{lemma}[Theorem 3.2 of  \citet{balakrishnanHypothesisTestingDensities2019}]\label{thm:multinomial_l1_rates}
For hypotheses \eqref{eq:fingerprint_testing}, the debiased Pearson's $\chi^2$ test \eqref{eq:fingerprint_test} controls type I error by $\alpha$. Furthermore, there exists a constant $C>0$ such that the type II error of the test is bounded by $\beta$ whenever $
\epsilon \geq C \cdot t^{1/4}/\sqrt{n}$. 
\end{lemma} Although the above result was originally derived for independent fingerprints, it holds for the case of dependent fingerprints due to the Poissonization trick; see Appendix C of \citet{canonneTopicsTechniquesDistribution2022}. By leveraging \zcref{thm:multinomial_l1_rates} and approximating the witness function in \eqref{eq:w1_to_l1}, the following lemma shows that the debiased Pearson's chi-squared test improves over the plug-in test when $t \lesssim n$. The proof can be found in \zcref[S]{sec:fingerprint_test_proof}.  \begin{restatable}{theorem}{FingerprintTest}\label{lemma:fingerprint_test} For hypotheses \eqref{eq:w1_testing}, the debiased Pearson's $\chi^2$ test \eqref{eq:fingerprint_test} controls type I error by $\alpha$. Furthermore, for any constant $\delta > 0$, there exists positive constant $C_\delta$ depending on $\delta$ such that the type II error is bounded by $\beta$ whenever \begin{equation}\label{eq:minimax_rates}
\epsilon \geq C_\delta \cdot \begin{dcases}
\frac{1}{t} &\for  t \lesssim \log n \\
\frac{1}{\sqrt{t \log n}} &\for \log n \lesssim t \lesssim \frac{n^{1/4-\delta}}{(\log n)^{3/8}}
\end{dcases}
\end{equation}
\end{restatable} 

\subsection{Lower bounds on the critical separation}\label{sec:lower_bounds_gof}

In this section, we argue that combining the plugin and debiased Pearson's $\chi^2$ test via a Bonferroni correction: \begin{equation}\label{eq:global_minimax_test}
\psi_{GM}^{\alpha}(X) = \max\left(\ \psi_{\hat{\ell}_2}^{\alpha/2}(X)\comma \psi_{\hat{W}_1}^{\alpha/2}(X)\ \right) \period
\end{equation} is optimal, up to constants and $\log n$ factors, outside the region $n^{1/4-\delta}\lesssim t \lesssim n$. 

The lower bounds on the critical separation \eqref{eq:critical_separation} rely on Le Cam's two-point method. We find $\pi_0$ and $\pi_1$ that maximize $W_1(\pi_1,\pi_0)$ while ensuring they induce similar marginal distributions, i.e., their total variation distance $V(P_{\pi_0}, P_{\pi_1})$ is small. Prior work \citep{tian2017,vinayakMaximumLikelihoodEstimation2019,kong2017} has established the following results by explicitly constructing mixing distributions or cleverly bounding the total variation distance. Here, we present an alternative proof for those results using approximation theory and Kravchuk polynomials, which avoids explicit construction and simplifies controlling the total variation between marginal distributions.

Le Cam's method states that an upper bound on the total variation between the marginal measures, denoted by $V$, implies a lower bound on the critical separation. \begin{restatable}[Le Cam's method, Theorem 2.2 of \citet{tsybakovIntroductionNonparametricEstimation2009}]{lemma}{LeCamLowerBound} \label{lemma:lecam_lb} Let $C_\alpha=1-(\alpha+\beta)$. For any $\pi_0,\pi_1 \in D$ such that $V\left(P_{\pi_0},P_{\pi_1}\right) < C_\alpha/2n$, the critical separation is lower bounded $\epsilon_*(n,t) \geq W_1(\pi_0,\pi_1)$.\end{restatable} The key to simplifying this optimization is relating the total variation distance between the marginal measures to a suitable distance between the mixing distributions. Since the binomial density preserves only the first $t$ moments of the mixing distributions, it is natural to compare them using moment differences. 
\begin{lemma}\label{lemma:tvbound} For any $p\in (0,1)$, and mixing distribution $\pi_0$ and $\pi_1$ in $D$, it holds that \begin{equation}
V\left(P_{\pi_0},P_{\pi_1}\right) \leq \frac{1}{2} \cdot \sqrt{M_p(\pi_0,\pi_1)} \where M_p(\pi_0,\pi_1)=\sum_{m=1}^t \binom{t}{m} \cdot \frac{\Delta_m^2(\pi_1,\pi_0)}{\mu_p^m},
\end{equation} $\mu_{p}=p\cdot (1-p)$ and $\Delta_m(\pi_1,\pi_0) = E_{u\sim\pi_1}[u-p]^m-E_{u\sim\pi_0}[u-p]^m$ is the $m$-th moment difference.
\end{lemma} The proof, found in \zcref[S]{sec:tvbound_proof}, relies on the orthogonality of the Kravchuk polynomials. It follows a structure similar to analogous results for mixtures of Gaussian or Poisson distributions, which use Hermite and Charlier polynomials~\citep{wuOptimalEstimationGaussian2019}. By \zcref{lemma:lecam_lb,lemma:tvbound}, finding a lower bound on the critical separation reduces to: \begin{equation}\label{eq:opt_2}
\epsilon_*(n,t) \geq \sup_{\pi_0,\pi_1 \in D}W_1(\pi_0,\pi_1) \st M_p(\pi_0,\pi_1) < (C_\alpha/n)^2\period
\end{equation} To connect the problem to approximation theory, let $D_\delta$ be the set of distributions distributions supported on $[1/2-\delta,1/2+\delta]$, and choose $(\delta,L)$ so that any pair of distributions in $D_{\delta}$ that share $L$ moments satisfy $M(\pi_0,\pi_1) < (C_\alpha/n)^2$ , then \eqref{eq:opt_2} is lower-bounded by: \begin{equation}\label{eq:opt_3}
\epsilon_*(n,t) \geq \sup_{\pi_0,\pi_1 \in D_{\delta}}W_1(\pi_0,\pi_1) \st m_l(\pi_0)=m_l(\pi_1) \for 1 \leq l \leq L\period
\end{equation}%
Recall from \eqref{eq:w1_dual} that the Wasserstein distance equals the maximum mean difference over all 1-Lipschitz functions. A lower bound for \eqref{eq:opt_3} arises by selecting a specific 1-Lipschitz function, such as the absolute value: \begin{align}\label{eq:opt_4}
\epsilon_*(n,t) \geq M(\delta,L) \where M(\delta,L)=&\sup_{\pi_0,\pi_1 \in D_\delta}E_{p\sim\pi_1}|p|-E_{p\sim\pi_0}|p|\\
&\st m_l(\pi_0)=m_l(\pi_1) \for 1 \leq l \leq L\period
\end{align} The dual of this moment-matching problem is the best polynomial approximation of the absolute value function. \begin{lemma}[Appendix E of \citet{wuMinimaxRatesEntropy2016}]\label{lemma:best_poly_approx} Let $P_L$ be the set of all $L$-order polynomials supported on $[-\delta,\delta]$, then $$M(\delta,L) = 2 \cdot A(\delta,L) \where A(\delta,L) =\inf_{f \in P_L}\sup_{|x|\leq \delta}\left||x|-f(x)\right|.$$
\end{lemma}
\citet{bernsteinOrdreMeilleureApproximation1912} studied the best polynomial approximation of the absolute value function and proved that: $
A(\delta,L) = (\beta_1 + C_{L}) \cdot \delta/L
$ where $C_{L} \to 0$ as $L \to \infty$ and $\beta_1 \approx 0.29$. Consequently, using \eqref{eq:opt_4} and \zcref{lemma:best_poly_approx}, we obtain a lower-bound on the critical separation: $
\epsilon_*(n,t) \geq 2(\beta_1 + C_{L}) \cdot \delta/L$. By adjusting the support of the mixing distributions and the number of matched moments, one can derive different lower bounds on the critical separation. This technique, known as the \textit{moment-matching} method \citep{ingsterTestingHypothesisWhich2001,wuPolynomialMethodsStatistical2020}, yields the following result. \begin{restatable}{theorem}{GlobalMinimaxRates}\label{lemma:global_minimax_rates}
The critical separation \eqref{eq:critical_separation} is lower-bounded by \begin{equation}\label{eq:global_minimax_rates} \epsilon_*(n,t) \gtrsim \max\left(\frac{1}{t}\ ,\ \frac{1}{\sqrt{t \log n}}\ ,\ \frac{1}{\sqrt{n}} \right).
\end{equation} \end{restatable} The proof is provided in \zcref[S]{sec:global_minimax_lowerbounds}. The lemma states that whenever the number of trials is much smaller than the sample size, meaning $t \lesssim \log n$, the statistical uncertainty with respect to $n$ is negligible. This follows from \zcref{lemma:tvbound}, which shows that two mixing distributions that share $t$ moments produce identical marginal measures. The farthest that two such distributions can be under $W_1$ is $1/t$. For applications, this implies that if $t$ is constant, increasing the number of observations does not help us distinguish the hypotheses, which motivates the local analysis in \zcref[S]{sec:homogeneity_testing}.

When the number of trials is larger than the sample size, i.e., $\log n/n \lesssim t$, we are in the
parametric regime. This can be foreseen by using an asymptotic argument. Scale the data $\tilde{X}_i = X_i/t$ and let $\tilde{P}_\pi$ denote its distribution. For a function $f$, let $B_t(f)$ denote its Bernstein polynomial approximation. If $f$ is a continuous bounded function, $B_t(f)$ converges to $f$ uniformly on $[0,1]$. Consequently, $\tilde{P}_\pi$ converges weakly to $\pi$ since $
\lim_{t\to\infty} E_{\tilde{P}_\pi}[f(X)]=\lim_{t\to\infty} E_{\pi}[B_t(f,p)]=E_{\pi}[f(p)]\period$ At the limit, the testing problem \eqref{eq:w1_testing} reduces to observing $\tilde{X}_i \iid \pi $ for $i \in\{1,\dots,n\}$ and testing $H_0: \pi = \pi_0$ versus  $H_1: W_1(\pi,\pi_0) \geq \epsilon$, for which the critical separation is known to be $\epsilon \asymp n^{-1/2}$ \citep{baSublinearTimeAlgorithms2011}.

In the intermediate regime, where $\log n \lesssim t \lesssim \log n/n$, the lower bound follows from constructing mixing distributions that match $t\log n$ moments, which is the optimal number of moments that one needs to estimate to approximate the underlying mixing distribution in this regime \citep{vinayakMaximumLikelihoodEstimation2019}.

Up to constants, the lower-bound \eqref{eq:global_minimax_rates} matches the separation rates achieved by the test \eqref{eq:global_minimax_test}, which indicates its optimality up to constants. Moreover, the critical separation matches the known estimation rates \citep{tian2017,vinayakMaximumLikelihoodEstimation2019}. Thus, testing the goodness of fit of arbitrary distributions is as hard as estimating the underlying mixing distribution.

\subsection{Simulations}\label{sec:general_testing_simulations}

To compare the plug-in test \eqref{eq:w1_plugin_test} and debiased Pearson's $\chi^2$ test \eqref{eq:fingerprint_test}, we evaluate them against the distributions used to produce the lower bounds on the critical separation. The plug-in test \eqref{eq:w1_plugin_test} achieves parametric rates for large $t$. To obtain that behaviour, the lower bound is constructed using the following family of distributions, where the probabilities of the null mixing distribution are perturbed.  \begin{equation}\label{eq:low_prob_perturbation}
\pi_0 = \frac{1}{2} \cdot \left( \delta_0 +  \delta_1 \right) \textand \pi_1 = \left(\frac{1}{2}-\epsilon\right) \cdot \delta_0 + \left(\frac{1}{2}+\epsilon\right) \cdot  \delta_1 \quad \for 0\leq \epsilon\leq \frac{1}{2} \period
\end{equation} Under these mixing distributions, the marginal measures are supported at $\{0,t\}$. Thus, increasing the number of trials does not yield additional information, limiting the power a test can achieve as $t$ grows.  \zcref[S]{fig:power_indentity_testing_prob_perturbation} confirms that all tests have the same power curve regardless of $t$.

\begin{figure}[ht]
\centering
\includegraphics[width=0.95\linewidth]{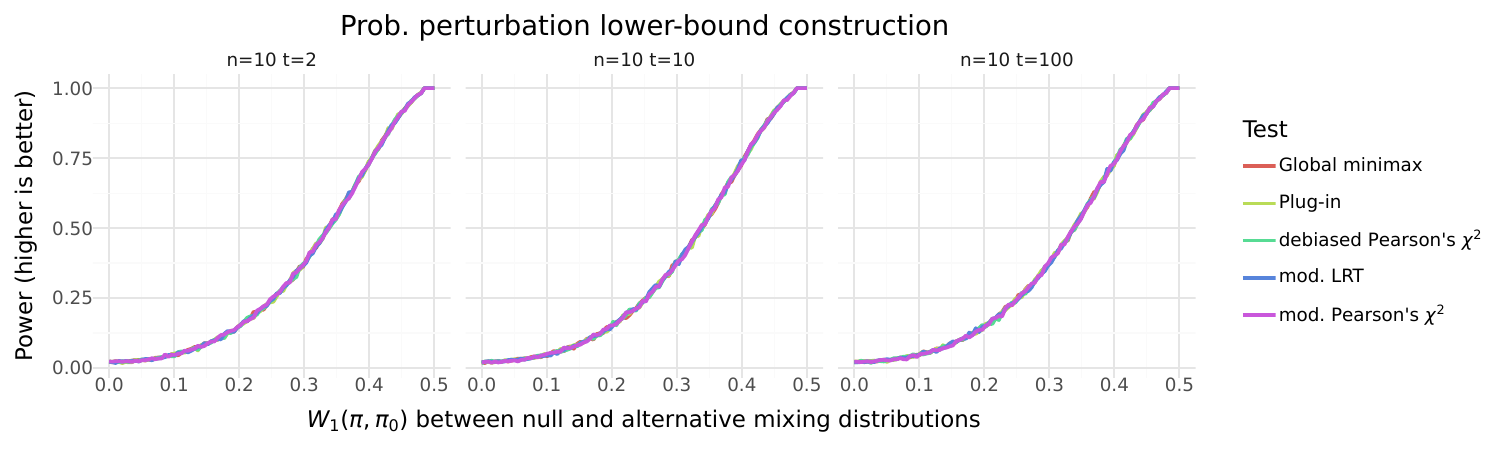}
\caption{Power curves for valid tests under alternative distribution generated by \eqref{eq:low_prob_perturbation}. All tests have the same power curves regardless of the number of trials $t$.}
\label{fig:power_indentity_testing_prob_perturbation}
\end{figure}

Note that additional tests appear in our simulation; their definitions are provided in \zcref[S]{sec:gof_test_definitions}. In particular, the modified Pearson's $\chi^2$ and likelihood ratio tests are inspired by the debiased Pearson's $\chi^2$ test \eqref{eq:fingerprint_test} and truncate the denominator whenever it vanishes. Without truncation, these tests produce numerical issues in our simulation because the mixing distributions put a non-negligible amount of mass at the endpoints of $[0,1]$.

\begin{figure}[ht]
\centering
\includegraphics[width=\linewidth]{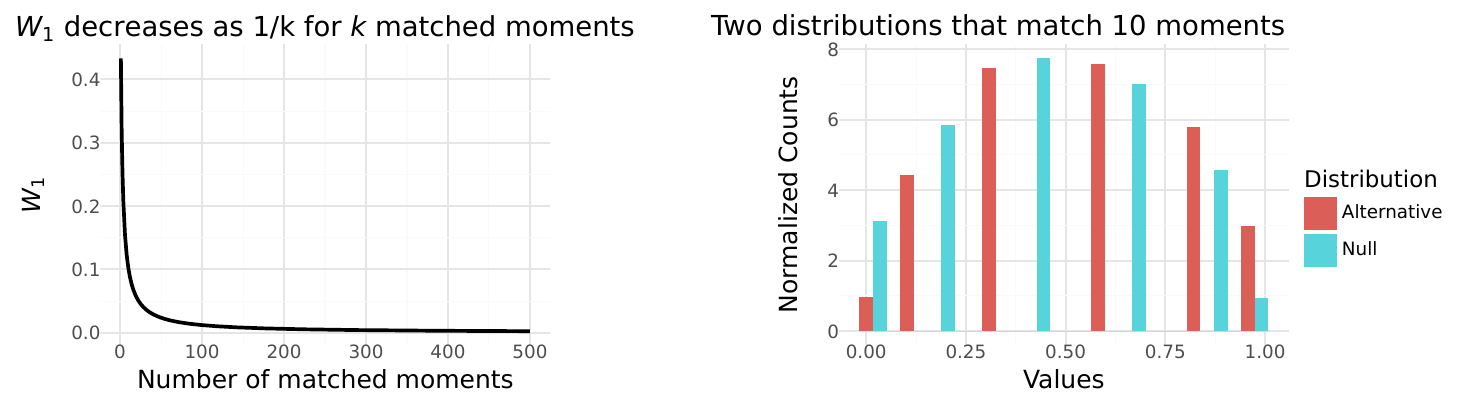}
\caption{The left panel displays the Wasserstein distance between the moment matching distributions as the number of matched moments increased. The right panel displays two distributions that match 10 moments. }
\label{fig:moment_matching}
\end{figure}

We use moment-matching distributions to match the separation rates of the debiased Pearson's $\chi^2$ test in \zcref{lemma:fingerprint_test}. A simple construction of such distributions is provided by \citep{kong2017}. \zcref[S]{fig:moment_matching} shows that as we match more moments, the $W_1$ between them decays like $1/k$ where $k$ is the number of matched moments.

\begin{proposition}[Proposition 2 of \citet{kong2017}]\label{lemma:moment_matching_distributions} Given $k \in \N$, there exists $\pi_0$ and $\pi_1$ mixing distributions supported on the $[0,1]$ interval that match their first $k$ moments and  $W_1(\pi_1,\pi_0)\geq c/k$ where $c$ is a positive constant. 
\end{proposition}

\zcref[S]{fig:power_indentity_testing_moment_matching} shows the power curves of the tests under the moment matching mixing distributions. The minimax theory predicts that when the number of trials $t$ is sufficiently smaller than the number of observations $n$, the debiased Pearson's $\chi^2$ test should perform better in the worst case. Conversely, when $t$ is much larger, the plug-in test should be able to obtain higher power. Both phenomena can be observed in the figure. \begin{figure}[ht]
\centering
\includegraphics[width=\linewidth]{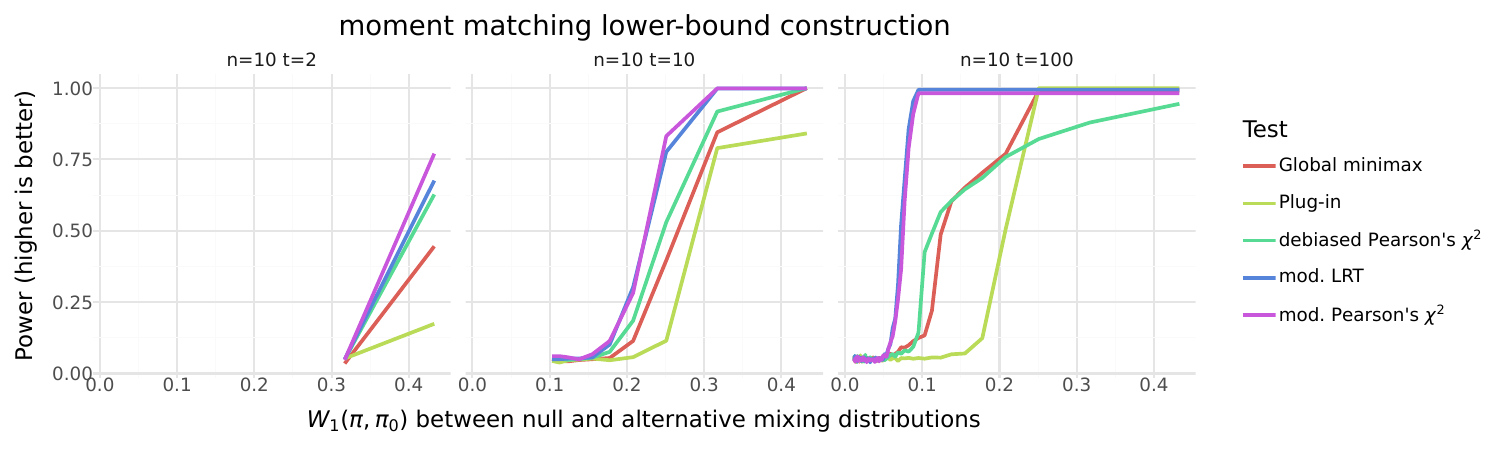}
\caption{Power curves for valid tests under alternatives generated by \zcref{lemma:moment_matching_distributions}.}
\label{fig:power_indentity_testing_moment_matching}
\end{figure}

Finally, \zcref{fig:w1_indentity} illustrates the empirical critical separation, which represents the minimum distance between the null and alternative mixing distributions required for the tests to have power greater than $1-\alpha$ and type I error below $\alpha$. Two regimes arise from the minimax critical separation \eqref{eq:global_minimax_rates}. When $t$ is small relative to $n$, the empirical critical separation is controlled by the moment-matching construction. Initially, increasing $t$ reduces the minimum detectable separation between the null and alternative mixing distributions. However, for sufficiently large $t$, the probability perturbation construction dominates, causing the empirical critical separation to plateau. \begin{figure}[!ht]
\centering
\includegraphics[width=0.9\linewidth]{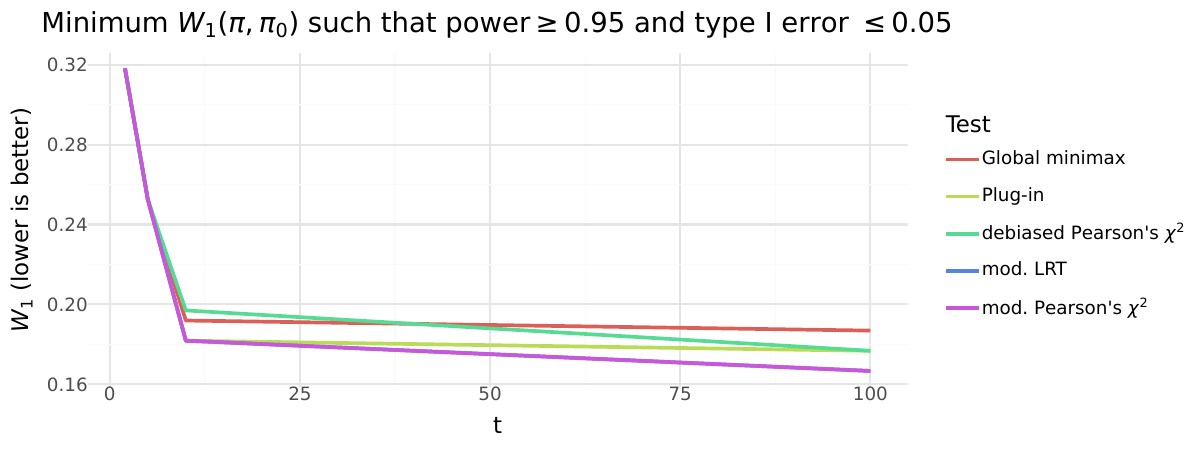}
\caption{Empirical critical separation as the number of trials $t$ is varied below and above the number of observations $n$}
\label{fig:w1_indentity}
\end{figure}

The figure also highlights a key issue of global minimax analysis. Although \zcref{fig:power_indentity_testing_moment_matching} shows that the modified Pearson's $\chi^2$ and likelihood ratio tests perform much better than the other tests, the global minimax perspective, which takes into account all possible mixing distributions, suggests a different conclusion. From this viewpoint, their performance is comparable to the debiased Pearson's $\chi^2$ test.

\section{Homogeneity testing with a reference effect}\label{sec:homogeneity_testing}

Homogeneity testing is a key case of goodness-of-fit testing with applications in statistical meta-studies, where assessing agreement among studies is necessary before pooling their information. It corresponds to testing if the random effects \eqref{eq:binomial_mixture} come from a single point mass: \begin{equation}\label{eq:homogeneity_testing}
H_0: \pi = \delta_{p_0} \quad \text{v.s.} \quad  H_1:  W_1(\pi,\delta_{p_0}) \geq \epsilon
\end{equation} where $p_0 \in [0,1]$ is known. The results of the previous section are loose when applied to homogeneity testing because the null distribution is not arbitrarily complex.

We introduce the local minimax framework to understand the critical separation's dependence on the null hypothesis. The local critical separation is the smallest detectable $\epsilon$-separation with respect to $\pi_0$: \begin{equation}\label{eq:local_critical_separation}
\epsilon_*(n,t,\pi_0) = \inf \left\{\epsilon : R_*(\epsilon,\pi_0) \leq \beta \right\}\quad \where \pi_0=\delta_{p_0} \period
\end{equation}

In \zcref[S]{sec:equivalence_random_and_fixed}, we reduce testing the homogeneity of random effects to testing the homogeneity of fixed effects. This equivalence reveals how the local critical separation depends on the null hypothesis's location. We defer extensive simulations of the proposed tests to \zcref[S]{sec:homogeneity_testing_simulations}.

\subsection{Reduction from random to fixed effects}\label{sec:equivalence_random_and_fixed}

Homogeneity testing of random effects closely relates to homogeneity testing of fixed effects. Consider the unobserved random effects $p_i$ of each observation in \eqref{eq:binomial_mixture}. Under the null hypothesis, all random effects are equal to $p_0$. Under the alternative, the distance between the unobserved empirical distribution of random effects $\tilde{\pi}=n^{-1}\sum_{i=1}^n\delta_{p_i}$ and the null distribution is a good proxy for $W_1(\pi,\pi_0)$, provided $\epsilon$ is not too small. \begin{restatable}{lemma}{HomogeneityReduction}\label{lemma:reduction}
Let $\delta \in (0,1)$, there exists a positive constant $C$ that depends on $\delta$ such that whenever $W_1(\pi,\delta_{p_0})\geq \epsilon \geq C/n$, it holds that
$W_1(\tilde{\pi},\delta_{p_0}) \geq \epsilon/2$ with probability at least $1-\delta$.\end{restatable} The proof can be found in \zcref[S]{sec:reduction_random_to_fixed}. Whenever $\epsilon \gtrsim n^{-1}$, the homogeneity testing problem \eqref{eq:homogeneity_testing} reduces to the fixed-effects problem studied by \citet{chhor2021}: given $Y_i=X_i|p_i \sim \Bin(t,p_i)$ for $1 \leq i \leq n$, test \begin{equation}
H_0: p_i = p_0 \for 1\leq i \leq n \quad \text{v.s.} \quad  H_1:  W_1(\tilde{\pi},\delta_{p_0}) \geq \epsilon/2. \label{eq:fixed_effects_homogeneity_testing}
\end{equation} This reduction is optimal because if $\epsilon \lesssim n^{-1}$, we can construct mixing distributions that, with constant probability, produce the same marginal measure as the null distribution, making it impossible for a valid test to distinguish them.

\begin{restatable}{lemma}{RandomEffectsLowerBound}\label{lemma:RandomEffectsLowerBound}
For hypotheses \eqref{eq:homogeneity_testing}, there exists a universal positive constant $C$ such that the local critical separation \eqref{eq:local_critical_separation} is lower-bounded by $C/n$.
\end{restatable}

The proof is in \zcref[S]{sec:lower_bound_random_effects}. In summary, if $\epsilon \gtrsim n^{-1}$, any valid test for fixed-effects  \eqref{eq:fixed_effects_homogeneity_testing} can be used, while if $\epsilon \lesssim n^{-1}$, no valid test can be powerful. For the fixed-effects problem \eqref{eq:fixed_effects_homogeneity_testing}, the analysis can be split into two cases. When $p_0$ is near the boundaries of $[0,1]$, under the alternative hypothesis, the mixing distribution must allocate some mass away from the boundaries, which can be detected by comparing the expected and observed means. In other words, if  $p_0 \leq \epsilon/8$ and $\epsilon/2 \leq W_1(\tilde{\pi},\pi_0)$, it follows that $\epsilon/4 \leq n^{-1}\sum_{i=1}^n p_i -p_0$. A test of mean deviation can be constructed using the first Kravchuk polynomial: \begin{align}\label{eq:1stmoment_test}
\psi_1^\alpha(X) = I\left(\ T_{1,p_0}(X) > q_\alpha(P_{\pi_0},T_{1,p_0}) \right) \where T_{1,p_0}(X) = \frac{1}{n}\sum_i^n \tilde{K}_1(X_i,p_0,t)\period
\end{align} Conversely, when $p_0$ is away from the boundaries, a mixing distribution under the alternative can match the null distribution's mean. To detect such deviations, we check if the observed variance is compatible with homogeneity. That is, if $\epsilon/2 \leq W_1(\tilde{\pi},\pi_0)$, it always holds that $\epsilon^2/4 \leq \frac{1}{n}\sum_{i=1}^n (p_i-p_0)^2$. The second Kravchuk polynomial estimates the above $\ell_2$ distance without bias, allowing us to construct a test for it: \begin{align}
\psi_2^\alpha(X) = I\left(\ T_{2,p_0}(X) > q_\alpha(P_{\pi_0},T_{2,p_0})\ \right) \where T_{2,p_0}(X) = \frac{1}{n}\sum_{i=1}^n \tilde{K}_2(X_i,p_0,t) \period \label{eq:2ndmoment_test}
\end{align} Combining both the mean test \eqref{eq:1stmoment_test} with the debiased $\ell_2$ test \eqref{eq:2ndmoment_test} via a Bonferroni correction \begin{equation}
\psi_{LM}^\alpha(X) = \max\left( \psi_1^{\alpha/2}(X) \comma \psi_2^{\alpha/2}(X) \right)\label{eq:max_test}
\end{equation} leads to a test that is local minimax optimal for both fixed and random effects.

\begin{restatable}[Homogeneity testing of random effects]{theorem}{RandomEffectsHomogeneityTesting}\label{lemma:RandomEffectsHomogeneityTesting}
For hypotheses \eqref{eq:homogeneity_testing}, the test \eqref{eq:max_test} is optimal. Let Furthermore, for $t \gtrsim \sqrt{n}$, the local critical separation \eqref{eq:local_critical_separation} satisfies 
$$\epsilon_*(n,t,\pi_0) \asymp 
\max\left(\frac{1}{n}\ ,\  
\frac{\mu_{p_0}^{1/2}}{t^{1/2}n^{1/4}}\right) \where \mu_{p_0}=p_0\cdot (1-p_0).$$ Conversely, if $t \lesssim \sqrt{n}$, the local critical separation satisfies: \begin{equation}\label{eq:small_t_rates}
\epsilon_*(n,t,\pi_0) \asymp \max\left(
\frac{1}{n}\ ,\ 
\mu_{p_0}\ ,\
\frac{\mu_{p_0}^{1/2}}{t^{1/2}n^{1/4}}\right).
\end{equation}
\end{restatable}

The proof is deferred to \zcref{appx:homonegenity_testing_known_null}. Intuitively, one expects that testing the mean \eqref{eq:1stmoment_test} is easier than testing the variance \eqref{eq:2ndmoment_test}. This is reflected by \eqref{eq:small_t_rates}, where the fast testing rates near the boundaries, i.e. $\mu_{p_0}\lesssim t^{-1}n^{-1/2}$, are achieved by testing for deviations in the mean \eqref{eq:1stmoment_test}, and the slow testing rates away from the boundaries are achieved by testing for deviations in the variance \eqref{eq:2ndmoment_test}. 

\paragraph{Lower bound on the critical separation} The lower-bound arguments in \zcref{lemma:RandomEffectsHomogeneityTesting} follow similar constructions to those in \zcref[S]{sec:lower_bounds_gof}. Using  \zcref{lemma:lecam_lb}, the inequality $V(P,Q)\leq\sqrt{\chi^2(P,Q)}$ and the tensorization property of the chi-squared distance, we have that the local critical separation is lower-bounded by the $W_1$ distance among mixing distributions whose marginal measures are close in chi-squared distance. \begin{corollary}\label{cor:critical_sep_lb_by_chi2} Let $C_\alpha=1-(\beta+\alpha)$ and $\pi_0=\delta_{p_0}$. For any $\pi_1\in D$ such that $\chi^2(P_{\pi_0}^n,P_{\pi_1}^n) < C_\alpha^2$, it follows that the critical separation is lower-bounded $\epsilon_*(n,t,\pi_0) \geq W_1(\pi_1,\pi_0).$ Alternatively, the same claim holds if $\chi^2(P_{\pi_0},P_{\pi_1}) < \log\left(1+C_\alpha^2\right) / n$.
\end{corollary} Analogously to \zcref{lemma:tvbound}, we can control the chi-squared distance by looking at the moment the difference between the null and the alternative mixing distributions.

\begin{restatable}{lemma}{ChiTwoBound}\label{lemma:chi2bound} Let $p_0 \in (0,1)$, $\pi_0=\delta_{p_0}$ and $\pi_1\in D$. It holds that $$
\chi^2\left(P_{\pi_1},P_{\pi_0}\right) =  M_{p_0}(\pi_1,\pi_0).$$\end{restatable} 

The proof can be found in \zcref[S]{sec:chi2_proof}. By \zcref{lemma:chi2bound}, it is clear that the maximum separation achieved in \zcref{cor:critical_sep_lb_by_chi2} depends on the number of moments matched by $\pi_1$ and $\pi_0$. Since $\pi_0$ is a point mass, an alternative mixing distribution $\pi_1$ can only differ in the first moment or match it. If $\pi_1$ differs on the first moment, then the mean test \eqref{eq:1stmoment_test} detects them, while if $\pi_1$ matches the first moment of $\pi_0$, then the debiased Pearson's chi-squared test \eqref{eq:2ndmoment_test} detects them, which explains the optimality of the combined test \eqref{eq:max_test}.

\section{Homogeneity testing without a reference effect}\label{sec:homogeneity_testing_unkown}

In \zcref[S]{sec:homogeneity_testing}, we assume that the distribution under the null hypothesis is known. However, in statistical meta-analyses, researchers must assess the homogeneity of a treatment effect without any reference distribution. Let $S$ denote the set of all point masses supported on $[0,1]$: $S = \left\{\delta_p : p \in [0,1]\right\}$, and define the distance to the set as the shortest distance to any of its members: $
W_1(\pi,S) = \sup_{s \in S}W_1(\pi,s) \period$ We want to differentiate between the mixing distribution being a point mass and being far away from the set of point masses: \begin{align}\label{eq:homogeneity_testing_unknown}
H_0: \pi \in S \vs H_1:  W_1(\pi,S)\geq \epsilon\period
\end{align} To determine the smallest $\epsilon$ for which a powerful valid test exists, we use the global minimax framework restricted to the set $S$. Consider all tests that guarantee type I error control under the null hypothesis: $$
\Psi(S) = \bigcap_{\pi_0 \in S}\ \Psi(\pi_0) = \left\{\psi : \sup_{\pi_0 \in S} P_{\pi_0}(\psi(X) = 1) \leq \alpha \right\},$$ define the risk as their maximum type II error under the alternative hypothesis $$
R_*(\epsilon,S)= \inf_{\psi \in \Psi(S)} \sup_{W_1(\pi,S)\geq \epsilon}P_{\pi}(\psi(X) = 0),$$ and let the critical separation be the smallest $\epsilon$ such that there exists a test that controls both errors \begin{equation}\label{eq:homogeneity_critical_separation}
\epsilon_* = \inf \left\{\epsilon: R_*(\epsilon,S)\leq \beta\right\} \period
\end{equation} Intuitively, distinguishing a mixing distribution from the set of point masses is difficult if it is highly concentrated around its mean. Thus, the $\epsilon$ separation between the hypotheses relates to the smallest variance that we can detect. Under the null hypothesis, the variance of the mixing distribution is zero, $V(\pi)=0$, whereas under the alternative hypothesis,
it must be greater than $\epsilon^2$ since $$
V(\pi)=W_2^2(\pi,\delta_{m_1(\pi)})\geq W_1^2(\pi,\delta_{m_1(\pi)}) \geq W_1^2(\pi,S) \geq \epsilon^2 \period$$ Hence, \eqref{eq:homogeneity_testing_unknown} reduces to testing the variance of the mixing distribution \begin{equation}\label{eq:variance_testing}
H_0: V(\pi)=0 \vs H_1:V(\pi)\geq \epsilon^2\period
\end{equation}

In \zcref[S]{sec:conservative_test}, we develop a test by debiasing the plug-in estimator of $V(\pi)$. Although minimax optimal, this test is overly conservative, leaving room to improve type I error control. \zcref[S]{sec:not_conservative_test} introduces a debiased Cochran's chi-squared test that achieves better type I control while preserving optimality. We defer extensive simulations of the proposed tests to \zcref[S]{sec:homogeneity_testing_unkown_simulations}.

\subsection{A conservative test}\label{sec:conservative_test}

The natural starting point is to study an estimator of the variance of the mixing distribution. The following U-statistic \begin{equation}\label{eq:ustat}
\hat{V}(X) = \binom{n}{2}^{-1}\sum_{i<j}\frac{h(X_i,X_j)}{2} \where h(X,Y)=\frac{\binom{X}{2}}{\binom{t}{2}}+\frac{\binom{Y}{2}}{\binom{t}{2}}-2\frac{\binom{X}{1}}{\binom{t}{1}}\frac{\binom{Y}{1}}{\binom{t}{1}}\comma
\end{equation} is an unbiased estimator of the mixing distribution's variance. The corresponding test is defined using its maximum quantile over the null hypothesis to guarantee validity \begin{equation}\label{eq:oracle_test}
\psi_{\hat{V}}(X) = I\left(\hat{V}(X) > \sup_{\pi \in S}q_{\alpha}(P_{\pi},\hat{V}(X))\right)\period
\end{equation} The following lemma, proved in \zcref[S]{sec:proof_oracle_test}, establishes that its performance is comparable to the worst-case separation rates achieved in \zcref{lemma:RandomEffectsHomogeneityTesting}. This result is expected, as the type I error must be controlled for all distributions in $S$.

\begin{restatable}{lemma}{OracleTestRates}
\label{lemma:rates_oracle_test} For hypotheses \eqref{eq:variance_testing}, the test  \eqref{eq:oracle_test} is optimal, and the critical separation \eqref{eq:homogeneity_critical_separation} satisfies $$\epsilon_* \asymp  \max\left(n^{-1/2},t^{-1/2}n^{-1/4}\right).$$
\end{restatable}

Although the test \eqref{eq:oracle_test} is minimax optimal, it can be conservative when the mixing distribution is a point mass near the boundary of $[0,1]$. Consider the mixing distribution $\pi_0=\delta_{p_0} \in S$, and let $\mu_{p_0}=p_0\cdot (1-p_0)$. The ratio between the variance under $\pi_0$ and the maximum variance under the null depends on $\mu_{p_0}^2$, i.e. $V_{P_{\pi_0}}\left[\hat{V}(X)\right]/\sup_{\pi_0 \in S}V_{P_{\pi_0}}\left[\hat{V}(X)\right] = \mu^2_{p_0}$. Since the mean of $\hat{V}$ is always zero under the null hypothesis, the previous ratio implies that the $1-\alpha$ quantile of $\hat{V}$ varies substantially across $S$, and the decision threshold in \eqref{eq:oracle_test} is overly conservative.

\subsection{The debiased Cochran's chi-squared test}\label{sec:not_conservative_test}

A practical solution is to normalize $\hat{V}$ using an estimator of $\mu_{p_0}$. To achieve this, assume access to two independent samples of equal size $
X_i,Y_i \sim P_\pi$ for $1\leq i \leq n$, and define the following test, called the debiased Cochran's $\chi^2$ test, \begin{equation}\label{eq:truncated_test}
\psi_{R}(X) = I\left(R(X,Y) > \sup_{\pi \in S}q_{\alpha}(P_{\pi},R)\right) \where  R(X,Y) =\frac{\hat{V}(X)}{\max(\hat{\mu}(Y),\gamma)},
\end{equation} where $\hat{\mu}(Y)$ is an unbiased estimator of $\mu_{m_1(\pi)}=m_1(\pi)\cdot (1-m_1(\pi))$ and $\gamma$ is a truncation parameter: $$
\hat{\mu}(Y)= \binom{\tilde{n}}{2}^{-1}\sum_{i<j}\frac{\tilde{h}(Y_i,Y_j)}{2} \textand \tilde{h}(X,Y)= \frac{X}{t}+\frac{Y}{t}-2\frac{XY}{t^2}.$$ That is, $R(X,Y)$ avoids estimating $\mu_{m_1(\pi)}$ whenever it is small to ensure its variance remains controlled. 

Examining the variance of the statistic under a point in $S$ versus the maximum variance under $S$, we observe that they match up to constants insofar as the underlying distribution is not too close to the origin. That is, for $\pi_0=\delta_{p_0}\in S$, it follows that $V_{P_{\pi_0}}\left[R(X,Y)\right]/\sup_{\pi_0 \in S}V_{P_{\pi_0}}\left[R(X,Y)\right]\asymp  \left(\mu_{p_0}/\max\left(\mu_{p_0},\gamma\right) \right)^2$. Thus, the debiased Cochran's $\chi^2$ test is conservative only for distributions in $S$ close to the boundaries of $[0,1]$. 

The following lemma states that the debiased Cochran's $\chi^2$ test performs the same as \eqref{eq:oracle_test} up to constants. The proof appears in \zcref[S]{sec:proof_truncated_test}. The sample splitting technique simplifies the proof, but simulations indicate it is unnecessary. 

\begin{restatable}{theorem}{TruncatedTestOtimality}\label{lemma:rates_truncated_test}
For hypotheses \eqref{eq:homogeneity_testing_unknown}, the test \eqref{eq:truncated_test} with threshold $\gamma \gtrsim n^{-1}$ is optimal.
\end{restatable}


\begin{figure}[!hb]
\centering
\includegraphics[width=\linewidth]{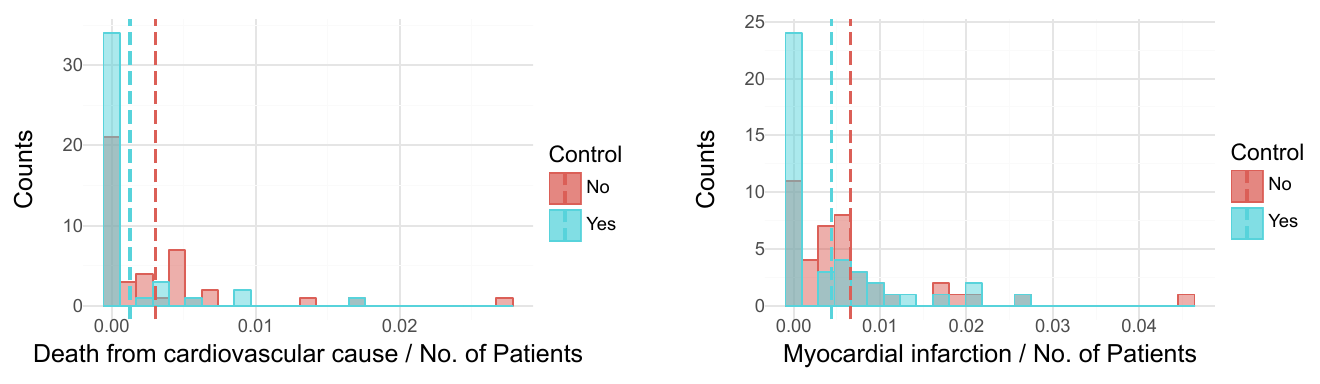}
\caption{Histograms of proportions of the proportion of patients that suffered a myocardial infarction and died from cardiovascular causes due to the application of the rosiglitazone treatment for the treatment and control groups.}
\label{fig:nissen2007_data}
\end{figure}

\section{Application to meta-analyses and model selection}\label{sec:applications}

In \zcref[S]{sec:metaanalysis}, we test the homogeneity of rare effects across a small number of studies with large number of participants. Additionally, in \zcref[S]{sec:county}, we use goodness-of-fit testing to assess the quality of different approximations to the underlying mixing distribution when the number of studies is much larger than the number of participants.

\subsection{Meta-analysis of safety concerns associated with rosiglitazone}\label{sec:metaanalysis}

\citet{nissenEffectRosiglitazoneRisk2007} conducted a meta-analysis of 42 studies, each with 40 to 2900 participants, to assess cardiovascular risks associated with the rosiglitazone treatment. \zcref[S]{fig:nissen2007_data} displays the proportion of patients who experienced a myocardial infarction (MI) or died from cardiovascular causes (D) in both the treatment and control groups. Two challenges are noteworthy: the heterogeneity in the study sizes and the low event rate. The maximum death rate in the treated group is 2\%, while over 80\% of control group studies report zero events. These low event counts suggest that standard asymptotic procedures may be unreliable, as they assume event probabilities are bounded away from zero.

\begin{table}[!th]
\centering
\begin{tabularx}{\textwidth}{lllll}
\toprule
Test & MI (Treatment) & MI (Control) & D (Treatment) & D (Control) \\
\midrule
Local Minimax \eqref{eq:max_test} & [0.005, 0.009] & [0.003, 0.007] & [0.003, 0.005] & [0.001, 0.003] \\
Mean \eqref{eq:1stmoment_test} & [0.006, 0.01] & [0.003, 0.01] & Rejected all & [0.001, 0.006] \\
Debiased $\ell_2$ \eqref{eq:2ndmoment_test} & [0.006, 0.01] & [0.003, 0.01] & [0.003, 0.006] & [0.001, 0.005] \\
$\ell_2$ & Rejected all & [0.003, 0.017] & [0.002, 0.013] & [0.001, 0.013] \\
Plug-in \eqref{eq:w1_plugin_test} & Rejected all & [0.004, 0.009] & [0.002, 0.007] & [0.001, 0.007] \\
Pearson's $\chi^2$ & Rejected all & Rejected all & [0.003, 0.006] & [0.001, 0.007] \\
Likelihood ratio & Rejected all & Rejected all & Rejected all & [0.001, 0.004] \\
\bottomrule
\end{tabularx}
\caption{$95\%$ Confidence intervals produced by inverting homogeneity tests for simple nulls.}
\label{tab:nissen_cis}
\end{table}

Following \citet{parkTestingHomogeneityProportions2019}, we test the homogeneity of myocardial infarction and death proportions across studies in both treatment and control groups.  Two approaches are considered: (I) constructing confidence intervals by inverting homogeneity tests for reference effects, and (II) obtaining P-values from homogeneity tests without reference effects.

First, we construct a $95\%$ confidence interval by inverting a test under a simple null hypothesis.  The tests in \zcref[S]{sec:homogeneity_testing} assume equal study sizes $(t)$, which does not hold here. Adjusting the methods for varying study sizes is straightforward and detailed in \zcref[S]{appx:homogeneity_testing_simple_null}. \zcref[S]{tab:nissen_cis} presents the resulting confidence intervals for each one of the methods, including additional standard tests defined in \zcref{appx:homogeneity_testing_simple_null}. Some classical tests reject all null hypotheses, indicating evidence of heterogeneity. However, the local minimax test \eqref{eq:max_test} provides non-empty confidence intervals, suggesting possible homogeneity. The confidence intervals also show improvements in tightness: the debiased $\ell_2$ test \eqref{eq:2ndmoment_test} outperforms the classical $\ell_2$ test, and the local minimax test \eqref{eq:max_test} further improves upon the debiased $\ell_2$ test by incorporating a mean test \eqref{eq:1stmoment_test}, which is useful for detecting small effects.

Second, we compute the P-values based on homogeneity tests without a reference effect (\zcref[S]{sec:homogeneity_testing_unkown}). We employ the asymptotically valid Cochran's $\chi^2$ test, a version of it valid for finite samples, and two versions of the debiased Cochran's $\chi^2$ introduced in \zcref{sec:not_conservative_test}. We refer the reader to \zcref[S]{sec:homogeneity_testing_unkown_simulations} for an analysis of their performance under simulations, and to \zcref[S]{appx:homogeneity_testing_composite_null} for their definition under varying numbers of patients per study. \zcref[S]{tab:nissen_cis} shows the obtained P-values. Cochran's $\chi^2$ test rejects the null hypothesis of homogeneity at the $\alpha=0.05$ level for all covariates except the death proportion in the control group. In contrast, its debiased versions either fail to reject or yield P-values near the decision threshold $\alpha=0.05$. These results align with the confidence interval from the local minimax test in \zcref{tab:nissen_cis}.

\begin{table}[!ht]
\centering
\begin{tabularx}{\textwidth}{lrrrr}
\toprule
Test & MI (Treat.) & MI (Control) & D (Treat.) & D (Control) \\
\midrule
Asymptotic Cochran's $\chi^2$ & 0 & 0 & 0.002 & 0.479 \\
Cochran's $\chi^2$ & 0 & 0.003 & 0.005 & 0.492 \\
Debiased Cochran's $\chi^2$ (V.1) & 0.062 & 0.285 & 0.049 & 0.286 \\
Debiased Cochran's $\chi^2$ (V.2) & 0.062 & 0.285 & 0.048 & 0.285 \\
\bottomrule
\end{tabularx}
\caption{P-values for homogeneity testing without a reference effect. }
\label{tab:nissen_pvalues}
\end{table}

\section{Discussion}

In this paper, we study goodness-of-fit testing for random binomial effects within the minimax framework. We demonstrate that combining a test based on the plug-in and debiased Pearson's $\chi^2$ tests \citep{balakrishnanHypothesisTestingDensities2019} is minimax optimal. The critical separation rates match previously established estimation rates \citep{tian2017,vinayakMaximumLikelihoodEstimation2019}, indicating that testing is as hard as estimation, a common feature of deconvolution problems. The use of Kravchuk polynomials \citep{krawtchoukGeneralisationPolynomesHermite1929} simplifies lower-bound arguments and may have broader applications in empirical Bayes methods, where bounding the distance between marginal measures via distances between mixing distributions is essential. Furthermore, the analysis of the plug-in test reveals its adaptivity to the null distribution's variance. Given the practicality of the test, it is of interest to understand when the local rates are sharp.

For homogeneity testing with respect to a reference effect, we connect the problem to its fixed-effects counterpart \citep{chhor2021} and establish how the critical separation depends on the null distribution's location. Inverting the local minimax test yields confidence intervals that exhibit adaptivity due to faster rejection of null hypotheses near the boundary. Formalizing these adaptivity properties could offer a valuable alternative to adaptive estimation-based confidence intervals.

For homogeneity testing without a reference effect, we propose a debiased version of Cochran's chi-squared test, which is minimax optimal. The debiased Cochran's chi-squared statistic has a smaller bias and variance than its classical counterpart. However, our tests do not adapt to the underlying variance of the data because they use the maximum quantile over the null hypothesis space. An alternative is to investigate \citet{bergerValuesMaximizedConfidence1994}'s method: sample split the data, use one half to construct a confidence interval for the homogeneous effect, and the other half to check whether the local minimax test \eqref{eq:max_test} rejects all members of the interval. We expect this approach to be less conservative since it avoids using the worst quantile under the null hypothesis space.

\textbf{Acknowledgments} We thank Edward H.~Kennedy, Tudor Manole, Ian Waudby-Smith, Kenta Takatsu, JungHo Lee, and Siddhaarth Sarkar for useful comments during the development of this work. The authors gratefully acknowledge funding from the National Science Foundation (DMS-2310632).

\pagebreak

\appendix

\begin{center}

{\large\bf SUPPLEMENTARY MATERIAL}

\end{center}

\addcontentsline{toc}{section}{Supplementary material}
\etocdepthtag.toc{mtappendix}
\etocsettagdepth{mtchapter}{none}
\etocsettagdepth{mtappendix}{subsection}
\etocsettagdepth{mtreferences}{section}
\renewcommand{\contentsname}{Supplementary material}
{
\renewcommand{\baselinestretch}{0.9}\normalsize
\parskip=0em
\renewcommand{\contentsname}{\normalsize Table of contents}
\tableofcontents
}


\section{Kravchuk polynomials}\label{sec:kravchuk}

Let $m\in\{0,\dots,t\}$ and $p\in(0,1)$, the $m$-th Kravchuk polynomial \citep{krawtchoukGeneralisationPolynomesHermite1929,szegoOrthogonalPolynomials1975,dominiciAsymptoticAnalysisKrawtchouk2005} is given by
\begin{equation}\label{eq:m_kravchuk}
K_m(x,p,t) = \sum_{v=0}^m (-1)^{m-v} \binom{t-x}{m-v}\binom{x}{v}p^{m-v}(1-p)^v\period
\end{equation} We also define the $m$-th normalized Kravchuk polynomials $$\tilde{K}_m(x,p,t) = \binom{t}{m}^{-1}K_m(x,p,t).$$ They are orthogonal under the binomial distribution \citep{szegoOrthogonalPolynomials1975}: \begin{equation} 
E_{X\sim \text{Bin}(t,p)}[\tilde{K}_n(X,p,t)\cdot \tilde{K}_m(X,p,t)] = \binom{t}{n}^{-1}\mu_p^n \cdot I(n=m)  \where \mu_p = p\cdot(1-p)\comma \label{eq:kravhuk_orthogonality}
\end{equation} and they provide unbiased estimates of centered moments $$
E_{X \sim \text{Bin}(t,q)} \left[\tilde{K}_m(X,p,t)\right] = (-1)^m \cdot (p-q)^m \for m\geq 1 \period$$

The generating series of the Kravchuk polynomials implies a simple local expansion of the Bernstein basis. The generating series is, see equation 2.82.4 of \citet{szegoOrthogonalPolynomials1975}, \begin{equation}\label{eq:kravchuk_generating_series}
\sum_{j=0}^t K_j(x,p,t) \cdot w^j = (1 + (1-p) \cdot w)^x \cdot (1-p \cdot w)^{t-x} \quad \for x \in \{0,\dots,t\}.
\end{equation} Hence, for $p\in (0,1)$, it holds that \begin{equation}\label{eq:local_expansion_no_center}
\frac{B_{j,t}(u+p)}{B_{j,t}(p)} = \left(\frac{u}{p} + 1\right)^j \cdot \left(1 - \frac{u}{1-p} \right)^{t-j} = \sum_{m=0}^t K_m(j,p,t) \cdot \frac{u^m}{p^m(1-p)^m} \period
\end{equation} A simple change of variables leads to the following expansion \begin{equation}\label{eq:local_expansion}
B_{j,t}(u) = B_{j,t}(p) \cdot \sum_{m=0}^t K_m(j,p,t) \cdot \frac{(u-p)^m}{\mu_p^m} \quad \for 0<p<1,
\end{equation} which is the Taylor polynomial expansion since \begin{equation}
B_{j,t}^{(l)}(p) = B_{j,t}(p) \cdot \frac{K_l(j,p,t)}{\mu_p^l} \cdot (l)! \quad \for 0<p<1 \textand l \in \{0,\dots,t\} \period
\end{equation}

\subsection{Bounding the total variation distance by moment differences}\label{sec:tvbound_proof}

\begin{lemma}[Extension of \zcref{lemma:tvbound}]\label{lemma:tvbound_ext} Let $\pi_0$ and $\pi_1$ be two mixing distributions supported on the $[0,1]$ interval. For $p\in (0,1)$, it follows that
\begin{align}\label{eq:tvbound}
\sqrt{p^t \wedge (1-p)^t}\cdot \frac{M_p(\pi_1,\pi_0)}{2} \leq V\left(P_{\pi_1},P_{\pi_0}\right) \leq \frac{M_p(\pi_1,\pi_0)}{2}\period
\end{align}
\end{lemma}
\begin{proof}[Proof of \zcref{lemma:tvbound_ext}]
We prove the lemma using \eqref{eq:local_expansion_no_center} rather than \eqref{eq:local_expansion}. A simple change of variables implies the lemma's statement.
\begin{align}
&V\left(E_{u\sim\pi_1}[\text{Bin}(t,u+p)],E_{u\sim\pi_0}[\text{Bin}(t,u+p)]\right) \\
&= \frac{1}{2}\cdot \sum_{j=0}^t\ \left|E_{u\sim\pi_1}[B_{j,t}(u+p)]-E_{u\sim\pi_0}[B_{j,t}(u+p)]\right|\\
&= \frac{1}{2}\cdot \sum_{j=0}^t B_{j,t}(p) \cdot \left|\sum_{m=1}^t K_m(j,p,t) \cdot \frac{\Delta_m}{p^m(1-p)^m}\right|\label{eq:last_step}\end{align} where in the last equality we used \eqref{eq:local_expansion_no_center} and $\Delta_m = E_{\pi_1}[u^m]-E_{\pi_0}[u^m]$. It follows that
\begin{align}
&V\left(E_{u\sim\pi_1}[\text{Bin}(t,u+p)],E_{u\sim\pi_0}[\text{Bin}(t,u+p)]\right)\\
&\leq  \frac{1}{2}\cdot \left[\sum_{j=0}^t B_{j,t}(p) \cdot \left|\sum_{m=1}^t K_m(j,p,t) \cdot \frac{\Delta_m}{p^m(1-p)^m}\right|^2\right]^{1/2} &&\text{by Jensen's inequality}\\
&=  \frac{1}{2}\cdot \left[\sum_{m=1}^t \binom{t}{m} \cdot \frac{\Delta_m^2}{p^m(1-p)^m}\right]^{1/2} &&\text{by orthogonality }\eqref{eq:kravhuk_orthogonality}
\end{align} The lower-bound proceeds similarly. By \eqref{eq:last_step} and norm monotonicity, it follows that
\begin{align}
V&\left(E_{u\sim\pi_1}[\text{Bin}(t,u+p)],E_{u\sim\pi_0}[\text{Bin}(t,u+p)]\right)\\
&\geq  \frac{1}{2}\cdot \left[\sum_{j=0}^t B_{j,t}^2(p) \cdot \left|\sum_{m=1}^t K_m(j,p,t) \cdot \frac{\Delta_m}{p^m(1-p)^m}\right|^2\right]^{1/2}
\end{align} Since $B_{j,t}(p)\geq p^t \wedge (1-p)^t$, it holds that \begin{align}
&V\left(E_{u\sim\pi_1}[\text{Bin}(t,u+p)],E_{u\sim\pi_0}[\text{Bin}(t,u+p)]\right)\\
&\geq  \frac{\sqrt{p^t \wedge (1-p)^t}}{2}\cdot \left[\sum_{j=0}^t B_{j,t}(p) \cdot \left|\sum_{m=1}^t K_m(j,p,t) \cdot \frac{\Delta_m}{p^m(1-p)^m}\right|^2\right]^{1/2} \\
&=  \frac{\sqrt{p^t \wedge (1-p)^t}}{2}\cdot \left[\sum_{m=1}^t \binom{t}{m} \cdot \frac{\Delta_m^2}{p^m(1-p)^m}\right]^{1/2}
\end{align} where we used the orthogonality of the Kravchuk polynomials in the last step.
\end{proof}

\subsection{Bounding the chi-squared distance by moment differences}\label{sec:chi2_proof}

\begin{lemma}\label{lemma:abstract_chi2_bound} Let $p \in (0,1)$, if \begin{equation}\label{eq:lb_condition}
E_{u\sim\pi_0}B_{j,t}(u) \geq C^{-1} \cdot B_{j,t}(p) \quad \for 1\leq j \leq t
\end{equation} then \begin{equation}\label{eq:lb_result}
\chi^2\left(P_{\pi_0},P_{\pi_1}\right) \leq C \cdot M_p(\pi_1,\pi_0)\period
\end{equation} Furthermore, \eqref{eq:lb_result} strengthens to equality if \eqref{eq:lb_condition} holds with equality.
\end{lemma}\begin{proof}[Proof of \zcref{lemma:abstract_chi2_bound}]
For any $0 <p <1$, it holds that
\begin{align}
\chi^2\left(P_{\pi_0},P_{\pi_1}\right)
&= \sum_{j=0}^t\ \frac{\left|E_{u\sim\pi_1}[B_{j,t}(u)]-E_{u\sim\pi_0}[B_{j,t}(u)]\right|^2}{E_{u\sim\pi_0}B_{j,t}(u)}\\
&= \sum_{j=0}^t \frac{B_{j,t}^2(p) \cdot \left|\sum_{m=1}^t K_m(j,p) \cdot \frac{\Delta_m(\pi_1,\pi_0)}{p^m(1-p)^m}\right|^2}{E_{\pi_0}B_{j,t}(u)} &&\text{by eq. \eqref{eq:local_expansion}}
\end{align} where $
\Delta_m(\pi_1,\pi_0) = E_{u\sim\pi_1}[u-p]^m-E_{u\sim\pi_0}[u-p]^m\period$ Then, by orthogonality of the Kravchuk polynomials \eqref{eq:kravhuk_orthogonality} and the assumption \eqref{eq:lb_condition}, we have that \begin{equation}
\chi^2\left(P_{\pi_0},P_{\pi_1}\right)
\leq C \cdot \sum_{m=1}^t \binom{t}{m} \cdot \frac{\Delta_m^2(\pi_1,\pi_0)}{\mu_p^m}\period
\end{equation}\end{proof}

\ChiTwoBound*
\begin{proof}[Proof of \zcref{lemma:chi2bound}]
For the particular case where $\pi_0$ is a point mass, i.e. $\pi=\delta_{p_0}$, choosing $p=p_0$ implies that $
E_{u\sim\pi_0}B_{j,t}(u) = B_{j,t}(p_0)$ for $1\leq j \leq t\period$ The statement follows from \zcref{lemma:abstract_chi2_bound}.
\end{proof}

We note that \zcref[S]{lemma:chi2bound} can be slightly generalized. \zcref[S]{prop:bernstein_deviation} and \zcref{lemma:general_chi2_bound}, show that one can replace $\pi_0=\delta_{p_0}$ in \zcref{lemma:chi2bound} by any distribution supported on $[p_0-\delta,p_0+\delta]$ where $\delta \lesssim p_0/t$.

\begin{proposition}\label{prop:bernstein_deviation}
For $p\in (0,1/2]$ and $0 \leq j \leq t$, it holds that \begin{equation}
\left|\frac{B_{j,t}(u)}{B_{j,t}(p)}-1\right|
\leq \left(1 + \left(\frac{1-p}{p}\right)^{j/t} \cdot \frac{|u-p|}{1-p}\right)^t - 1 \leq \left(1 + \frac{|u-p|}{p}\right)^t - 1 \period
\end{equation}
\end{proposition} \begin{proof}[Proof of \zcref{prop:bernstein_deviation}]
By the Taylor expansion \eqref{eq:local_expansion}, it holds that \begin{equation}\label{eq:bernstein_deviation}
\left|\frac{B_{j,t}(u)}{B_{j,t}(p)}-1\right| \leq \sum_{m=1}^t |K_m(j,t,p)| \cdot \frac{|u-p|^m}{\mu_p^m}\period
\end{equation} Furthermore, by \eqref{eq:m_kravchuk}, we have that \begin{equation}
|K_m(j,t,p)|\leq \binom{t}{m} \cdot p^m \cdot \sum_{v=0}^m \left(\frac{1-p}{p}\right)^v \cdot w_v
\end{equation} where \begin{equation}
w_v = \frac{\binom{t-j}{m-v}\binom{j}{v}}{\binom{t}{m}} \geq 0 \textand \sum_{v=0}^m w_v = 1
\end{equation} due to the Chu-Vandermonde's identity. Since $(1-p)/p\geq 1$, by Jensen's inequality, we have that \begin{align}
|K_m(j,t,p)|&\leq \binom{t}{m} \cdot p^m \cdot \exp\left\{\log\left(\frac{1-p}{p}\right) \cdot \sum_{v=0}^t v \cdot w_v\right\}\\
&= \binom{t}{m} \cdot p^m \cdot \left(\frac{1-p}{p}\right)^{j \cdot \binom{t-1}{m-1}/\binom{t}{m}}\period
\end{align} Pluging the inequality in \eqref{eq:bernstein_deviation}, we get \begin{align}
\left|\frac{B_{j,t}(u)}{B_{j,t}(p)}-1\right| &\leq \sum_{m=1}^t \binom{t}{m} \cdot \left(\frac{|u-p|}{1-p}\right)^m \cdot \left(\frac{1-p}{p}\right)^{j \cdot \binom{t-1}{m-1}/\binom{t}{m}}\\
&= \left(1 + \left(\frac{1-p}{p}\right)^{j/t} \cdot \frac{|u-p|}{1-p}\right)^t - 1\period
\end{align}
\end{proof}

The following bound is an immediate consequence.

\begin{restatable}{corollary}{GeneralChiTwoBound}\label{lemma:general_chi2_bound} Let $p\in(0,1/2]$ and $\pi_0$ be a mixing distribution such that $|u-p|\leq \delta$ almost surely for $u \sim \pi_0$ where $\delta \leq \log(2-C) \cdot p/t$ for some $C \in (0,1]$. Then, for any mixing distribution $\pi_1$, it holds that $
\chi^2(P_{\pi_0},P_{\pi_1}) \leq C^{-1} \cdot M_p(\pi_1,\pi_0)\period$
\end{restatable}

\begin{proof}[Proof of \zcref{lemma:general_chi2_bound}]

By \zcref{prop:bernstein_deviation}, it holds that \begin{equation}
\frac{B_{j,t}(u)}{B_{j,t}(p)}\geq2 - \left(1 + \frac{|u-p|}{p}\right)^t \geq 2 - \exp\left\{\frac{t \delta}{p}\right\}\geq C
\end{equation} The statement follows by \zcref{lemma:abstract_chi2_bound}.
\end{proof}

\subsection{Expectations and variances of 1st and 2nd Kravchuk polynomials}

In the next two remarks, we adopt the following notation: \begin{equation}
\Delta = q-p \comma \mu_{p} = p(1-p) \textand d_l = \int \Delta^l\ d\pi(p)
\end{equation}

\begin{remark}[First Kravchuk polynomial]\label{sec:1stKravchuk} The expectation and second moment  of the first Kravchuk polynomial under a binomial distribution are given by
\begin{align}
E_{X \sim \text{Bin}(t,q)} \left[\tilde{K}_1(X,p,t)\right] &= \Delta\\
E_{X\sim \text{Bin}(t,q)}[\tilde{K}_1(X,p,t)^2] &= \Delta^2 + \frac{\mu_q}{t}
\end{align} Using them, we compute its variance under the binomial distribution. \begin{align}
V_{X\sim \text{Bin}(t,p)}[\tilde{K}_1(X,p,t)] &= E_{X\sim \text{Bin}(t,p)}[\tilde{K}_1(X,p,t)^2] = \frac{\mu_p}{t}\\
V_{X\sim \text{Bin}(t,q)}[\tilde{K}_1(X,p,t)] &= \frac{\mu_q}{t}= \frac{\mu_p}{t} + (1 - 2  p) \frac{\Delta}{t} - \frac{\Delta^2}{t}\leq\frac{\mu_p}{t} +  \frac{\Delta}{t} - \frac{\Delta^2}{t}
\end{align}

The expectation of the second Kravchuk polynomial under a mixture of binomials is given by
\begin{align}
E_{X\sim P_\pi}[\tilde{K}_1(X,p,t)] &= d_1
\end{align}

We proceed to compute its variance under a binomial mixture.
\begin{align}
V_{q\sim\pi}[E_{X \sim \text{Bin}(t,q)}[\tilde{K}_1(X,p,t)]] &= d_2 - d_1^2 \\
E_{q\sim\pi}[V_{X \sim \text{Bin}(t,q)}[\tilde{K}_1(X,p,t)]] &\leq \frac{\mu_p}{t} + \frac{d_1}{t}-\frac{d_2}{t}
\end{align} Consequently, \begin{align}
V_{X\sim P_\pi}[\tilde{K}_1(X,p,t)] &\leq \frac{\mu_p}{t} + \frac{d_1}{t}+ (1-\frac{1}{t})\cdot d_2 - d_1^2
\end{align}
\end{remark}

\begin{remark}[Second Kravchuk polynomial]\label{sec:2ndKravchuk}

The expectation and second moment  of the second Kravchuk polynomial under a binomial distribution are given by
\begin{align}
E_{X \sim \text{Bin}(t,q)} \left[\tilde{K}_2(X,p,t)\right] &= \Delta^2\\
E_{X\sim \text{Bin}(t,q)}[\tilde{K}_2(X,p,t)^2] &= \Delta^4 + 2 \frac{\mu_q^2}{t-1} + 2 \frac{\mu_q}{t} \cdot \left[2 p^2 - (1 + 4 p) q + 3 q^2\right]
\end{align} Using them, we compute its variance under the binomial distribution. \begin{align}
V_{X\sim \text{Bin}(t,p)}[\tilde{K}_2(X,p,t)] &= E_{X\sim \text{Bin}(t,p)}[\tilde{K}_2(X,p,t)^2] = \frac{\mu_p^2}{(t-1)t}\\
V_{X\sim \text{Bin}(t,q)}[\tilde{K}_2(X,p,t)] &= 2 \frac{\mu_q^2}{t-1} + 4p^2 \frac{q}{t} - 2(1 + 2p(2 + p)) \frac{q^2}{t} + 8(1 + p) \frac{q^3}{t} -
 6 \frac{q^4}{t}\\
&= \frac{4(1-2p)(t-2)}{t(t-1)} \cdot \Delta^3 \\
&+ \frac{2}{t}\left[\frac{1}{t-1}-2\right] \cdot \Delta^4 \\
&+ \frac{4(1-2p)}{(t-1)t}\cdot \mu_p \cdot \Delta \\
&+ \frac{2}{(t-1)t}\cdot \mu_p^2 \\
&+ \frac{2(1+2(t-4)\mu_p)}{(t-1)t} \cdot \Delta^2 \\
&\leq
4\cdot \frac{\mu_p^2 }{t^2}
+ 8\cdot \frac{\mu_p}{t^2} \cdot \Delta
+ 2\left[\frac{1}{t^2} + \frac{\mu_p}{t}\right] \cdot \Delta^2 + 4 \cdot \frac{\Delta^3 }{t} - 2 \cdot \frac{\Delta^4}{t}
\end{align}
\end{remark}

The expectation of the second Kravchuk polynomial under a mixture of binomials is given by
\begin{align}
E_{X\sim P_\pi}[\tilde{K}_2(X,p,t)] &= d_2
\end{align} We compute its variance under a binomial mixture. \begin{align}
V_{q\sim\pi}[E_{X \sim \text{Bin}(t,q)}[\tilde{K}_2(X,p,t)]] &= d_4 - d_2^2 \\
E_{q\sim\pi}[V_{X \sim \text{Bin}(t,q)}[\tilde{K}_2(X,p,t)]] &\leq 4\cdot \frac{\mu_p^2 }{t^2}
+ 8\cdot \frac{\mu_p}{t^2} \cdot d_1
+ 2\left[\frac{1}{t^2} + \frac{\mu_p}{t}\right] \cdot d_2 + 4 \cdot \frac{d_3 }{t} - 2 \cdot \frac{d_4}{t}
\end{align} Consequently, \begin{align}
V_{X\sim P_\pi}[\tilde{K}_2(X,p,t)] &\leq 4\cdot \frac{\mu_p^2 }{t^2}
+ 8\cdot \frac{\mu_p}{t^2} \cdot d_1
+ 2\left[\frac{1}{t^2} + \frac{\mu_p}{t}\right] \cdot d_2 + 4 \cdot \frac{d_3 }{t} + (1-\frac{2}{t}) \cdot d_4 - d_2^2
\end{align}

\section{Wasserstein metric}\label{sec:wasserstein}

\subsection{Translation of measures}

\begin{lemma}\label{lemma:w1_translation}
Let $\tilde{\pi}_1$ and $\tilde{\pi}_0$ be two distributions supported on $[-1,1]$ that share they first $k$ moments \begin{equation}
	E_{\tilde{\pi}_1}[X^l]-E_{\tilde{\pi}_0}[X^l]=0\quad \for l\leq k
\end{equation} There exists $\pi_1$ and $\pi_0$ supported on $[a,b]$ that share they first $k$ moments and satisfy \begin{equation}
	W_1(\pi_0,\pi_1) =  \frac{b-a}{2} \cdot W_1(\tilde{\pi}_0,\tilde{\pi}_1)
\end{equation}
\end{lemma} \begin{proof}[Proof of \zcref{lemma:w1_translation}]
Define the linear transformation from $[a,b]$ to $[-1,1]$\begin{equation}
    m(x) = \frac{2 \cdot x - (b+a)}{b-a}
\end{equation} and the shifted distributions \begin{equation}
    \pi_i(x) = m'(x) \cdot \tilde{\pi}_i \circ m(x).
\end{equation} Then the first $k$ moments of $\pi_0$ and $\pi_1$ are the same since \begin{equation}
    \int_{[a,b]} x^k\ d\pi_i = \int_{[-1,1]} x^k\ d\tilde{\pi}_i \quad \for i \in \{0,1\}
\end{equation} and \begin{align}
W_1(\pi_0,\pi_1) &= \sup_{f \in \Lip_1[a,b]} \int_a^b f\ d\pi_1 - d\pi_0\\
&= \sup_{f \in \Lip_1[a,b]} \int_{-1}^1 f \circ m^{-1}\ d\tilde{\pi}_1 - d\tilde{\pi}_0 \\
&= \frac{b-a}{2} \cdot \sup_{f \in \Lip_1[-1,1]} \int_{-1}^1 f\  d\tilde{\pi}_1 - d\tilde{\pi}_0\\
&= \frac{b-a}{2} \cdot W_1(\tilde{\pi}_0,\tilde{\pi}_1).
\end{align}
\end{proof}

\subsection{Concentration and locality}

\begin{lemma}\label{lemma:w1_concentration} Let $p_i \iid \pi$ for $1\leq i \leq n$, define the empirical average \begin{equation}
\tilde{\pi}=\frac{1}{n}\sum_{i=1}^n\delta_{p_i}
\end{equation} and consider the expectation of $W_1(\tilde{\pi},\pi)$. It follows that \begin{equation}
E_\pi\left[W_1(\tilde{\pi},\pi)\right] \leq \frac{J(\pi)}{\sqrt{n}} \quad \where J(\pi)=\int_{0}^1\mu^{1/2}_{F_{\pi}(x)}\ dx \textand F_{\pi}(x)=E_{p\sim\pi}I[p \leq x]\period
\end{equation} Furthermore, the variance satisfies \begin{equation}
V_{\pi}\left[W_1(\tilde{\pi},\pi)\right] \leq \left[\frac{J(\pi)}{\sqrt{n}}\right]^2\period
\end{equation} Finally, $J(\pi)$ satisfies \begin{equation}
J(\pi) \leq J(\pi_0) + \sqrt{3} \cdot \sqrt{W_1(\pi,\pi_0)}\period
\end{equation}\end{lemma}\begin{proof}[Proof of \zcref{lemma:w1_concentration}]

By Theorem 3.2 of \citet{bobkovOnedimensionalEmpiricalMeasures2019},we have that \begin{equation}
E_\pi W_1(\tilde{\pi},\pi) \leq \frac{J(\pi)}{\sqrt{n}}.
\end{equation} Due to the triangle inequality in the space $L_2$ \begin{equation}
V_{\pi}W_1(\tilde{\pi},\pi) \leq E_\pi \left[W_1(\tilde{\pi},\pi)\right]^2 \leq \left(E_\pi W_1^2(\tilde{\pi},\pi)\right) \leq \left[\frac{J(\pi)}{\sqrt{n}}\right]^2.
\end{equation} Alternatively, the generalized Minkowski inequality can be used to obtain the same result \begin{align}
&V_\pi\left[W_1(\tilde{\pi},\pi)\right]\\
&\leq E_\pi \left[W_1(\tilde{\pi},\pi)\right]^2\\
&=E_\pi\left[\int_{0}^1 |F_{\tilde{\pi}}(x)-F_\pi(x)|\ dx\right]^2\\
&\leq \left[\int_{0}^1 \left(E_\pi\left[F_{\tilde{\pi}}(x)-F_\pi(x)\right]^2\right)^{1/2}\ dx\right]^2 &&\text{By the generalized Minkowski inequality}\\
&= \left[\int_{0}^1 \left(V_\pi\left[F_{\tilde{\pi}}(x)\right]\right)^{1/2}\ dx\right]^2\\
&=\left[\int_{0}^1 \sqrt{\frac{\mu_{F_{\pi}(x)}}{n}}\ dx\right]^2\\
&= \left[\frac{J(\pi)}{\sqrt{n}}\right]^2.
\end{align}

Finally, it holds that $|\mu_{F_{\pi}(x)}-\mu_{F_{\pi_0}(x)}|=|F_{\pi}(x)-F_{\pi_0}(x)| \cdot |1+F_{\pi}(x)+F_{\pi_0}(x)|
$, which we use to expand $J(\pi)$ around $\pi_0$:  \begin{align}
\frac{J(\pi)}{\sqrt{n}} &\leq \frac{J(\pi_0)}{\sqrt{n}} + \sqrt{\frac{3}{n}}\cdot \int \left|F_{\pi}(x)-F_{\pi_0}(x)\right|^{1/2}\\
&\leq   \frac{J(\pi_0)}{\sqrt{n}} + \sqrt{3} \cdot \sqrt{\frac{W_1(\pi,\pi_0)}{n}}
\end{align} where, in the last step, we used Jensen's inequality.
\end{proof}

\subsection{Upper bound on \texorpdfstring{$W_1$}{W1} by approximating the witness function}

\begin{corollary}\label{cor:approx_witness} Let $\pi,\pi_0 \in D$, the following upper-bound holds \begin{equation}
W_1(\pi,\pi_0) \leq \frac{2\alpha_k}{k} + c_k \cdot \norm{b_t(\pi)-b_t(\pi_0)}_1
\end{equation}where $\alpha_k$ and $c_k$ are defined in \zcref{lemma:uniform_approx_by_Bernstein}.\end{corollary}
\begin{proof}[Proof of \zcref{cor:approx_witness}] Let $f_k$ be the polynomial in Bernstein form defined in \zcref{lemma:uniform_approx_by_Bernstein}, it follows that \begin{align}
W_1(\pi,\pi_0) &= \sup_{f \in \Lip_1[0,1]} \int_{0}^1 f(p)\ d\pi-d\pi_0\\
&=\sup_{f \in \Lip_1[0,1]}\int_{0}^1 f(p)-f_k(p)\ d\pi-d\pi_0 +\int_{0}^1 f_k(p)\ d\pi-d\pi_0 \period
\end{align} Using the definition of $f_k$ and recalling that $b_{j,t}(\pi)=E_{p\sim\pi}[B_{j,t}(p)]$, the following upper-bound holds\begin{align}
W_1(\pi,\pi_0) &\leq \sup_{f \in \Lip_1[0,1]} 2 \cdot \norm{f-P_f}_\infty + \sum_{j=0}^t |c_{j,k}| \cdot |b_{j,t}(\pi)-b_{j,t}(\pi_0)|\\
&\leq \sup_{f \in \Lip_1[0,1]} 2 \cdot \norm{f-P_f}_\infty + \norm{c_k}_\infty \cdot \norm{b(\pi)-b(\pi_0)}_1\end{align}	The lemma follows by the guarantees given in \zcref{lemma:uniform_approx_by_Bernstein}.
\end{proof}

\begin{lemma}\label{lemma:uniform_approx_by_Bernstein}
Let $f \in \Lip_1[0,1]$ and $k \in N \st k \leq t$, there exists a Bernstein polynomial $p_k$ of order $t$ \begin{equation}
p_k(x) = \sum_{j=0}^t c_{j,k} \cdot B_{j,t}(x)
\end{equation} that satisfies \begin{equation}
\sup_{x \in \left[0,1\right]}|f(x)-p_k(x)| \leq \frac{\alpha_k}{k} \where \alpha_k = \begin{cases}
2 &\textif k =\sqrt{t}\\
\pi/2 &\textif \sqrt{t}<k\leq t
\end{cases}
\end{equation} and \begin{equation}
\norm{c_k}_\infty \leq \begin{cases}
	\sqrt{t}\cdot2^t &\textif k=t\\
	\sqrt{k}\cdot (t+1) \cdot e^{k^2/t} &\textif k<t\\
	1 &\textif k =\sqrt{t}
	\end{cases}
\end{equation}
\end{lemma}\begin{proof}[Proof of \zcref{lemma:uniform_approx_by_Bernstein}]

\underline{\textbf{For $k=\sqrt{t}$:}} Given $f \in \Lip_1[0,1]$
Given $f \in \Lip_1[0,1]$, by proposition 4.9 of \citet{bustamanteBernsteinOperators2017}, it holds that \begin{equation}
\sup_{x \in \left[0,1\right]}|f(x)-p_k(x)| \leq \frac{2}{\sqrt{t}},
\end{equation} where $
|c_{j,k}| = |f\left(\frac{j}{t}\right)| \leq \frac{j}{t}$ since $f \in \Lip_1$. Thus, $
\norm{c_k}_\infty \leq  1$.

\underline{\textbf{For $\sqrt{t} < k \leq t$:}} given $f \in \Lip_1[0,1]$, there exists a trigonometric polynomial of degree at most $k$ that uniformly approximates $f$ on its domain, see Theorem 3 of \cite{tian2017} or Theorem 7.2 of \cite{plaskotaApproximationComplexity2021}, \begin{equation}
\inf_{p_k \in\mathcal{P}_k}\norm{f-p_k}_{\infty[0,1]} \leq \frac{\pi}{2} \cdot \frac{1}{k}\period
\end{equation} Let $p_k$ be the polynomial that achieves the infimum; it can be expressed in the shifted Chebyshev basis \begin{equation}\label{eq:opt_polynomial}
p_k(x) = \sum_{m=0}^k a_m(f) \cdot T_m(x) \where \norm{a}_2^2 \leq 1\period
\end{equation} Furthermore, since $k \leq t$, $p_k$ can be expressed in the Bernstein basis of degree $t$. \citet{vinayakMaximumLikelihoodEstimation2019} showed that the Chebyshev basis can be rewrriten as  \begin{equation}
T_m(x) = \sum_{j=0}^t C_{t,m,j} \cdot B_{j,t}(x) \where C_{t,m,j} = \sum_{l=\max(0,j+m-t)}^{\min(j,m)} (-1)^{m-l} \cdot  \frac{\binom{2m}{2l}\binom{t-m}{j-l}}{\binom{t}{j}}
\end{equation} when $m=t$, it reduces to \begin{equation}
	T_m(x) = \sum_{j=0}^m C_{t,m,j} B_{j,m}(x) \where C_{t,m,j} = (-1)^{m-j} \cdot  \frac{\binom{2m}{2j}}{\binom{m}{j}}
\end{equation} Thus, by plugging into equation \eqref{eq:opt_polynomial}, we get \begin{equation}
p_k(x) = \sum_{m=0}^k a_m T_m(x) = \sum_{j=0}^t c_j \cdot B_{j,t}(x) \where c_j = \sum_{m=0}^k a_m \cdot C_{t,m,j}
\end{equation} The proof finishes using proposition \zcref{prop:coeff_bound}.\end{proof}

\begin{proposition}[Extension of \citet{rababahTransformationChebyshevBernstein2003} and \citet{vinayakMaximumLikelihoodEstimation2019}]\label{prop:coeff_bound}
\begin{equation}
\norm{c}_\infty \leq \begin{cases}
\sqrt{t+1}\cdot2^t &\textif k=t\\
\left(\sqrt{k+1} \cdot e^k\right) \wedge \left(\sqrt{k}\cdot(t+1)\cdot e^{k^2/t}\right) &\textif \sqrt{t}<k\leq t
\end{cases}
\end{equation}
\end{proposition}
\begin{proof}[Proof of \zcref{prop:coeff_bound}]

\underline{\textbf{For $k= t$:}}  The following upper-bound on $\norm{c}_\infty$ always holds:\begin{align}
\norm{c}_\infty = \max_{j \leq t} |c_j| &\leq \max_{j \leq t} \norm{a}_1 \cdot \max_{m\leq k} |C_{t,m,j}|\\
&\leq \sqrt{t+1} \cdot \max_{j\leq t, m\leq k} |C_{t,m,j}| &&\since  \norm{a}_2\leq 1.
\end{align} We exploit the strategy employed in Lemma 5 of \citet{rababahTransformationChebyshevBernstein2003} by interpreting the coefficients as a convex combination: \begin{equation}
	|C_{t,m,j}| \leq \sum_{l=\max(0,j+m-t)}^{\min(j,m)} \frac{\binom{2m}{2l}\binom{t-m}{j-l}}{\binom{t}{j}} = \sum_{l=\max(0,j+m-t)}^{\min(j,m)} w_l \cdot \frac{\binom{2m}{2l}}{\binom{m}{l}}\quad  \where w_l = \frac{\binom{m}{l}\binom{t-m}{j-l}}{\binom{t}{j}}.
\end{equation} Note that $\{w_l\}_{l=0}^j$ form a partition of unity, i.e. $\sum_{l=0}^j w_l = 1$ by the Chu–Vandermonde's identity. Thus, the upper bound of $|C_{t,m,j}|$ is a convex combination of $\frac{\binom{2m}{2l}}{\binom{m}{l}}$ and it follows that \begin{equation}
|C_{t,m,j}| \leq \max_{\max(0,j+m-t)\leq l\leq\min(j,m)} \frac{\binom{2m}{2l}}{\binom{m}{l}}.
\end{equation} Furthermore,  $\frac{\binom{2m}{2l}}{\binom{m}{l}}$ achieves a maximum at $l=m/2$. Thus, it follows that \begin{equation}
|C_{t,m,j}| \leq \frac{\binom{2m}{m}}{\binom{m}{m/2}} \leq 2^m \quad \since \binom{2m}{m} \asymp \frac{4^m}{\sqrt{m}}(1-1/m).
\end{equation} Finally, we get that \begin{equation}
\max_{j\leq t, m\leq k} |C_{t,m,j}| \leq 2^t.
\end{equation}

\underline{\textbf{For $k\leq t$:}} By the Cauchy-Schwarz inequality, 
\begin{equation}
|c_j| \leq \norm{a}_2 \cdot \left(\sum_{m=0}^k C_{t,m,j}^2\right)^{1/2} \leq \left(\sum_{m=0}^k C_{t,m,j}^2\right)^{1/2} \since \norm{a}_2 \leq 1.
\end{equation} Furthermore, tecall that \begin{equation}
C_{t,m,j} = \sum_{l} (-1)^{m-j} \cdot \frac{\binom{2m}{2l}}{\binom{m}{l}} \cdot w_l.
\end{equation} Thus, \begin{align}
\sum_{m=0}^kC_{t,m,j}^2 &\leq \sum_{m=0}^k\sum_{l}^{\min(j,m)} w_l \cdot \left[\frac{\binom{2m}{2l}}{\binom{m}{l}} \right]^2 &&\text{by Jensen's inequality}\\
&\leq \sum_{m=0}^k\sum_{l}^{\min(j,m)}  w_l \cdot e^{2l}\\
&\leq \sum_{m=0}^k e^{2\min(j,m)} \cdot \sum_{l} w_l\\
&\leq (k+1) \cdot e^{2k} &&\since \sum_{l} w_l \leq 1.
\end{align} Consequently $|c_j| \leq \sqrt{k+1} \cdot e^k$ for all $j$, and $\norm{c}_\infty \leq \sqrt{k+1} \cdot e^k$.

\underline{\textbf{For $k<t$:}} The following bounds follows from the results of \citet{vinayakMaximumLikelihoodEstimation2019}: \begin{align}
	\norm{c}_\infty = \max_{j \leq t} |c_{j}| &\leq \max_{j \leq t} \norm{a}_1 \cdot \max_{m\leq k} |C_{t,m,j}|\\
	&\leq \sqrt{k} \cdot \max_{j\leq t, m\leq k} |C_{t,m,j}| &&\since  \norm{a}_2\leq 1\\
	&\leq \sqrt{k} \cdot \max_{m\leq k} \sqrt{\sum_{j=0}^t C_{t,m,j}^2}\\
	&\leq  \sqrt{k}(t+1)e^{\frac{k^2}{t}} &&\text{by lemma 4.4 of \citet{vinayakMaximumLikelihoodEstimation2019}}.
\end{align}

\end{proof}

\section{General upper bound on the critical separation}

\begin{lemma}\label{mean_difference_dominates_std} Consider the hypotheses \begin{equation}
H_0: P \in D_0 \vs H_1: P \in D_1
\end{equation} and the test statistic $T$. Let $q_{\alpha}(P,T)$ be $1-\alpha$ quantile of the test statistic under the distribution $P$ \begin{equation}\label{eq:quantile_definition}
q_{\alpha}(P,T) = \inf\left\{u : P(T > u)\leq \alpha \right\},
\end{equation} and define the test \begin{equation}\label{eq:test_definition}
\psi = I(T \geq t) \where t = \sup_{P \in D_0} q_{\alpha}(P,T).
\end{equation} The test controls the type I error: \begin{equation}
\sup_{P\in D_0} P(\psi = 1)\leq \alpha \period
\end{equation} Furthermore, if it holds that  \begin{equation}\label{eq:power_condition}
t \leq \inf_{P \in D_1} E_{P}[T] - \sqrt{\frac{V_{P}[T]}{\beta}},
\end{equation} then the test controls the type II error \begin{equation}
\sup_{P\in D_1} P(\psi = 0)\leq \beta \period
\end{equation} Finally, since $q_\alpha(P,T)\leq E_P[T] +  \sqrt{V_P[T]/\alpha}$, it follows that \eqref{eq:power_condition} is implied by \begin{equation}
\sup_{P \in D_0}E_{P_0}[T] + \sqrt{\frac{V_{P_0}[T]}{\alpha}} \leq \inf_{P \in D_1}  E_{P_1}[T] - \sqrt{\frac{V_{P_1}[T]}{\beta}}\period
\end{equation}
\end{lemma}

\begin{proof}[Proof of \zcref{mean_difference_dominates_std}]

If $t = \sup_{P \in D_0} q_{\alpha}(P)$, it follows that the type I error is controlled
\begin{align}
\sup_{P\in D_0} P(\psi = 1) &= \sup_{P \in D_0} P(T \geq t) &&\by \eqref{eq:test_definition}\\
&= \sup_{P \in D_0} P(T \geq \sup_{P \in D_0} q_\alpha(P,T))\\
&\leq \sup_{P \in D_0} P(T \geq  q_\alpha(P,T))\\
&\leq \alpha &&\by \eqref{eq:quantile_definition}.
\end{align}

Let $
t_{\inf} =\inf_{P \in D_1} E_{P}[T] - \sqrt{\frac{V_{P}[T]}{\beta}}$, then the type II error is controlled: \begin{align}
&\sup_{P\in D_1} P(\psi = 0)\\
&= \sup_{P \in D_1} P(T < t)\\
&\leq\sup_{P\in D_1} P(T < t_{\inf}) \\
&= \sup_{P\in D_1} P(E_P[T] - T > E_P[T] - t_{\inf})\\
&\leq \sup_{P\in D_1} P(|E_P[T] - T| > E_P[T] - t_{\inf})\\
&\leq\sup_{P\in D_1} \frac{V_P[T]}{(E_{P}[T]-t_{\inf})^2} &&\text{By Markov's inequality since } t_{\inf}\leq E_P[T] \\
&\leq\sup_{P\in D_1} \frac{V_P[T]}{\sup_{P\in D_1} V_P[T]/\beta}=\beta.\end{align}\end{proof}

\section{General lower bound on the critical separation}\label{sec:general_lemma_lb}

\begin{lemma}[General lower-bound]\label{lemma:general_lb}
Let $D_0$ be a set of mixing distributions. Let $\Psi_\alpha$ be the set of all tests that control the type I error over $D_0$: \begin{equation}
\Psi_\alpha = \left\{\psi : \sup_{\pi \in D_0} P_{\pi}^n(\psi(X) = 1)\leq \alpha\right\},
\end{equation} and define the risk $R$ as the maximum type II error over such set \begin{equation}
R_*(\epsilon) = \inf_{\psi\in \Psi_\alpha}\sup_{\pi_0\in D_0}\sup_{\pi: W_1(\pi,D_0)\geq \epsilon} P_{\pi}^n(\psi(X)=0)
\end{equation} where $
W_1(\pi,D_0) = \inf_{\pi_0\in D_0}W_1(\pi,\pi_0)$. Finally, consider a mixing distribution $\pi_1$ and a distribution over mixing distributions on the null space $\Gamma_0$, that is, $
\supp(\Gamma_0) \subseteq D_0$. If $\pi_1$ and $\Gamma_0$ are such that \begin{equation}
V\left(E_{\pi_0 \sim \Gamma_0}\left[P_{\pi_0}^n\right]\ ,\ P_{\pi_1}^n \right) < C_\alpha \textor \chi^2\left(E_{\pi_0 \sim \Gamma_0}\left[P_{\pi_0}^n\right]\ ,\ P_{\pi_1}^n \right) < C_\alpha^2
\end{equation} where $C_\alpha = 1 - (\alpha+\beta)$, then the critical separation 
$\epsilon_* = \inf\{\epsilon : R_*(\epsilon) \leq \beta \}$ is lower-bounded: \begin{equation}
\epsilon_* \geq W_1(\pi_1,D_0)\period
\end{equation}

\end{lemma} \begin{proof}[Proof of \zcref{lemma:general_lb}]

By definition, we have that
\begin{align}
R_*(\epsilon)&= \inf_{\psi\in \Psi_\alpha}\sup_{\pi_0\in D}\sup_{\pi: W_1(\pi,\pi_0)\geq \epsilon} P_{\pi}^n(\psi(X)=0)\\
&= \inf_{\psi\in \Psi_\alpha}\sup_{\pi_0\in D_0}\sup_{\pi: W_1(\pi,D_0)\geq \epsilon} P_{\pi}^n(\psi(X)=0)+P_{\pi_0}^n(\psi(X)=1)-P_{\pi_0}^n(\psi(X)=1)\\
&\geq \inf_{\psi\in \Psi_\alpha}\sup_{\pi_0\in D_0}\sup_{\pi: W_1(\pi,D_0)\geq \epsilon} P_{\pi}^n(\psi(X)=0)+P_{\pi_0}^n(\psi(X)=1)-\alpha
\end{align} where in the last line we used the fact that $\psi \in \Psi_\alpha $. Since the supremum dominates any particular instance, we have that \begin{equation}
R_*(\epsilon)\geq \inf_{\psi\in \Psi_\alpha}\sup_{\pi_0\in D_0}P_{\pi_1}^n(\psi(X)=0)+P_{\pi_0}^n(\psi(X)=1)-\alpha \period
\end{equation} Since the supremum dominates any average, we have that \begin{equation}
R_*(\epsilon)\geq \inf_{\psi\in \Psi_\alpha}P_{\pi_1}^n(\psi(X)=0)+E_{\pi_0\sim \Gamma_0}P_{\pi_0}^n(\psi(X)=1)-\alpha \period
\end{equation} By Theorem 2.2 of \citet{tsybakovIntroductionNonparametricEstimation2009}, it follows that \begin{equation}
R_*(\epsilon)\geq 1-V\left(E_{\pi_0\sim \Gamma_0}P_{\pi_0}^n\ ,\ P_{\pi_1}^n\right)-\alpha \period
\end{equation} Finally, by the assumption $V\left(E_{\pi_0\sim \Gamma_0}P_{\pi_0}^n\ ,\ P_{\pi_1}^n\right)<C_\alpha$, it follows that \begin{equation}
R_*(\epsilon) \geq 1- C_\alpha - \alpha > \beta \period
\end{equation} Alternatively, the same claim follows whenever $\chi^2\left(E_{\pi_0\sim \Gamma_0}P_{\pi_0}^n\ ,\ P_{\pi_1}^n\right)<C_\alpha^2$ since \begin{equation}
V\left(E_{\pi_0\sim \Gamma_0}P_{\pi_0}^n\ ,\ P_{\pi_1}^n\right) \leq \sqrt{\chi^2\left(E_{\pi_0\sim \Gamma_0}P_{\pi_0}^n\ ,\ P_{\pi_1}^n\right)}<C_\alpha.
\end{equation} Consequently, $
R_*(\epsilon) > \beta$ $\forall\ \epsilon \geq W_1(\pi_1,D_0)$. Thus, it follows by definition of the critical separation that $\epsilon_* > W_1(\pi_1, D_0)$.

\end{proof}

\LeCamLowerBound*
\begin{proof}[Proof of \zcref{lemma:lecam_lb}]

Let $D_0=\left\{\pi_0\right\}$ and $\Gamma_0=\delta_{\pi_0}$, by \zcref{lemma:general_lb} by setting, it follows that \begin{equation}
\epsilon_* \geq W_1(\pi,\pi_0) \quad\textif V\left(P_{\pi_0}^n,P_{\pi_1}^n\right) < C_\alpha.
\end{equation} Optimizing over all mixing distributions in $D$ and using the sublinearity of the total variation distance $V(P^n, Q^n)\leq 2nV(P, Q)$ yields the lemma's statement.
\end{proof}

\section{Critical separation for goodness-of-fit testing}

The section characterizes the critical separation \eqref{eq:critical_separation} for goodness-of-fit testing \eqref{eq:w1_testing}. We analyze the performance of the debiased Pearson's $\chi^2$ test \eqref{eq:fingerprint_test} in \zcref{sec:fingerprint_test_proof}, and the plug-in test \eqref{eq:w1_plugin_test} in \zcref{sec:plugin_test_proof}. Furthermore, we prove the general lower-bound on the critical separation (\zcref{lemma:global_minimax_rates}) in \zcref{sec:global_minimax_lowerbounds}.

\subsection{Upper bounds on the critical separation}

\subsubsection{Debiased Pearson's \texorpdfstring{$\chi^2$}{X2}   test}\label{sec:fingerprint_test_proof}

\FingerprintTest*
\begin{proof}[Proof of \zcref{lemma:fingerprint_test}]
Let $\phi$ be the $\alpha$-level minimax test for \begin{align}
H_0: b_t(\pi) = b_t(\pi_0) \vs
H_1: \norm{b_t(\pi)-b_t(\pi_0)}_1 \geq \gamma\period
\end{align} By \zcref{thm:multinomial_l1_rates}, it holds that \begin{equation}
\gamma \leq C \cdot \frac{t^{1/4}}{n^{1/2}}\period
\end{equation} where $C$ is a positive constant. Under the alternative, by \zcref{cor:approx_witness}, we have that: \begin{equation}\label{eq:upper_bound_coeff}
\epsilon \leq W_1(\pi,\pi_0) \leq \frac{\pi}{k} + c_k \cdot \norm{b(\pi)-b(\pi_0)}_1 \where c_k \leq \begin{cases}
\sqrt{t}\cdot2^t &\textif k=t\\
\sqrt{k}\cdot (t+1) \cdot e^{k^2/t} &\textif k<t
\end{cases}
\end{equation} Thus, we have that the debiased Pearson's $\chi^2$ test \eqref{eq:fingerprint_test} controls the type II error by $\beta$ whenever \begin{equation}\label{eq:power_condition_approx}
\frac{\epsilon}{2} \geq \frac{\pi}{k} + C \cdot  c_k \cdot \frac{t^{1/4}}{n^{1/2}}\period
\end{equation} \zcref[S]{eq:power_condition_approx} is implied by the following condition if the approximation error dominates \begin{equation}\label{eq:approx_error_dominates}
\epsilon \geq \frac{4\pi}{k} \quad\textif\quad \frac{\pi}{k} \geq C \cdot c_k \cdot \frac{t^{1/4}}{n^{1/2}}\period
\end{equation}

\underline{\textbf{Case I: $\epsilon \gtrsim 1/t$ for $t \lesssim \log n$}}

Choose $k=t$, then $c_k \leq \sqrt{t}\cdot 2^t$ by \eqref{eq:upper_bound_coeff}. Thus, $\epsilon \geq 4\pi/t$ in so far as  \begin{equation}
\frac{\pi}{C}\cdot n^{1/2} \geq t^{3/4} \cdot 2^t\comma
\end{equation} which is implied by $
\left(\frac{\pi}{C}\cdot n^{1/2}\right)^{4/3} \geq t \cdot e^t$. Consequently, the condition is satisfied in so far as \begin{equation}\label{eq:lambert_cond}
t \leq W_0\left[\left(\frac{\pi}{C}\cdot n^{1/2}\right)^{4/3}\right]
\end{equation} where $W_0$ is the real branch of the Lambert $W$ function. Furthermore, we have that \eqref{eq:lambert_cond} is implied by \begin{equation}
t \leq \frac{1}{2}\cdot \log\left(\left[\frac{\pi}{C}\cdot n^{1/2}\right]^{4/3}\right) \for \left(\frac{\pi}{C}\cdot n^{1/2}\right)^{4/3} \geq e \period
\end{equation} Thus, it follows that:  \begin{equation}
\epsilon \geq \frac{C_0}{t} \for n \geq C_1 \textand t \leq C_2 \cdot \log n\period
\end{equation}

\underline{\textbf{Case II: $\epsilon \gtrsim 1/\sqrt{t\log n}$ for $\log n \lesssim t$ and $\log t \lesssim \log n$ }}

Let $c>0$, and choose $k = \sqrt{t \log n^c}$, then $k < t$ if $t > c \cdot \log n$. Thus, by \eqref{eq:upper_bound_coeff}, it holds that $c_k \leq \sqrt{k} \cdot 2t \cdot e^{k^2/t}$. Furthermore, by \eqref{eq:approx_error_dominates}, it holds that $\epsilon \geq 4\pi/\sqrt{t\log n^c}$ in so far  \begin{equation}
\frac{\pi}{2C}\cdot \frac{n^{1/2}}{t^{5/4}} \geq k^{3/2} \cdot 2^{k^2/t}\end{equation} which is implied by \begin{equation}
\left(\frac{\pi}{2C}\cdot \frac{n^{1/2}}{t^{5/4}}\right)^{4/3}\cdot \frac{1}{t} \geq \frac{k^2}{t}\cdot \left(2^{4/3}\right)^{\frac{k^2}{t}}
\end{equation} and can be reduced to \begin{equation}
\left(\frac{\pi}{2C}\cdot \frac{n^{1/2}}{t^{5/4}}\right)^{4/3}\cdot \frac{1}{t} \geq \frac{k^2}{t}\cdot e^{\frac{k^2}{t}}\period
\end{equation} Plugging the chosen $k$, we get \begin{equation}
\left(\frac{\pi}{2C}\cdot \frac{n^{1/2}}{t^{5/4}}\right)^{4/3}\cdot \frac{1}{t} \geq n^c \cdot \log n^c
\end{equation} which holds if \begin{equation}
c \cdot \log n< t \leq \left[\left(\frac{\pi}{2C} \right)^{4/3}\cdot \frac{n^{2/3-c}}{c\cdot \log n}\right]^{3/8}\period
\end{equation} Choose a fixed $\delta >0 $ and letting $c = \delta \frac{8}{3}$, it holds that: \begin{equation}
\epsilon \geq \frac{C_0}{\sqrt{\delta \cdot t\log n}} \quad \textif \quad C_1\cdot \delta \cdot \log n < t \leq C_2 \cdot \frac{n^{1/4-\delta}}{(\delta \log n)^{3/8}}
\end{equation} where $C_0 = \frac{4\pi}{\sqrt{8/3}}$, $C_1=\frac{8}{3}$ and $C_2 = \left(\frac{3}{8}\right)^{3/8}\left(\frac{\pi}{2C} \right)^{1/2}$.
\end{proof}

\subsubsection{Plug-in estimator of $W_1(\pi,\pi_0)$}\label{sec:plugin_test_proof}

\PlugInTestRates*
\begin{proof}[Proof of \zcref{lemma:plugin_test}]
Let $
T(X) = W_1(\hat{\pi},\pi_0)
$. By \zcref{mean_difference_dominates_std}, the test control the type II error by $\beta$ if \begin{equation}\label{eq:power_condition_w1}
E_{P_\pi}\left[T\right]-E_{P_{\pi_0}}\left[T\right] \geq \sqrt{\frac{V_{P_{\pi}}\left[T\right]}{\beta}} + \sqrt{\frac{V_{P_{\pi_0}}\left[T\right]}{\alpha}} \period
\end{equation} Since the standard deviation of the statistic is dominated by its expectation  \begin{equation}
\sqrt{V_{P_\pi}\left[T\right]}\leq \sqrt{E_{P_\pi}\left[W_1^2(\hat{\pi},\pi_0)\right]} \leq \sqrt{\left(E_{P_\pi}\left[W_1(\hat{\pi},\pi_0)\right]\right)^2}=E_{P_\pi}\left[T\right],
\end{equation} condition \eqref{eq:power_condition_w1} reduces to \begin{equation}\label{eq:power_condition_w1_2}
E_{P_\pi}\left[T\right] \geq C \cdot E_{P_{\pi_0}}\left[T\right] \where C=\frac{1-\beta^{-1/2}}{1-\alpha^{-1/2}} \period
\end{equation} By the triangle inequality, we have that
\begin{equation}
|E_{P_{\pi}}\left[T\right]-W_1(\pi,\pi_0)| \leq E_{\pi}\left[W_1(\tilde{\pi},\pi)\right] + E_{P_\pi}\left[W_1(\hat{\pi},\tilde{\pi})\right] \period
\end{equation} For the first term, by \zcref{lemma:w1_concentration}, it holds that \begin{align}
E_{\pi}\left[W_1(\tilde{\pi},\pi)\right] \leq \frac{J(\pi)}{\sqrt{n}} \leq \frac{J(\pi_0)}{\sqrt{n}} + \sqrt{3} \cdot \sqrt{\frac{W_1(\pi,\pi_0)}{n}}
\end{align} For the second term, it will be shown that \begin{equation}\label{eq:locality_empirical_measure}
E_{P_\pi}\left[W_1(\tilde{\pi},\hat{\pi})\right] \leq C \cdot \left[\sqrt{\frac{E_{\pi_0}\mu_p}{t}} +  \sqrt{\frac{W_1(\pi,\pi_0)}{t}} + \frac{1}{t} \right]
\end{equation} Using these inequalities, we have that \eqref{eq:power_condition_w1_2} is implied by \begin{equation}
W_1(\pi,\pi_0) \geq C' \cdot \left[\frac{J(\pi_0)}{\sqrt{n}} +  \sqrt{\frac{W_1(\pi,\pi_0)}{n}} + \sqrt{\frac{E_{\pi_0}\mu_p}{t}} +  \sqrt{\frac{W_1(\pi,\pi_0)}{t}} + \frac{1}{t}\right]\comma
\end{equation} which is implied by \begin{equation}
\epsilon \geq C' \cdot \left[\frac{J(\pi_0)}{\sqrt{n}} + \sqrt{\frac{E_{\pi_0}\mu_p}{t}} + \frac{1}{n} + \frac{1}{t}\right] \period
\end{equation} Thus, the lemma is proved.

Finally, we prove \eqref{eq:locality_empirical_measure}. By triangle inequality, it holds that \begin{align}
&E_{P_\pi}\left[W_1(\tilde{\pi},\hat{\pi})\right]\\
&= E_{P_\pi}\left[\int_{0}^1\left|F_{\hat{\pi}}(x)-F_{\tilde{\pi}}(x)\right|\ dx\right]\\
&= E_{P_\pi}\left[\int_{0}^1\left|\frac{1}{n}\sum_{i=1}^n\left(I\left(\frac{X_i}{t}\leq x\right)-I\left(p_i\leq x\right)\right)\right|\ dx\right]\\
&\leq n^{-1}\cdot \sum_{i=1}^n\ \left[\ E_\pi\int_{0}^1 E_{P_\pi|\pi} \left|I\left(\frac{X_i}{t}\leq x\right)-I\left(p_i\leq x\right) \right|\ dx\right]
\end{align} where, in the last inequality, we used the triangle inequality and Tonelli's Theorem. Since all the terms are identically distributed, it holds that: \begin{align}
E_{P_\pi}\left[W_1(\tilde{\pi},\hat{\pi})\right] &\leq E_\pi\int_{0}^1 E_{P_\pi|\pi} \left|I\left(\frac{X}{t}\leq x\right)-I\left(p\leq x\right) \right|\ dx \period
\end{align} Recall that $P_\pi|\pi \stackrel{d}{=} \Bin(t,p)$. Hence, we condition on the mixing distribution and proceed to bound the integrand. By definition of the indicator function, it follows that \begin{align}
E_{P_\pi|\pi}|I(X/t \leq x)-I(p \leq x)|&= E_{P_\pi|\pi}\left[I(X/t \leq x \leq p) \lor I(p \leq x \leq X/t)\right]\\
&\leq E_{P_\pi|\pi}\left[I(X/t \leq x \leq p)\right] \lor E_{P_\pi|\pi}\left[I(p \leq x \leq X/t)\right]\period
\end{align} For the first term, by Bernstein's inequality, it holds that \begin{align}
E_{P_\pi|\pi}\left[I(X_i/t \leq x \leq p_i)\right] &\leq E_{P_\pi|\pi}\left[I\left(\frac{X}{t}-p \leq x-p\right)\right] \\
&\leq C \cdot \exp\left[-t\left(\frac{(x-p)^2}{\mu_{p}} \wedge (x-p)\right)\right]\period
\end{align} An analogous bound follows for the second term.  Consequently \begin{align}
&E_{P_\pi}\left[W_1(\tilde{\pi},\hat{\pi})\right]\\
&\leq E_{\pi}\left[\int_0^1E_{P_\pi|\pi}\left[\ |I(X/t \leq x)-I(p \leq x)|\ \right]\right]\\
&\leq C \cdot E_{\pi}\left[\int_0^1\exp\left[-t\left(\frac{(x-p)^2}{\mu_{p}} \wedge (x-p)\right)\right]\right]\\
&=C \cdot  E_{\pi}\int_{p-\mu_{p}}^{p+\mu_{p}}\exp\left[-t\frac{(x-p)^2}{\mu_{p_i}}\right] + \int_{0}^{p-\mu_{p}} \exp\left[-t(p-x)\right] +  \int_{p+\mu_{p}}^1 \exp\left[-t(x-p)\right]\\
&\leq C \cdot  E_{\pi}\left[\sqrt{\frac{\mu_{p}}{t}} + \frac{2\exp[-t\mu_{p}]-\exp[-t(1-p)]-\exp[-t p]}{t} \right]\\
&\leq  C \cdot E_{\pi} \left[\sqrt{\frac{\mu_{p}}{t}} + \frac{\exp[-t\cdot \mu_{p}]}{t}\right]\\
&\leq C \cdot \left[\frac{E_{\pi}\mu^{1/2}_{p}}{t^{1/2}} + \frac{1}{t}\right]\\
&\leq C \cdot \left[\sqrt{\frac{E_{\pi}\mu_{p}}{t}} + \frac{1}{t}\right] \period
\end{align}

Finally, to localize the result, note that $\mu_p$ is $3$-Lipschitz since \begin{equation}
|\mu_x-\mu_y|\leq |x-y|+|x^2-y^2|=|x-y|(1+|x+y|)\leq 3|x-y|\period
\end{equation} Thus, using the fact that $W_1$ can be represented as taking a supremum over $1$-Lipschitz, see \eqref{eq:w1_dual}, it follows that \begin{align}
E_\pi\mu_p\leq E_{\pi_0}\mu_p+|E_{\pi}\mu_p-E_{\pi_0}\mu_p|\leq E_{\pi_0}\mu_p + 3 \cdot W_1(\pi,\pi_0) \period
\end{align} Hence, we have that \begin{equation}
E_{P_\pi}\left[W_1(\tilde{\pi},\hat{\pi})\right] \leq C \cdot \left[\sqrt{\frac{E_{\pi_0}\mu_p}{t}} +  \sqrt{\frac{W_1(\pi,\pi_0)}{t}} + \frac{1}{t} \right] \period
\end{equation} which proves \eqref{eq:locality_empirical_measure}.\end{proof}

\subsection{Lower bounds on the critical separation}\label{sec:global_minimax_lowerbounds}

The following lemma is extensively used to prove \zcref{lemma:global_minimax_rates}.

\begin{lemma}\label{lemma:w1_max_distance} Let $D[a,b]$ denote the set of distributions supported on $[a,b]$. Let $\delta \in [0,1/2]$, and consider the following optimization problem over moment-matching distributions \begin{align}\label{eq:w1_opt_1}
\epsilon = \sup_{\pi_1,\pi_0 \in D[\frac{1}{2}-\delta,\frac{1}{2}+\delta]} W_1(\pi_1,\pi_0) &\st m_l(\pi_1)=m_l(\pi_0)\for 1\leq l\leq L.
\end{align} There exists a universal positive constant $C$ such that $
\epsilon \geq C \cdot \delta/L\period$\end{lemma}\begin{proof}[Proof of \zcref{lemma:w1_max_distance}]
By \zcref{lemma:w1_translation}, \eqref{eq:w1_opt_1} is equal to \begin{equation}\label{eq:w1_opt_2}
\delta \cdot \sup_{\pi_1,\pi_0\in D[-1,1]} W_1(\pi_1,\pi_0) \st m_l(\pi_1)=m_l(\pi_0)\for 1\leq l\leq L.
\end{equation} Using the duality of 1-Wasserstein distance \eqref{eq:w1_dual}, it holds that \begin{equation}\label{eq:w1_opt_3}
W_1(\pi_1,\pi_0) \geq E_{X\sim\pi_1}|X|-E_{X\sim\pi_0}|X| \quad \since |\cdot| \in \Lip_1[-1,1].
\end{equation} Therefore, by \eqref{eq:w1_opt_2} and \eqref{eq:w1_opt_3}, it holds that \eqref{eq:w1_opt_1} is lower-bounded by \begin{equation}
\delta \cdot \sup_{\pi_1,\pi_0 \in D[-1,1]} E_{X\sim \pi_1}|X|-E_{X\sim\pi_0}[X] \st m_l(\pi_1)=m_l(\pi_0)\for l\leq L\period
\end{equation} The statement follows by applying \zcref{lemma:best_poly_approx}, and using the fact that $
A(1,L) = (\beta_1 + C_{L})/L
$ where $C_{L} \to 0$ as $L \to \infty$ and $\beta_1 \approx 0.29$ \citep{bernsteinOrdreMeilleureApproximation1912}.
\end{proof}

\GlobalMinimaxRates*
\begin{proof}[Proof of \zcref{lemma:global_minimax_rates}]
The statement is equivalent to proving that there exist mixing distributions $\pi_1$ and $\pi_0$ in $D$ that satisfy \begin{equation}
W_1(\pi_1,\pi_0) \gtrsim \begin{dcases}
\frac{1}{t} &\for  t \lesssim \log n\\
\frac{1}{\sqrt{t \log n}} &\for  \log n \lesssim t \lesssim \frac{n}{\log n}\\
\frac{1}{\sqrt{n}} &\for t \gtrsim \frac{n}{\log n}
\end{dcases}
\end{equation} and no valid test can control the type II error by $\beta$.

Recall that by \zcref{lemma:lecam_lb}, we want to find two mixing distribution that maximize $W_1(\pi_1,\pi_0)$ constrained to satisfy $V(P_{\pi_1},P_{\pi_0})\lesssim\frac{1}{n}$.

\underline{\textbf{Small $t$ regime}} 

By \zcref{lemma:w1_max_distance}, there exists mixing distributions $\pi,\pi_0$ supported on $[0,1]$ that match their first $t$ moments and satisfy \begin{equation}
W_1(\pi_1,\pi_0) \geq C \cdot \frac{1}{2t}.
\end{equation} By \zcref{lemma:tvbound}, their marginal measures are indistinguishable \begin{equation}
V(P_{\pi_1},P_{\pi_0}) \leq \frac{d(\pi_1,\pi_0)}{2} = 0.
\end{equation} Therefore, by \zcref{lemma:lecam_lb}, it holds that
\begin{equation}
\epsilon_*(n,t) \geq W_1(\pi_0,\pi_1)\geq C \cdot \frac{1}{2t}.
\end{equation}

\underline{\textbf{Medium $t$ regime}} 

By \zcref{lemma:w1_max_distance}, there exist mixing distributions $\pi,\pi_0$ supported on $[\frac{1}{2}-\delta,\frac{1}{2}+\delta]$ where $0 < \delta \leq \frac{1}{\sqrt{8}} $ that match their first $L$ moments and satisfy \begin{equation}
W_1(\pi_1,\pi_0) \geq C \cdot \frac{\delta}{L}
\end{equation} By \zcref{lemma:tvbound} with $p=\frac{1}{2}$, the following bound on their marginal measures follows \begin{align}
&V(P_{\pi_1},P_{\pi_0})\\
&\leq \frac{1}{2} \cdot \sqrt{\sum_{m=L+1}^t \binom{t}{m} \cdot \Delta_m^2 \cdot 4^m} &&\where \Delta_m=E_{\pi_1}\left[X-\frac{1}{2}\right]^m-E_{\pi_0}\left[X-\frac{1}{2}\right]^m\\
&\leq \sqrt{\sum_{m=L+1}^t \binom{t}{m} (4\delta^2)^m} &&\since \forall x\in \supp(\pi_0)\cup\supp(\pi_1)\  |x-\frac{1}{2}|\leq\delta\\
&\leq \binom{t}{L+1}^{1/2} a^{\frac{L}{2}} \cdot \sqrt{a\cdot \frac{1-a^{t-L}}{1-a}} &&\where a = 4\delta^2\\
&\leq \binom{t}{L+1}^{1/2} (2\delta)^L &&\since a\frac{1-a^{t-L}}{1-a} \leq 1  \impliedby a \leq \frac{1}{2} \impliedby \delta \leq \frac{1}{\sqrt{8}}\\
&\lesssim \left(2\delta \cdot \sqrt{\frac{t}{L}}\right)^L.
\end{align} Therefore,  \begin{equation}
W_1(\pi_1,\pi_0) \gtrsim \frac{\delta}{L} \textand V(P_{\pi_1},P_{\pi_0}) \lesssim \left(2\delta \cdot \sqrt{\frac{t}{L}}\right)^L.
\end{equation} Let the separation rate be $r \asymp \frac{\delta}{L}$ , we have that \begin{equation}
V(P_{\pi_1},P_{\pi_0}) \lesssim \left(2r\sqrt{tL}\right)^L.
\end{equation} Since we need that $V(P_{\pi_1}, P_{\pi_0}) \lesssim \frac{1}{n}$, we choose $L$ so that the separation rate can be as large as possible \begin{equation}
r \lesssim \sup_{1\leq L \leq t}\ \frac{1}{n^{1/L} \cdot \sqrt{Lt}}\ \asymp \frac{1}{\sqrt{t\cdot \log n }}\quad  \textif \log n \lesssim t \period
\end{equation} Consequently \begin{equation}
W_1(\pi_1,\pi_0) \gtrsim \frac{1}{\sqrt{t\log n}} \textand V(P_{\pi_1},P_{\pi_0}) \lesssim \frac{1}{n}\quad  \textif \log n \lesssim t \period
\end{equation} Finally, by \zcref{lemma:lecam_lb}, it holds that
\begin{equation}
\epsilon_*(n,t) \gtrsim \frac{1}{\sqrt{t\log n}} \quad  \textif \log n \lesssim t \period
\end{equation}

\underline{\textbf{Large $t$ regime}} 

For this regime, we construct the lower-bound manually rather than using the solution of an optimization problem. Let \begin{equation}\label{eq:tv_distance_large_t}
	\pi_0 = \frac{1}{2} \cdot \delta_0 + \frac{1}{2} \cdot  \delta_1  \textand \pi_1 = \left(\frac{1}{2}-\epsilon\right) \cdot \delta_0 + \left(\frac{1}{2}+\epsilon\right) \cdot  \delta_1.
\end{equation} It holds that $W_1(\pi_0,\pi_1)=\epsilon$. Since \begin{equation}
B_{j,t}(0) = \delta_{j}(0) \textand B_{j,t}(1) = \delta_{t}(1) \comma
\end{equation} the marginal measures of the data are \begin{align}
P_{\pi_0} =  \frac{1}{2} \cdot \delta_0 + \frac{1}{2} \cdot  \delta_t \textand  P_{\pi_1} = \left(\frac{1}{2}-\epsilon\right) \cdot \delta_0 + \left(\frac{1}{2}+\epsilon\right) \cdot  \delta_t
\end{align} and their $\chi^2$ distance is given by \begin{align}
\chi^2&\left(P_{\pi_0}^n,P^n_{\pi_1}\right) + 1 \\
&=\prod_{i=1}^n \chi^2\left(P_{\pi_0},P_{\pi_1}\right) + 1 \\
&=\prod_{i=1}^n \frac{1}{2} \cdot \left( (1-2\epsilon)^2 + (1+2\epsilon)^2 \right)\\
&\leq \prod_{i=1}^n \frac{1}{2} \left[ \exp\left(-4\epsilon\right) + \exp\left(4\epsilon\right)  \right]&&\text{since }1+x\leq e^x\\
&= \prod_{i=1}^n \cosh\left(4\epsilon\right)\\
&\leq \prod_{i=1}^n \exp\left(8\epsilon^2\right) &&\text{since } \cosh x \leq e^{x^2/2}\\
&= \exp\left(8n\epsilon^2\right)
\end{align} Therefore, it holds that \begin{equation}
V\left(P_{\pi_0}^n,P^n_{\pi_1}\right) \leq \sqrt{\chi^2\left(P_{\pi_0}^n,P^n_{\pi_1}\right)} \leq C_\alpha \for \epsilon = \left(\frac{\log(C_\alpha+1)}{8}\right)^{1/2}\cdot \frac{1}{n^{1/2}} \comma
\end{equation} and consequently, by \zcref{lemma:lecam_lb}, it follows that \begin{equation}
\epsilon_*(n,t) \gtrsim \frac{1}{\sqrt{n}} \period\end{equation}\end{proof}

\section{Reduction from random to fixed effects}\label{sec:reduction_random_to_fixed}

\zcref[S]{lemma:reduction} is a corollary of \zcref{lemma:reduction_ext}, since if the null hypothesis is a point mass $\pi_0=\delta_{p_0}$. It follows that $J(\pi_0)=0$, and the statement of \zcref{lemma:reduction} follows.

\begin{lemma}[Extension of \zcref{lemma:reduction}]\label{lemma:reduction_ext} Consider the hypotheses \eqref{eq:homogeneity_testing}. Under the null hypothesis, it follows that \begin{equation}
W_1(\tilde{\pi},\pi_0) \leq C_\delta \cdot \frac{J(\pi_0)}{\sqrt{n}} \quad \text{with probability }1-\delta.
\end{equation} Additionally, if \begin{equation}
\epsilon  \geq C_\delta \cdot  \left[\frac{J(\pi_0)}{\sqrt{n}} + \frac{1}{n}\right] \where C_\delta = \left[2\cdot\sqrt{3}\cdot(2+\delta^{-1/2})\right]^2,
\end{equation} under the alternative hypothesis, it holds that \begin{equation}
W_1(\tilde{\pi},\pi_0) \geq \epsilon \quad \text{with probability at least }1-\delta.
\end{equation}
\end{lemma}

\begin{proof}[Proof of \zcref{lemma:reduction_ext}]
By Chebyshev's inequality, it follows that  \begin{align}
\left|W_1(\tilde{\pi},\pi_0)-W_1(\pi,\pi_0)\right| &\leq W_1(\tilde{\pi},\pi)\\
&\leq E_\pi W_1(\tilde{\pi},\pi) + |W_1(\tilde{\pi},\pi)-EW_1(\tilde{\pi},\pi)|\\
&\leq E_{\pi}W_1(\tilde{\pi},\pi) + (1+\delta^{-1/2}) \sqrt{V_{\pi}W_1(\tilde{\pi},\pi)} \quad \wpa 1-\delta.
\end{align} By Theorem 3.2 of \citet{bobkovOnedimensionalEmpiricalMeasures2019},we have that $
E_\pi W_1(\tilde{\pi},\pi) \leq \frac{J(\pi)}{\sqrt{n}}$, and \begin{equation}
V_{\pi}W_1(\tilde{\pi},\pi) \leq E_\pi \left[W_1(\tilde{\pi},\pi)\right]^2 \leq \left(E_\pi W_1^2(\tilde{\pi},\pi)\right) \leq \left[\frac{J(\pi)}{\sqrt{n}}\right]^2.
\end{equation} Consequently, it holds with probability at least $1-\delta$ that \begin{align}
\left|W_1(\tilde{\pi},\pi_0)-W_1(\pi,\pi_0)\right| \leq (2+\delta^{-1/2})\cdot\frac{J(\pi)}{\sqrt{n}}.
\end{align} Thus, under the null hypothesis, $W_1(\pi,\pi_0)=0$, it follows that, with probability at least $1-\delta$, \begin{equation}
W_1(\tilde{\pi},\pi_0) \leq (2+\delta^{-1/2})\cdot\frac{J(\pi_0)}{\sqrt{n}}.
\end{equation} Under the alternative hypothesis, we have $W_1(\pi,\pi_0)\geq \epsilon$. Note that \begin{equation}
|\mu_{F_{\pi}(x)}-\mu_{F_{\pi_0}(x)}|=|F_{\pi}(x)-F_{\pi_0}| \cdot |1+F_{\pi}(x)+F_{\pi_0}|,
\end{equation} which  implies that  \begin{align}
\frac{J(\pi)}{\sqrt{n}} &\leq \frac{J(\pi_0)}{\sqrt{n}} + \sqrt{\frac{3}{n}}\cdot \int \left|F_{\pi}(x)-F_{\pi_0}\right|^{1/2}\\
&\leq   \frac{J(\pi_0)}{\sqrt{n}} + \sqrt{3} \cdot \sqrt{\frac{W_1(\pi,\pi_0)}{n}}
\end{align} where in the last inequality, we used Jensen's inequality. Thus, we have that \begin{equation}
\left|W_1(\tilde{\pi},\pi_0)-W_1(\pi,\pi_0)\right| \leq (2+\delta^{-1/2})\cdot \left[\frac{J(\pi_0)}{\sqrt{n}} + \sqrt{3} \cdot \sqrt{\frac{W_1(\pi,\pi_0)}{n}}\right],
\end{equation} with probability at least $1-\delta$. Consequently, $
W_1(\tilde{\pi},\pi_0) \geq \epsilon/2$ with probability at least $1-\delta$ under the following condition \begin{equation}
\frac{\epsilon}{2} \geq (2+\delta^{-1/2})\cdot \left[\frac{J(\pi_0)}{\sqrt{n}} + \sqrt{3} \cdot \sqrt{\frac{W_1(\pi,\pi_0)}{n}}\right],
\end{equation} which is implied by \begin{equation}
\epsilon \geq \left[2\cdot\sqrt{3}\cdot(2+\delta^{-1/2})\right]^2\cdot\left[\frac{J(\pi_0)}{\sqrt{n}}\ \lor\  \frac{1}{n}\right].
\end{equation}
\end{proof}

\section{Local critical separation for homogeneity testing with a reference effect}\label{appx:homonegenity_testing_known_null}

\zcref{lemma:RandomEffectsHomogeneityTesting} is the main result of this section. Its proof guides the reader through the analysis of the performance of the mean test \eqref{eq:1stmoment_test} (\zcref{sec:homogeneity_first_moment_test}) and debiased $\ell_2$ test \eqref{eq:2ndmoment_test} (\zcref{sec:homogeneity_chi2_test}), and the lower bounds on the local critical separation (\zcref{sec:lower_bound_random_effects} and \zcref{sec:homogeneity_lowerbounds}). 

\RandomEffectsHomogeneityTesting*

\begin{proof}[Proof of \zcref{lemma:RandomEffectsHomogeneityTesting}]
Henceforth, let $\pi_0=\delta_{p_0}$. For the upper-bound, \zcref{lemma:t1homogeneity} indicates that the mean test \eqref{eq:1stmoment_test} controls the type II error whenever \begin{equation}
\epsilon \geq C \cdot r_1 \where r_1= p_0 \lor \frac{1}{n}
\end{equation} while \zcref{lemma:t2homogeneity} indicates that the debiased $\ell_2$ test \eqref{eq:2ndmoment_test} controls the type II error whenever \begin{equation}
\epsilon \geq C \cdot r_2 \where r_2 = \frac{p_0^{1/2}}{t^{1/2}n^{1/4}} \lor \frac{1}{n^{1/2}}.
\end{equation} Consequently, consider the test that rejects whenever one of them rejects \begin{equation}
\psi^{\alpha} = \psi_1^{\alpha/2}\ \lor\ \psi_2^{\alpha/2}.
\end{equation} By union bound, $\psi^{\alpha}$ controls the type I error by $\alpha$. Furthermore, it controls type II error whenever $
\epsilon \geq C \cdot \left[r_1 \lor r_2\right]$, which implies that \begin{align}\label{eq:rates1}
\epsilon_*(n,t,\pi_0) \lesssim \begin{cases}
\frac{1}{n} &\textif \mu_{p_0} \lesssim \frac{t}{n^{3/2}}\\
\frac{\mu_{p_0}^{1/2}}{t^{1/2}n^{1/4}} &\textif \mu_{p_0} \gtrsim \frac{t}{n^{3/2}} \\
\end{cases}\qquad \textif t \gtrsim \sqrt{n}
\end{align} and \begin{align}\label{eq:rates2}
\epsilon_*(n,t,\pi_0) \lesssim \begin{cases}
\frac{1}{n} &\textif \mu_{p_0} \lesssim \frac{1}{n}\\
\mu_{p_0} &\textif \frac{1}{n}\lesssim \mu_{p_0} \lesssim \frac{1}{tn^{1/2}} \\
\frac{\mu_{p_0}^{1/2}}{t^{1/2}n^{1/4}} &\textif \frac{1}{tn^{1/2}} \lesssim  \mu_{p_0}\\
\end{cases}\qquad \textif t \lesssim \sqrt{n}
\end{align}

Lemmas \eqref{lemma:small_p0_lb}, \eqref{lemma:medium_p0_lb}, and \eqref{lemma:large_p0_lb} provide following lower-bound on the local critical separation
\begin{equation}
\epsilon_*(n,t,\pi_0) \gtrsim \gamma_*(n,t,\pi_0)\quad  \where  \gamma_*(n,t,\pi_0)=\begin{cases}
\frac{1}{tn} &\textif \mu_{p_0} \lesssim \frac{1}{tn}\\
\mu_{p_0} &\textif \frac{1}{tn}\lesssim \mu_{p_0} \lesssim \frac{1}{tn^{1/2}}\\
\frac{\mu_{p_0}^{1/2}}{t^{1/2}n^{1/4}} &\textif \frac{1}{tn^{1/2}} \lesssim  \mu_{p_0}\\
\end{cases}
\end{equation} Additionally, \zcref{lemma:RandomEffectsLowerBound} proves that $\epsilon_*$ cannot be below $1/n$. Thus, we have that  \begin{equation}
\epsilon_*(n,t,\pi_0) \gtrsim \frac{1}{n}\ \lor\ \gamma_*(n,t,\pi_0),
\end{equation} which matches up-to-constants the rates obtained in \eqref{eq:rates1} and \eqref{eq:rates2}. Hence, the statement of the lemma follows.\end{proof}

\subsection{Upper bounds on the local critical separation}\label{sec:homogeneity_upperbounds}

\subsubsection{Mean test}\label{sec:homogeneity_first_moment_test}

\begin{remark}\label{norm_ineq}
For $q \geq  p \geq 1$ and $\pi=k^{-1}\cdot \sum_{i=1}^k\delta_{p_i}$, it holds that \begin{align}
\int |p-p_0|^q\ d\pi(p) &= \frac{1}{k} \left(\sum_{i=1}^k |p_i-p_0|^q\right)\\
&\leq k^{q/p-1} \cdot \left(\frac{1}{k}\sum_{i=1}^k |p_i-p_0|^p\right)^{q/p}\\
&= k^{q/p-1} \cdot \left(\int |p-p_0|^p d\pi(p)\right)^{q/p}.
\end{align}
\end{remark}

\begin{lemma}\label{lemma:t1homogeneity}
For hypotheses \eqref{eq:homogeneity_testing}, the mean test \eqref{eq:1stmoment_test} controls type I error. Furthermore, there exists universal positive constant $C$ such that the test controls type II error by $\beta$ whenever \begin{equation}
\epsilon \geq  C \cdot \left[\mu_{p_0} \lor \frac{1}{n} \right] \where C = 4 \left[\frac{1}{\beta} \lor \frac{1}{\alpha} \right]^{1/4}.
\end{equation}\end{lemma}
\begin{proof}[Proof of \zcref{lemma:t1homogeneity}]
Let $d_l = \int (p-p_0)^l\ d\pi(p)$. \zcref[S]{sec:1stKravchuk} states that \begin{equation}
E_{P_\pi}[T_1] = d_1 \textand n \cdot V_{P_\pi}[T_1] \leq \frac{p_0}{t} + \frac{1}{t}d_1+ (1-\frac{1}{t})\cdot d_2 - d_1^2.
\end{equation}

\zcref[S]{mean_difference_dominates_std} indicates that we need to find $\epsilon$ such that the following condition is satisfied \begin{equation}\label{eq:t1homogeneity_condition}
\sqrt{\frac{V_{P_\pi}[T_1]}{\beta}} + \sqrt{\frac{V_{P_{\pi_0}}[T_1]}{\alpha}} \leq E_{P_\pi}[T_1]-E_{P_{\pi_0}}[T_1] \period
\end{equation} Note that \begin{align}
\sqrt{\frac{V_{P_\pi}[T_1]}{\beta}} + \sqrt{\frac{V_{P_{\pi_0}}[T_1]}{\alpha}} &\leq \frac{C}{n^{1/2}} \cdot \left[2\frac{p_0}{t} + \frac{1}{t}d_1+ (1-\frac{1}{t})\cdot d_2 - d_1^2\right]^{1/2}\comma
\end{align} where $C = \frac{2}{\sqrt{2}} \left[\frac{1}{\beta} \lor \frac{1}{\alpha} \right]^{1/2}$, is implied by \begin{equation}\label{eq:reject_condition}
d_1 \geq C^{1/2} \cdot \left[  2^{1/2}\cdot\frac{p_0^{1/2}}{n^{1/2}t^{1/2}} \lor \frac{1}{nt} \lor \frac{d_2^{1/2}}{n^{1/2}}\right] \period
\end{equation}

Recalling that $d_1 \geq \epsilon \geq 2p_0$, we obtain the following simple upper-bound  \begin{equation}
d_2 \leq \int |p-p_0| d\pi(p) \leq m_1(\pi) + p_0 = d_1 + 2p_0 \leq 2d_1 \period
\end{equation} Combining it with \eqref{eq:reject_condition}, we get that \eqref{eq:t1homogeneity_condition} is implied by \begin{equation}
\epsilon \geq (2C)^{1/2} \cdot \left[  \frac{p_0^{1/2}}{n^{1/2}t^{1/2}} \lor \frac{1}{n} \right] \lor 4p_0 \period
\end{equation} Thus \begin{equation}
\epsilon \geq  C \cdot \left[p_0 \lor \frac{1}{n} \right] \where C = 4 \left[\frac{1}{\beta} \lor \frac{1}{\alpha} \right]^{1/4}
\end{equation} sufices for \eqref{eq:t1homogeneity_condition} to hold.\end{proof}

\subsubsection{Debiased $\ell_2$ test}\label{sec:homogeneity_chi2_test}

\begin{lemma}\label{lemma:t2homogeneity}
For the homogeneity testing problem \eqref{eq:homogeneity_testing}, the debiased $\ell_2$ test \eqref{eq:2ndmoment_test} controls type I error. Furthermore, there exists universal positive constant $C$ such that the test controls type II error by $\beta$ whenever \begin{equation}
\epsilon \geq \frac{C}{\beta^{1/2}} \cdot \left[\frac{\mu_{p_0}^{1/2}}{t^{1/2}n^{1/4}} \lor \frac{1}{ n^{1/2}}\right].
\end{equation}\end{lemma}
\begin{proof}[Proof of \zcref{lemma:t2homogeneity}]
Let $d_l = \int (p-p_0)^l\ d\pi(p)$. \zcref[S]{sec:2ndKravchuk} states that \begin{align}
E_{P_\pi}[T_2] &= d_2\\
n \cdot V_{P_\pi}[T_2] &\lesssim \frac{p_0^2}{t^2} + \frac{p_0}{t^2} d_1 + \frac{1}{t^2} d_2 + \frac{p_0}{t} d_2 +  \frac{1}{t}d_3 + (1- \frac{1}{t})\cdot  d_4 - d_2^2
\end{align}

\zcref[S]{mean_difference_dominates_std} indicates that we need to find $\epsilon$ such that the following condition is satisfied \begin{equation}
\sqrt{\frac{V_{P_\pi}[T_2]}{\beta}} + \sqrt{\frac{V_{P_{\pi_0}}[T_2]}{\alpha}} \leq E_{P_\pi}[T_2]-E_{P_{\pi_0}}[T_2] \label{eq:cheby_cond}
\end{equation} Noting that \begin{align}
\sqrt{\frac{V_{P_\pi}[T_2]}{\beta}} &+ \sqrt{\frac{V_{P_{\pi_0}}[T_2]}{\alpha}}\\
&\leq \frac{C}{n^{1/2}} \cdot \left[2\cdot\frac{p_0^2}{t^2} + \frac{p_0}{t^2} d_1 + \frac{1}{t^2} d_2 + \frac{p_0}{t} d_2 +  \frac{1}{t}d_3 + (1- \frac{1}{t})\cdot  d_4 - d_2^2\right]^{1/2}
\end{align} where $ C = \frac{2}{\sqrt{2}} \left[\frac{1}{\beta} \lor \frac{1}{\alpha} \right]^{1/2}$ , it follows that \eqref{eq:cheby_cond} is implied by
\begin{align}
d_2 \gtrsim  \frac{p_0}{tn^{1/2}}\lor \frac{p_0^{1/2}d_1^{1/2}}{tn^{1/2}} \lor \frac{1}{t^2n} \lor \frac{p_0}{tn} \lor \frac{d_3^{1/2}}{t^{1/2}n^{1/2}} \lor \frac{d_4^{1/2}}{n^{1/2}}.
\end{align} By Jensen's inequality, it holds that $d_1 \leq d_2^{1/2}$. Therefore, \begin{align}
d_2 &\gtrsim  \frac{p_0}{tn^{1/2}}\lor \frac{p_0^{2/3}}{t^{4/3}n^{2/3}} \lor  \frac{1}{t^2n} \lor \frac{d_3^{1/2}}{t^{1/2}n^{1/2}} \lor \frac{d_4^{1/2}}{n^{1/2}}\\
&\asymp \frac{p_0}{tn^{1/2}} \lor  \frac{1}{t^2n} \lor \frac{d_3^{1/2}}{t^{1/2}n^{1/2}} \lor \frac{d_4^{1/2}}{n^{1/2}}
\end{align} since \begin{equation}
\frac{p_0^{2/3}}{t^{4/3}n^{2/3}} \leq \frac{p_0}{tn^{1/2}} \lor  \frac{1}{t^2n} \period
\end{equation}

If no assumptions on $\pi$ are made, using $d_3 \leq d_2$ and $d_4 \leq d_2$, we get that \eqref{eq:cheby_cond} holds whenever \begin{equation}
d_2 \gtrsim \frac{p_0}{tn^{1/2}}\lor  \frac{1}{t^2n} \lor \frac{1}{tn} \lor \frac{1}{n}.
\end{equation} Recalling that $d_2 \geq \epsilon^2$, we get that \eqref{eq:cheby_cond} is implied by \begin{equation}
\epsilon \gtrsim  \frac{p_0^{1/2}}{n^{1/4}t^{1/2}} \lor \frac{1}{n^{1/2}}.
\end{equation}\end{proof}


\subsection{Lower bound on the local critical separation for random effects}\label{sec:lower_bound_random_effects}

\begin{lemma}\label{conditional_lb}
Consider event $A$ such that $\pi_1^n(A) \leq \pi_0^n(A)$, and denote the resulting conditional distributions by $
P_{\pi_0}|A = \tilde{P}_{\pi_0}$ and $P_{\pi_1}|A  = \tilde{P}_{\pi_1}$. It follows that local minimax risk is lower-bounded by \begin{equation}
R_*(\epsilon,\pi_0) \geq \pi_1^n(A) - \pi_0^n(A) \cdot V(\tilde{P}_{\pi_0},\tilde{P}_{\pi_1})-\alpha.
\end{equation}
\end{lemma}
\begin{proof}[Proof of \zcref{conditional_lb}] First, recall that the local minimax risk can be lower-bound as follows \begin{align}
R_*(\epsilon,\pi_0) &= \inf_{\psi \in \Psi(\pi_0)}\sup_{\pi \st W_1(\pi,\pi_0)\geq \epsilon} P_\pi^n(\psi(X)=0)\\
&\geq \inf_{\psi \in \Psi(\pi_0)}\sup_{\pi : W_1(\pi,\pi_0)\geq \epsilon} P_{\pi_0}^n(\psi(X)=1)+P_\pi^n(\psi(X)=0)-\alpha \label{eq:last_eq}\end{align} where is the last inequality we used the fact that $\psi \in \Psi(\pi_0)$. For any mixing distributions $\pi_0$ and $\pi_1$, it follows that \begin{align}
&P_{\pi_0}^n(\psi(X) = 1) + P_{\pi}^n(\psi(X) = 0)\\
&\geq P_{\pi_0}^n(\psi(X) = 1|A) \cdot \pi_0^n(A) + P_{\pi}^n(\psi(X) = 0|A) \cdot \pi^n(A)\\
&= \pi^n(A) + \inf_{\psi} \left[P_{\pi_0}^n(\psi(X) = 0|A) \cdot \pi_0^n(A) - P_{\pi}^n(\psi(X) = 0|A) \cdot \pi^n(A)\right].
\end{align} Since $\pi^n(A) \leq \pi_0^n(A)$, it holds that \begin{align}
&P_{\pi_0}^n(\psi(X) = 1) + P_{\pi}^n(\psi(X) = 0)\\
&\geq  \pi^n(A) + \pi_0^n(A) \cdot \inf_{\psi} \left[P_{\pi_0}^n(\psi(X) = 0|A)  - P_{\pi}^n(\psi(X) = 0|A)\right]\\
&= \pi^n(A) - \pi_0^n(A) \cdot \sup_{\psi} \left[P_{\pi}^n(\psi(X) = 0|A)  - P_{\pi_0}^n(\psi(X) = 0|A)\right] \\
&\geq  \pi^n(A) - \pi_0^n(A) \cdot V(\tilde{P}_{\pi_0},\tilde{P}_\pi) \period \label{eq:last_eq2}
\end{align} Bounding \eqref{eq:last_eq} from below by \eqref{eq:last_eq2} proves the statement.\end{proof}

\RandomEffectsLowerBound*
\begin{proof}[Proof of \zcref{lemma:RandomEffectsLowerBound}]
Let $\pi = (1-C_\alpha/n)\cdot \delta_{p_0} + (C_\alpha/n) \cdot \delta_1$, and consider the event $A = \cup_{i=1}^n\{p_i=p_0\}$. Note that it always happens under the null mixing distribution $\pi_0=\delta_{p_0}$, and that it happens with constant probability under the mixing distribution $\pi$
\begin{equation}
(1-C_\alpha)^n=\pi^n(A) \leq \pi_0^n(A) = 1\period
\end{equation} Furthermore, conditioned on the event, the marginal distribution of the data is identical under both mixing distributions \begin{equation}
P_{\pi_0} | A = P_{\pi_0} \textand P_{\pi} | A = P_{\pi_0}\period
\end{equation} Consequently, by \zcref{conditional_lb}, the local minimax risk is lower-bounded by a constant \begin{equation}
R_*(\epsilon,\pi_0) \geq \pi^n(A) - \pi_0^n(A) \cdot V(P_{\pi_0},P_{\pi_0}) -\alpha = (1-C_\alpha)^n -\alpha > \beta \period
\end{equation} Since the distance between $\pi$ and $\pi_0$ satisfies $W_1(\pi_1,\pi_0)\geq \frac{C_\alpha}{2n}$, the statement of the lemma follows.
\end{proof}

\subsection{Lower bounds on the local critical separation for fixed effects}\label{sec:homogeneity_lowerbounds}

\subsubsection{Small $p_0$ regime}\label{sec:lowerbound_small_p}

\begin{lemma}\label{lemma:small_p0_lb}
Let $0 \leq p_0 \lesssim \frac{1}{tn}$ and $\epsilon \asymp \frac{1}{tn}$, and consider the following mixing distributions $\pi_0 = \delta_{p_0}$ and $\pi_1=\delta_{p_0+\epsilon}$. It follows that no test can distinguish them based on a sample of size $n$. Consequently, the critical separation is lower-bounded \begin{equation}
\epsilon_*(n,t,\pi_0) \geq W_1(\pi,\pi_0) \asymp \frac{1}{tn} \period
\end{equation}
\end{lemma} \begin{proof}[Proof of \zcref{lemma:small_p0_lb}] Let $\pi_0 = P_{\delta_0}$, then it follows that \begin{align}
V(P_{\pi_0},P_{\pi}) &\leq V(P_{\pi_0},P_{0}) + V(P_{\pi},P_{0}) \\
&= 1 - (1-p_0)^t + 1 - (1-p_0-\epsilon)^t
\end{align} Let $\xi_1 \in [1-p_0,1] \textand \xi_2 \in [1-p_0-\epsilon,1]$, by Taylor's expansion, let it follows that \begin{align}
V(P_{\pi_0},P_{\pi}) = p_0 t \xi_1^{t-1} + (p_0+\epsilon) t \xi_2^{t-1}
\leq 2 p_0 t  + \epsilon t \label{eq:prev_step}
\end{align} Therefore, \begin{align}
V(P_{\pi_0}^n,P_{\pi}^n) &\leq n \cdot V(P_{\pi_0},P_{\pi})\\
&\lesssim  p_0 t n+ \epsilon t n &&\text{by \eqref{eq:prev_step}}\\
&\lesssim 1 &&\since p_0 \lesssim \frac{1}{tn} \textand \epsilon \lesssim \frac{1}{tn},
\end{align} where the last constant can be made arbitrarily small for large enough $n$ or $t$. Thus, by \zcref{lemma:lecam_lb}, we have that $
\epsilon_*(n,t,\pi_0) \geq W_1(\pi,\pi_0) \asymp \epsilon \asymp \frac{1}{tn}$.\end{proof}

\subsubsection{Medium $p_0$ regime}\label{sec:lowerbound_medium_p}

\begin{lemma}\label{lemma:medium_p0_lb}
Let $0 \leq p_0 \lesssim \frac{1}{t\sqrt{n}}$ and consider the following mixing distributions $\pi_0 = \delta_{p_0}$ and $\pi_1 = \frac{1}{2} \cdot \left[\delta_{2p_0} + \delta_0 \right]\period$ It follows that no test can distinguish them based on a sample of size $n$. Consequently, \begin{equation}
\epsilon_*(n,t,\pi_0) \geq W_1(\pi,\pi_0) = p_0 \period
\end{equation}
\end{lemma}\begin{proof}[Proof of \zcref{lemma:medium_p0_lb}] The distribution share the first moment. If they shared the second moment, then they would be identical. Hence, we let them differ in their second moment. Therefore, assume $t\geq2$ so that the binomial density preserves the second-moment information. Finally, note that their separation is given by $p_0$ \begin{equation}
W_1(\pi_1,\pi_0) = 2(1-\frac{1}{2})p_0=p_0.
\end{equation} In the following, we show that  $\chi^2\left(P_{\pi_0},P_{\pi_1}\right) \lesssim \frac{1}{n}$. Consequently, by the tensorization property of the $\chi^2$ distance, the distance between the product marginal measures can be bounded by an arbitrary constant for $n$ large enough \begin{equation}
\chi^2\left(P^n_{\pi_0},P^n_{\pi_1}\right) = \left(1+\chi^2\left(P_{\pi_0},P_{\pi_1}\right)\right)^n \leq \chi^2\left(P_{\pi_0},P_{\pi_1}\right) \cdot n \lesssim 1\period
\end{equation} By \zcref{cor:critical_sep_lb_by_chi2}, it follows that $\epsilon_*(n,t,\pi_0) \gtrsim p_0$ for $p_0 \lesssim \frac{1}{\sqrt{n}t}$.

We finish the proof with a derivation that asserts that $\chi^2\left(P_{\pi_0}, P_{\pi_1}\right) \lesssim \frac{1}{n}$:
\begin{align}
\chi^2\left(P_{\pi_0},P_{\pi_1}\right)
&= \sum_{m=1}^t \binom{t}{m} \cdot \frac{\left[E_{\pi_1}[X-p_0]^m\right]^2}{(p_0(1-p_0))^m} &&\text{by \zcref{lemma:chi2bound}}\\
&=\sum_{m=2}^t \binom{t}{m} \cdot \frac{\left[E_{\pi_1}[X-p_0]^m\right]^2}{(p_0(1-p_0))^m}  &&\text{since they share the first moment}\\
&=\frac{1}{4} \cdot \sum_{m=2}^t \binom{t}{m} \cdot a^m &&\where a = \frac{p_0}{1-p_0}\\
&=\frac{1}{4} \cdot \left[(1+a)^t-(1+at)\right] &&(1+a)^t-(1+at)\geq 0\ \since t\geq2\\
&\leq\frac{1}{4} \cdot \left[e^{ta}-(1+at)\right]\\
&\leq \frac{1}{4} \cdot \frac{e^{ta}}{2} \cdot (ta)^2 \\
&\lesssim t^2p_0^2\\
&\lesssim \frac{1}{n} &&\since p_0 \lesssim \frac{1}{t\sqrt{n}}.
\end{align}
\end{proof}

\subsubsection{Large $p_0$ regime}\label{sec:lowerbound_large_p}

\begin{lemma}\label{lemma:large_p0_lb}
Let $\frac{1}{n^{1/2}t} \lesssim p_0 \lesssim 1$ and $\epsilon \asymp \frac{p_0^{1/2}}{t^{1/2}n^{1/4}}$, and consider the following mixing distributions $
\pi_0 = \delta_{p_0}$ and $\pi_1 = \frac{1}{2} \cdot \left[\delta_{p_0+\epsilon} + \delta_{p_0-\epsilon} \right]\period$ It follows that no test can distinguish them based on a sample of size $n$. Consequently, \begin{equation}
\epsilon_*(n,t,\pi_0) \geq W_1(\pi,\pi_0) = \epsilon \asymp \frac{p_0^{1/2}}{t^{1/2}n^{1/4}} \period
\end{equation}
\end{lemma}
\begin{proof}[Proof of \zcref{lemma:large_p0_lb}]
Without loss of generality, assume that $p_0 \leq \frac{1}{2}$, then it holds that $0 \lesssim p_0 - \epsilon \leq p_0 + \epsilon \lesssim 1 $. We proceed to bound the chi-squared distance between the marginal measures \begin{align}
&\chi^2\left(P_{\pi_0},P_{\pi_1}\right)\\
&= \sum_{m=1}^t \binom{t}{m} \cdot \frac{\left[E_{\pi_1}[X-p_0]^m\right]^2}{(p_0(1-p_0))^m} &&\text{by \zcref{lemma:chi2bound}}\\
&=\sum_{m=2}^t \binom{t}{m} \cdot \frac{\left[E_{\pi_1}[X-p_0]^m\right]^2}{(p_0(1-p_0))^m}  &&\text{since they share the first moment}\\
&=\sum_{m=2}^t \binom{t}{m} \cdot \frac{\left[\frac{\epsilon^m+(-\epsilon)^m}{2}\right]^2}{(p_0(1-p_0))^m} \\
&=-1 + \frac{1}{2} \left[1-\frac{\epsilon^2}{p_0(1-p_0)}\right]^t + \frac{1}{2} \left[1+\frac{\epsilon^2}{p_0(1-p_0)}\right]^t\\
&\leq -1 + e^{\frac{t^2\epsilon^4}{2p_0^2(1-p_0)^2}}.
\end{align} It follows that $\chi^2\left(P_{\pi_0}^n,P_{\pi_1}^n\right)\leq e^{\frac{nt^2\epsilon^4}{2p_0^2(1-p_0)^2}} \lesssim 1$ and, by \zcref{cor:critical_sep_lb_by_chi2}, it holds that \begin{equation}
\epsilon_*(n,t,\pi_0) \gtrsim \frac{p_0^{1/2}}{t^{1/2}n^{1/4}} \quad \for \frac{1}{n^{1/2}t} \lesssim p_0 \period
\end{equation}
\end{proof}

\section{Critical separation for homogeneity testing without a reference effect} We characterize the critical separation \eqref{eq:homogeneity_critical_separation} for homogeneity testing without a reference effect. \zcref{sec:proof_oracle_test} and \zcref{sec:proof_truncated_test} analyze the conservative test \eqref{eq:oracle_test} and the non-conservative test \eqref{eq:truncated_test}, both of which are optimal. We prove a matching lower bound for the critical separation in \zcref{sec:global_minimax_homogeneity_lowerbounds}.

\subsection{Upper bounds on the critical separation}

\subsubsection{A conservative test}\label{sec:proof_oracle_test}

\OracleTestRates*
\begin{proof}[Proof of \zcref{lemma:rates_oracle_test}]
The theorem follows from \zcref{lemma:rates_oracle_test_ub,lemma:global_minimax_rates_homogeneity}.
\end{proof}

\begin{lemma}
\label{lemma:rates_oracle_test_ub} For testing problem \eqref{eq:variance_testing}, the test \eqref{eq:oracle_test} controls type I error by $\alpha$. Furthermore, there exists a universal positive constant $C$ such that the type II error of the test is bounded by $\beta$ whenever \begin{equation}\label{eq:debiased_l2_rates}
\epsilon \geq C \cdot \begin{dcases}
\frac{1}{n^{1/2}} &\for \sqrt{n} \lesssim t\\
\frac{1}{t^{1/2}n^{1/4}} &\otherwise
\end{dcases}
\end{equation}
\end{lemma}
\begin{proof}[Proof of \zcref{lemma:rates_oracle_test_ub}] Henceforth, let $T(X) = \hat{V}(X)$ to simplify the notation. By \zcref{mean_difference_dominates_std}, it holds that the test controls type I and II errors if the following condition is satisfied \begin{equation}\label{eq:control_condition}
\sup_{\pi \in S} E_{P_\pi}[T] + \sqrt{\frac{V_{P_{\pi}}[T]}{\alpha}} \leq \inf_{\pi : W_1(\pi,S)\geq \epsilon} E_{P_\pi}[T] - \sqrt{\frac{V_{P_{\pi}}[T]}{\beta}}\period
\end{equation} 

\underline{\textbf{Bound under the null hypothesis}}

To upper-bound the LHS of \eqref{eq:control_condition}, we can use the fact that $\hat{V}(X)$ is an unbiased estimator and the upper-bound on its variance provided by \zcref{lemma:variance_ustat}: \begin{equation}\label{eq:null_bound}
\sup_{\pi_0 \in S} E_{P_{\pi_0}}[T] + \sqrt{\frac{V_{P_{\pi_0}}[T]}{\alpha}} = \sup_{p_0 \in [0,1]} \sqrt{\frac{\frac{\mu_{p_0}^2}{t(t-1)}\cdot\frac{n-1/t}{\binom{n}{2}}}{\alpha}} \leq \frac{1}{\sqrt{2\alpha}} \cdot \frac{1}{tn^{1/2}}\quad \for t\geq 2 \textand n\geq 2
\end{equation} 

\underline{\textbf{Bound under the alternative hypothesis}} 

To lower-bound the RHS of \eqref{eq:control_condition}, recall that for any mixing distribution $\pi$: $E_{P_\pi}[T] = V(\pi)$. By \zcref{lemma:variance_ustat}, it follows that  \begin{equation}\label{eq:expectation_dominates}
E_{P_\pi}[T]/2 \geq \sqrt{\frac{V_{P_{\pi_0}}[T]}{\beta}}
\end{equation}if \begin{equation}
V(\pi)\gtrsim  \frac{1}{n^{1/2}} \cdot \left(\frac{m_1(\pi)}{t}+ V(\pi) + V^{1/2}(\pi)\right),
\end{equation} which is implied by \begin{equation}\label{eq:cond_var}
V(\pi)\gtrsim \frac{m_1^{1/2}(\pi)}{tn^{1/2}} + \frac{1}{n}.
\end{equation} In particular, if $\pi$ satisfies $W_1(\pi,S)\geq \epsilon$, then it must satisfy that $V(\pi)\geq \epsilon^2$: \begin{equation}
V(\pi)=W_2^2(\pi,\delta_{m_1(\pi)})\geq W_1^2(\pi,\delta_{m_1(\pi)}) \geq W_1^2(\pi,S) \geq \epsilon^2.
\end{equation} Consequently, \eqref{eq:cond_var} is implied by \begin{equation}
\epsilon\gtrsim \frac{1}{t^{1/2}n^{1/4}} + \frac{1}{n^{1/2}} \quad \for\ \pi \st W_1(\pi,S)\geq \epsilon
\end{equation} where we used the fact that $m_1(\pi)\leq 1$ for any mixing distribution $\pi$. Thus, we can state the following lower-bound for the RHS of \eqref{eq:control_condition} \begin{equation}\label{eq:alternative_bound}
\inf_{\pi : W_1(\pi,S)\geq \epsilon} E_{P_\pi}[T] - \sqrt{\frac{V_{P_{\pi}}[T]}{\beta}} \gtrsim \epsilon^2 \quad\for \epsilon\gtrsim \frac{1}{t^{1/2}n^{1/4}} + \frac{1}{n^{1/2}} \period
\end{equation} Finally, using \eqref{eq:null_bound} and \eqref{eq:alternative_bound}, we have that \eqref{eq:control_condition} is satisfied if \begin{equation}
\epsilon\gtrsim \frac{1}{t^{1/2}n^{1/4}} + \frac{1}{n^{1/2}}\period\end{equation}\end{proof}

\begin{lemma}\label{lemma:variance_ustat} For the estimator $\hat{V}$, defined in \eqref{eq:ustat}, the following upper-bounds on its variance hold. Let $\pi$ be a point mass mixing distribution, $\pi=\delta_{p_0}$ with $p_0 \in [0,1]$, it follows that \begin{equation}
V_{P_{\pi}}[T] = \frac{\mu_{p_0}^2}{t(t-1)}\cdot\frac{n-1/t}{\binom{n}{2}} \qquad \forall t\geq 2
\end{equation}

For a mixing distribution $\pi$ such that its mean $m_1=E_{p\sim \pi}\left[p\right]$ satisfies $m_1 \leq \frac{1}{2}$, it holds that \begin{equation}
V_{P_{\pi}}[T] \lesssim \frac{1}{n} \cdot \left(\frac{m_1^2}{t^2}+ V^2(\pi) + V(\pi)\right) \qquad \forall t\geq 3
\end{equation}\end{lemma} \begin{proof}[Proof of \zcref{lemma:variance_ustat}] Henceforth, we let $T(X) = \hat{V}(X)$ to simplify the notation. Since $T$ is a U-statistic with a kernel of degree $2$, it follows by Theorem 3 of \cite[sec.~1.3]{leeUStatisticsTheoryPractice2019} that its variance is given by \begin{align}
V_{P_{\pi}}[T] &= \frac{2(n-2)}{\binom{n}{2}} \cdot \sigma_1^2 + \frac{1}{\binom{n}{2}} \cdot \sigma_2^2  \\
\where \sigma_1^2 &= V[E[\frac{h(X_1,X_2)}{2}|X_1=x_1]]\\
&=\frac{1}{4}\cdot  V\left[\frac{X_1}{2}+m_2(\pi)-2\frac{X_1}{t}m_1(\pi)\right]\\
&=\frac{1}{4}\cdot V\left[ \frac{\binom{X_1}{2}}{\binom{t}{2}}-2\frac{X_1}{t}m_1(\pi)\right]\\
\textand \sigma_2^2 &= \frac{1}{4}\cdot V[h(X_1,X_2)]
\end{align}

\underline{\textbf{Variance for $\pi \in S$}} 

Let $\pi \in S$, there exists $p_0 \in [0,1]$ such that $\pi = \delta_{p_0}$. Note that \begin{equation}
\sigma_1^2 = \frac{1}{2} \cdot \frac{\mu^2_{p_0}}{(t-1)t} \textand
\sigma_2^2 = \frac{\mu^2_{p_0}}{(t-1)t} \cdot (2-1/t)
\end{equation} Thus, we have that \begin{equation}\label{eq:u_stat_var_null}
V_{P_{\pi}}[T] = \frac{\mu_{p_0}^2}{t(t-1)}\cdot\frac{n-1/t}{\binom{n}{2}}
\end{equation}

\underline{\textbf{Variance for general mixing distribution}} 

We use the fact that by Theorem 4 of \citet{leeUStatisticsTheoryPractice2019}, it holds that \begin{equation}
\sigma_1^2 \leq \frac{\sigma_2^2}{2}.
\end{equation} Consequently, \begin{equation}\label{eq:variance_2nd_order}
V_{P_{\pi}}[T] \leq \left(\frac{(n-2)}{\binom{n}{2}} + \frac{1}{\binom{n}{2}}\right) \cdot \sigma_2^2 = \frac{2}{n}\cdot \sigma_2^2
\end{equation} By the law of total variance, we have that \begin{align}
\sigma_2^2 = V_{P_{\pi}}[h(X_1,X_2)] &= A + B\\
\where A &= E_{p,q\iid\pi}V_{X_1|p\sim\Bin(t,p),X_2|q\sim\Bin(t,p)}\left[h(X_1,X_2)\right]\\
B &= V_{p,q\iid\pi}E_{X_1|p\sim\Bin(t,p),X_2|q\sim\Bin(t,p)}\left[h(X_1,X_2)\right]=V_{p,q\iid\pi}[(p-q)^2]
\end{align} Henceforth, we use $m_l$ to denote $m_l(\pi)=E_{p\sim\pi}\left[p^l\right]$ in order to simplify the notation. Term $A$ reduces to \begin{align}
A = &\frac{4 (m_2 - 2 m_3 + m_4)}{t-1}\\
&+ \frac{4 (m_1 - m_2)^2}{t^2}\\
&- \frac{4 (m_2 + 2 m_1 m_2 + 2 m_2^2 - 4 (1 + m_1) m_3 + 3 m_4)}{t}
\end{align} Let $d_l = \int (p-m_1(\pi))^l\ d\pi(p)$, then it can be shown that \begin{align}
m_2 &= d_2 + m_1^2\\
m_3 &= d_3 + 3m_1m_2 - 2m_1^3\\
m_4 &= d_4 + 4 m_1m_3 + 3m_1^4 - 6 m_1^2m_2
\end{align} We get that \begin{align}
A&= 4\frac{m_1^2}{t^2}\frac{2t-1}{t-1}+4\frac{d_2}{(t-1)t}+8\frac{d_3}{t}\frac{t-2}{t-1}(1-2m_1)\\
&+4\frac{m_1^4}{t^2}\frac{2t-1}{t-1}+8\frac{m_1d_2}{(t-1)t^2}(1+2(t-3)t)\\
&-4\frac{d_4}{(t-1)t}(2t-3)-4\frac{d_2^2}{t^2}(2t-1)\\
&-8\frac{m_1^2d_2}{(t-1)t^2}(1+2(t-3)t)-8\frac{m_1^3}{t^2(t-1)}(2t-1)\\
&\lesssim \frac{m_1^2}{t^2}+\frac{d_2}{t^2}+\frac{d_3}{t}+\frac{m_1^4}{t^2}+\frac{m_1d_2}{t} &&\since m_1 \leq \frac{1}{2} \textand t \geq 3\\
&\lesssim \frac{m_1^2}{t^2}+\frac{d_2}{t^2}+\frac{d_3}{t}+\frac{m_1d_2}{t}&&\since m_1 \leq 1 \label{eq:A_equiv}.
\end{align} Regarding $B$, we have that \begin{align}
B &= E_{p,q\iid\pi}(p-q)^4-(2V_1[\pi])^2\\
&= 2m_4(\pi)-8m_1(\pi)m_3(\pi)+6m_2^2(\pi) -(2d_2)^2\\
&= 2 (3 d_2^2 + d_4)-(2d_2)^2\\
&= 2(d_2^2 + d_4) \label{eq:B_equiv}.
\end{align} Thus, plugging \eqref{eq:A_equiv} and \eqref{eq:B_equiv} into \eqref{eq:variance_2nd_order}, we get \begin{align}
V_{P_{\pi}}[T] &\lesssim \frac{1}{n} \cdot \left(\frac{m_1^2}{t^2}+\frac{d_2}{t^2}+\frac{d_3}{t}+\frac{m_1\cdot d_2}{t} + d_2^2 + d_4\right)\\
&\lesssim \frac{1}{n} \cdot \left(\frac{m_1^2}{t^2}+\frac{V(\pi)}{t^2}+\frac{d_2}{t}+\frac{m_1\cdot V(\pi)}{t} + V^2(\pi) + V(\pi)\right)\\
&\lesssim \frac{1}{n} \cdot \left(\frac{m_1^2}{t^2} + V^2(\pi) + V(\pi)\right)
\end{align} where in the penultimate inequality, we use the facts that $V(\pi)=V_{p\sim \pi}\left[p\right]=d_2$, $d_3 \leq d_2$, and $d_2 \leq d_4$.
\end{proof}

\subsubsection{Non-conservative test}\label{sec:proof_truncated_test}

\TruncatedTestOtimality*
\begin{proof}[Proof of \zcref{lemma:rates_truncated_test}]
The theorem follows from \zcref{lemma:rates_truncated_test_ub,lemma:global_minimax_rates_homogeneity}.
\end{proof}

\begin{lemma}\label{lemma:rates_truncated_test_ub}
For testing problem \eqref{eq:homogeneity_testing_unknown}, the debiased Cochran's $\chi^2$ test \eqref{eq:truncated_test} controls type I error by $\alpha$. Furthermore, for $\gamma \geq c/n$ where $c$ is a universal positive constant, there exists a universal positive constant $C$ such that the type II error is bounded by $\beta$ whenever \eqref{eq:debiased_l2_rates} is satisfied.
\end{lemma}
\begin{proof}[Proof of \zcref{lemma:rates_truncated_test_ub}] We consider the case where $t \geq 2$ and  $m_1(\pi)<\frac{1}{2}$. If the second assumption does not follow, replace $Y_i$ by $\tilde{Y}_i = t - Y_i$ and proceed, and do the same with $X_i$. Henceforth, let \begin{align} R(X,Y) = \frac{T(X)}{\hat{\eta}(Y)} &\comma
T(X) = \hat{V}(X) \comma\\
\hat{\eta}(Y) = \max(\hat{\mu}(Y),\gamma) &\textand \eta(\pi) = \max(\mu_{m_1(\pi)},\gamma)
\end{align} to simplify the notation. By \zcref{mean_difference_dominates_std}, it holds that the test controls type I and II errors if the following condition is satisfied \begin{equation}\label{eq:generalized_power_condition}
\sup_{\pi \in S} E_{P_\pi}[R(X,Y)] + \sqrt{\frac{V_{P_{\pi}}[R(X,Y)]}{\alpha}} \leq \inf_{\pi : W_1(\pi,S)\geq \epsilon} E_{P_\pi}[R(X,Y)] - \sqrt{\frac{V_{P_{\pi}}[R(X,Y)]}{\beta}}.
\end{equation} Recall that $T(X)$ is independent from  $\hat{\mu}(Y)$. Therefore, \begin{equation}\label{eq:independence}
E_{P_\pi}[T(X,Y)]=E_{P_\pi}\left[T(X)\right]\cdot E_{P_{\pi}}\left[\hat{\eta}^{-1}(Y)\right].
\end{equation}

\underline{\textbf{Bound under null hypothesis}} 

To bound the LHS of \eqref{eq:generalized_power_condition} note that for $\pi_0 \in  S$, it holds that $E_{P_{\pi_0}}[R(X,Y)]=0$ since $E_{P_{\pi_0}}\left[T(X)\right]=0$ and $\hat{\eta}^{-1}(Y) \geq \gamma^{-1}$ with probability one. Furthermore, by \zcref{lemma:eta_bounds} it holds that \begin{equation}
E_{P_\pi}\left[\hat{\eta}^{-2}(Y)\right] \asymp \eta^{-2} \qquad \for \gamma \gtrsim n^{-1},
\end{equation} together with \eqref{eq:independence} it implies that \begin{equation}
V_{P_{\pi_0}}\left[R(X,Y)\right]=E_{P_{\pi_0}}\left[T^2(X)\right]\cdot E_{P_\pi}\left[\hat{\eta}^{-2}(Y)\right]\asymp  \frac{E_{P_{\pi_0}}\left[T^2(X)\right]}{\eta^{2}(\pi)}.
\end{equation} Let $\pi_0 = \delta_{p_0}$ where $p_0 \in [0,1]$, then using \zcref{lemma:variance_ustat}, we get the following upper-bound \begin{equation}\label{eq:null_bound_2}
V_{P_{\pi_0}}\left[R(X,Y)\right]\asymp  \frac{1}{t(t-1)}\cdot\frac{n-1/t}{\binom{n}{2}} \cdot \left(\frac{\mu_{p_0}}{\max\left(\mu_{p_0},\gamma\right)} \right)^2 \lesssim \frac{1}{t^2n}.
\end{equation} Consequently, we can upper-bound the LHS of \eqref{eq:generalized_power_condition} by \begin{equation}
\sup_{\pi_0 \in S} E_{P_{\pi_0}}[R(X,Y)] + \sqrt{\frac{V_{P_{\pi_0}}[R(X,Y)]}{\alpha}} \lesssim \frac{1}{tn^{1/2}} \qquad \for t\geq 2 \textand
 n \geq 2.\end{equation}
 
\underline{\textbf{Bound under alternative hypothesis}} 

To lower-bound the RHS of \eqref{eq:generalized_power_condition}, consider $\pi \st W_1(\pi,S)\geq \epsilon$. By \zcref{lemma:eta_bounds} \begin{equation}
E_{P_\pi}[R(X,Y)] \geq C_0^{-1} \cdot \frac{E_{P_\pi}[T(X)]}{\eta(\pi)} \textand V_{P_{\pi}}\left[R(X,Y)\right] \leq \frac{C_2+C_1^2}{\eta^2(\pi)}\cdot V_{P_\pi}\left[T(X)\right]
\end{equation} Since $\eta(\pi) >0$, it holds that $E_{P_\pi}[R(X,Y)] \geq \sqrt{V_{P_{\pi}}\left[R(X,Y)\right]/\beta}$ is implied by \begin{equation}
E_{P_\pi}[T(X)] \gtrsim  \sqrt{V_{P_{\pi}}\left[T(X)\right]}
\end{equation} which matches up-to-constants \eqref{eq:expectation_dominates}, following the argument thereafter, we have that \begin{equation}\label{eq:alternative_bound_2}
\inf_{\pi : W_1(\pi,S)\geq \epsilon} E_{P_\pi}[T] - \sqrt{\frac{V_{P_{\pi}}[T]}{\beta}} \gtrsim \epsilon^2 \for \epsilon\gtrsim \frac{1}{t^{1/2}n^{1/4}} + \frac{1}{n^{1/2}}
\end{equation} Finally, using \eqref{eq:null_bound_2} and \eqref{eq:alternative_bound_2}, we have that \eqref{eq:generalized_power_condition} is satisfied if \begin{equation}
\epsilon\gtrsim \frac{1}{t^{1/2}n^{1/4}} + \frac{1}{n^{1/2}}  \period\end{equation}\end{proof}

\begin{lemma}\label{lemma:eta_bounds} For any distribution $\pi$ supported on $[0,1]$. Let \begin{equation}
\eta(\pi) = \max(\mu_{m_1(\pi)},\gamma) \textand \hat{\eta}(Y) = \max(\hat{\mu}(Y),\gamma).
\end{equation} There exists universal positive constants $c_0$, $C_0$, $C_1$ and $C_2$ such that for $\gamma > c_0/n$ it holds that \begin{align}
E[\hat{\eta}(Y)] &\leq C_0 \cdot \eta, \label{eq:claim_1}\\
C^{-1}_0 \cdot \eta^{-1}  \leq E[\hat{\eta}^{-1}(Y)] &\leq C_1 \cdot \eta^{-1} \textand\label{eq:claim_2}\\
C^{-1}_0 \cdot \eta^{-2}  \leq  E[\hat{\eta}^{-2}(Y)] &\leq C_2 \cdot \eta^{-2}.\label{eq:claim_3}
\end{align} 
\end{lemma}\begin{proof}[Proof of \zcref{lemma:eta_bounds}]

\underline{\textbf{Expectation and variance of $\hat{\mu}$}} 

We consider the case where $t \geq 2$ and  $m_1(\pi)<\frac{1}{2}$. If the second assumption does not follow, replace $Y_i$ by $\tilde{Y}_i = t - Y_i$ and proceed.

It can be checked that $\hat{\mu}$ is an unbiased estimator of $\mu_{m_1(\pi)}$
\begin{equation}\label{eq:mu_hat_mean}
E_{P_\pi}[\hat{\mu}] = E_{P_\pi}[h(X_1,X_2)]=2\cdot \left(m_1(\pi)-m_1^2(\pi)\right)= \mu_{m_1(\pi)}
\end{equation} For the variance, recall that $\hat{\mu}$ is a U-statistic. Therefore, analogously to \eqref{eq:variance_2nd_order}, it follows that \begin{equation}
V_{P_{\pi}}[\hat{\mu}] \leq \frac{1}{2n}\cdot V_{X,Y\iid P_\pi}\left[\tilde{h}(X,Y)\right].
\end{equation} We split the analysis of the variance into two terms \begin{align}
V_{P_{\pi}}[h(X_1,X_2)] &= A + B\\
\where A &= E_{p,q\iid\pi}V_{X_1|p\sim\Bin(t,p),X_2|q\sim\Bin(t,p)}\left[\tilde{h}(X_1,X_2)\right]\\
B &= V_{p,q\iid\pi}E_{X_1|p\sim\Bin(t,p),X_2|q\sim\Bin(t,p)}\left[\tilde{h}(X_1,X_2)\right].
\end{align} By direct computation, the first term can be bounded by \begin{equation}
A = \frac{2}{t^2}\cdot (m_1(\pi)-m_2(\pi)) \cdot \left(t -2\cdot (m_1(\pi)-m_2(\pi))\cdot (2t-1)  \right) \leq \frac{2}{t}\cdot \mu_{m_1(\pi)}.
\end{equation} For the second term, we have that \begin{equation}
B = V_{p,q\iid\pi}\left[\left(p+q-2\cdot p\cdot q\right)^2\right] \leq 8 m_1(\pi) \leq 16 \mu_{m_1(\pi)},
\end{equation} where we used the fact that $m_2(\pi)\leq m_1(\pi) \leq 2\mu_{m_1(\pi)}$ since $m_1(\pi) \leq 1/2$. Consequently, \begin{equation}\label{eq:mu_hat_var}
V_{P_{\pi}}[h(X_1,X_2)] \leq c' \cdot \mu_{m_1(\pi)} \textand V_{P_{\pi}}[\hat{\mu}] \leq c' \cdot \frac{\mu_{m_1(\pi)}}{n}.
\end{equation}

\underline{\textbf{Bernstein bounds for $\hat{\mu}$}} 

Noting that $\tilde{h}$ is bounded, by \citet{hoeffdingProbabilityInequalitiesSums1963}, see also equation 1.4. of \citet{arconesBernsteintypeInequalityUstatistics1995a}, the following Bernstein's inequality for the U-statistics $\hat{U}$ holds by \eqref{eq:mu_hat_mean} and \eqref{eq:mu_hat_var}: \begin{equation}\label{eq:mu_hat_bernstein}
P_\pi\left(\hat{\mu}-\mu_{m_1(\pi)} > \delta \right) \leq \exp\left(- c' \cdot \frac{\delta^2 \cdot n}{\mu_{m_1(\pi)} +\delta}\right),
\end{equation} where $c'$ is a positive constant. Thus, for any $c>0$, it follows that
\begin{align}
P_\pi\left(\hat{\eta} < \frac{\eta}{b} \right) &= P_\pi\left(\hat{\mu} \lor \gamma < \frac{\eta}{b} \right)\\
&\leq P_\pi\left(\hat{\mu}  < \frac{\eta}{b} \right) \quad \for \gamma < \frac{\eta}{b} \text{ otherwise the probability is zero}\\
&= P_\pi\left(\hat{\mu}  < \frac{\mu_{m_1(\pi)}}{b} \right) \quad \since\gamma < \frac{\eta}{b} \text{ if and only if } \mu_{m_1(\pi)} > \gamma \cdot b \for b \geq 1  \\
&\leq P_\pi\left(|\hat{\mu}-\mu_{m_1(\pi)}| \geq \mu_{m_1(\pi)}\cdot (1-1/b) \right)\\
&\leq 2\exp\left(- c' \cdot \mu_{m_1(\pi)} \cdot n \cdot \frac{2}{3} \right) \quad \for b\geq 2 \text{ due to } \eqref{eq:mu_hat_bernstein}\\
&\leq 2\exp\left(- b \cdot c \right) \quad \for \gamma >   \frac{2}{3}\cdot \frac{c}{c'}\cdot\frac{1}{n}. \label{eq:eta_hat_ub1}
\end{align}

Analogously,
\begin{align}
P_\pi\left(\hat{\eta} > b\eta \right) &= P_\pi\left(\hat{\mu} \lor \gamma > b\eta \right)\\
&\leq P_\pi\left(\hat{\eta} > b\eta \right)\\
&= P_\pi\left(\hat{\mu} > \gamma \lor b\eta \right) + P_\pi\left(\gamma  > \hat{\mu} \lor b\eta \right)\\
&\leq P_\pi\left(\hat{\mu} > b\eta \right) &&\for b > 1 \since \eta \geq \gamma\\
&\leq P_\pi\left(\hat{\mu}-\mu_{m_1(\pi)} > c_\gamma \right),
\end{align} where \begin{equation}
c_\gamma = b\eta-\mu_{m_1(\pi)} \geq (b-1) \cdot \begin{cases}
\mu_{m_1(\pi)} &\textif \mu_{m_1(\pi)} \geq \gamma\\
\gamma &\textif \mu_{m_1(\pi)} < \gamma
\end{cases} \geq (b-1)\cdot \gamma
\end{equation} If $\mu_{m_1(\pi)} \geq \gamma$, \begin{align}
P_\pi\left(\hat{\eta} > b\eta \right) &\leq P_\pi\left(\hat{\mu}-\mu_{m_1(\pi)} > (b-1)\mu_{m_1(\pi)} \right) \\
&\leq \exp\left(- c' \cdot \frac{(b-1)^2\cdot \mu_{m_1(\pi)} \cdot n}{b}\right)\\
&\leq \exp\left(- \frac{c'}{2} \cdot b \cdot \mu_{m_1(\pi)} \cdot n\right) &&\for b \geq 2 \text{ due to }\eqref{eq:mu_hat_bernstein}\\
&\leq \exp(- b \cdot c) &&\for \gamma > \frac{2}{c'}\cdot \frac{c}{n}.
\end{align} Otherwise, if $\mu_{m_1(\pi)} < \gamma$, it follows that \begin{align}
P_\pi\left(\hat{\eta} > b\eta \right) &\leq P_\pi\left(\hat{\mu}-\mu_{m_1(\pi)} > (b-1)\gamma \right)\\
&\leq \exp\left(- \frac{c'}{2} \cdot \gamma \cdot b \cdot n\right) &&\for b \geq 2 \text{ due to } \eqref{eq:mu_hat_bernstein}\\
&\leq \exp(- b \cdot c) &&\for \gamma > \frac{2}{c'}\cdot \frac{c}{n}.
\end{align} Thus, it always holds that \begin{equation}
P_\pi\left(\hat{\eta} > b\eta \right) \leq \exp(- b \cdot c) \for \gamma > \frac{2}{c'}\cdot \frac{c}{n}. \label{eq:eta_hat_ub2}
\end{equation}

\underline{\textbf{Bounds of  the expectation of $\hat{\eta}$}} 

Let $c_0$ be any constant such that $
c_0 \geq \frac{2}{c'}\cdot c \lor \frac{2}{3}\cdot \frac{c}{c'}$. \zcref{eq:claim_2} follows by the following argument \begin{align}
E_{P_\pi}\left[\hat{\eta}^{-1}(Y)\right] &= \eta^{-1} \cdot \int_{0}^\infty P_\pi\left(\hat{\eta} < \frac{\eta}{b} \right) db\\
&\leq \eta^{-1} \cdot \left(2 + \int_{2}^\infty P_\pi\left(\hat{\eta} < \frac{\eta}{b} \right) db\right)\\
&\leq \eta^{-1} \cdot C_1  \where C_1 = 2 \left(1 + \frac{\exp(-2 c)}{c} \right).
\end{align} where the last step holds for $\gamma > \frac{c_0}{n}$ due to \eqref{eq:eta_hat_ub1}.

The upper-bound of \eqref{eq:claim_3} follows by
\begin{align}
E_{P_\pi}\left[\hat{\eta}^{-2}(Y)\right] &= 2\eta^{-2} \cdot \int_{0}^\infty P_\pi\left(\hat{\eta}^2 < \frac{\eta^2}{b^2} \right)\cdot b\ db\\
&= 2\eta^{-2} \cdot \int_{0}^\infty P_\pi\left(\hat{\eta} < \frac{\eta}{b} \right)\cdot b\ db\\
&\leq 2\eta^{-2} \cdot \left(2 + \int_{2}^\infty P_\pi\left(\hat{\eta} < \frac{\eta}{b} \right)\cdot b\ db\right)\\
&\leq \eta^{-2} \cdot C_2 \where C_2=4\left(1 + \frac{(2 c + 1)}{c^2}\cdot \exp(-2\cdot c) \right) 
\end{align} where the last step holds for $\gamma > \frac{c_0}{n}$ due to \eqref{eq:eta_hat_ub1}.

The upper-bound of \eqref{eq:claim_1} follows by
\begin{align}
E_{P_\pi}\left[\hat{\eta}(Y)\right] &= \eta \cdot \int_{0}^\infty P_\pi(\hat{\eta} > b\eta ) db\\
&\leq \eta \cdot \left(2+\int_{2}^\infty P_\pi(\hat{\eta} > b\eta ) db \right)\\
&\leq \eta \cdot C_0 \where C_0 = \left(2+ \frac{\exp(-2 c)}{c}\right).
\end{align} where the last step holds for $\gamma > \frac{c_0}{n}$ due to \eqref{eq:eta_hat_ub2}.

Finally, by Jensen's inequality, the lower-bounds of \eqref{eq:claim_2} and \eqref{eq:claim_3} follow
\begin{align}
E_{P_\pi}\left[\hat{\eta}^{-1}(Y)\right] &\geq \left(E_{P_\pi}\left[\hat{\eta}(Y)\right]\right)^{-1} \geq C_0^{-1} \cdot \eta^{-1} \\
E_{P_\pi}\left[\hat{\eta}^{-2}(Y)\right] &\geq \left(E_{P_\pi}\left[\hat{\eta}(Y)\right]\right)^{-2} \geq C_0^{-1} \cdot \eta^{-2} \period\end{align}\end{proof}

\subsection{Lower bounds on the critical separation}\label{sec:global_minimax_homogeneity_lowerbounds}

\begin{lemma}\label{lemma:global_minimax_rates_homogeneity} The critical separation \eqref{eq:homogeneity_critical_separation} is lower-bounded by \begin{equation}
\epsilon_*(n,t) \gtrsim \max\left(\ n^{-1/2} ,\ t^{-1/2}n^{-1/4}\right).
\end{equation}\end{lemma}
\begin{proof}[Proof of \zcref{lemma:global_minimax_rates_homogeneity}]

\underline{\textbf{For $t \lesssim \sqrt{n}$}}, we can rely on the lower-bound construction for fixed effects. Let $\pi_0 = \delta_{0.5}$, and use the lower-bound construction in \zcref{lemma:large_p0_lb}. The claim for this case follows.

\underline{\textbf{For $t \gtrsim \sqrt{n}$}}, we do a similar lower-bound construction to \eqref{eq:tv_distance_large_t}. Let \begin{equation}
\Gamma_0 = \frac{1}{2} \cdot \delta_0 + \frac{1}{2} \cdot  \delta_1  \textand \pi_1 = \left(\frac{1}{2}-\gamma\right) \cdot \delta_0 + \left(\frac{1}{2}+\gamma\right) \cdot  \delta_1
\end{equation} We compute their $\chi^2$ distance by the Ingster–Suslina $\chi^2$-method \citep{ingsterNonparametricGoodnessofFitTesting2003}, it follows that
\begin{align}
\chi^2\left(E_{p\sim \Gamma_0}\left[P_{\delta_p}^n\right]\ ,\ P_{\pi_1}^n\right) &=  E_{p,q\sim \Gamma_0}\left[\ \left(\sum_{j=0}^t \frac{P_{\delta_p}(j)\cdot P_{\delta_q}(j)}{P_{\pi_1}(j)}\right)^n  \ \right]
\end{align} Noting that $
P_{\delta_0} = \delta_0 \textand P_{\delta_1} = \delta_t$, the distance simplifies to \begin{align}
\chi^2\left(E_{p\sim \Gamma_0}\left[P_{\delta_p}^n\right]\ ,\ P_{\pi_1}^n\right) &=  \frac{1}{4}\cdot \left(\frac{1}{P_{\pi_1}(0)}\right)^n +\frac{1}{4}\cdot \left(\frac{1}{P_{\pi_1}(t)}\right)^n.
\end{align} Noting that $
P_{\pi_1} = \left(\frac{1}{2}-\gamma\right) \cdot \delta_0 + \left(\frac{1}{2}+\gamma\right) \cdot  \delta_t$, we get that \begin{align}
\chi^2\left(E_{p\sim \Gamma_0}\left[P_{\delta_p}^n\right]\ ,\ P_{\pi_1}^n\right) &=  \frac{1}{4\cdot 2^n}\cdot\left[ \left(1-2\gamma\right)^n +\left(1+2\gamma\right)^n\right]\\
&\leq \frac{1}{4\cdot 2^n}\cdot\left[ \exp\left\{-2n\gamma\right\} +\exp\left\{2n\gamma\right\}\right]\\
&= \frac{\cosh{2n\gamma}}{2^{n+1}}\\
&\leq \frac{\exp\{2n^2\gamma^2\}}{2^{n+1}}  \period
\end{align} Thus, we have that if we choose $\gamma$ such that \begin{equation}\label{eq:gamma_choice}
\gamma = \sqrt{\frac{n+1}{2n^2} \cdot \log 2+ \frac{\log C^2_\alpha}{2n^2}}\comma
\end{equation} we control the $\chi^2$ distance between the distributions by $C_\alpha^2$: \begin{equation}
\chi^2\left(E_{p\sim \Gamma_0}\left[P_{\delta_p}^n\right]\ ,\ P_{\pi_1}^n\right) \leq C^2_\alpha \period
\end{equation} Consequently, by \zcref{lemma:general_lb}, it follows that $
R_*(\epsilon) > \beta$ for $\epsilon \geq W_1(\pi_1,S)$. Note that for $\gamma<\frac{1}{4}$, it follows that: \begin{equation}
W_1(\pi_1,S) = \min_{p_0\in[0,1]}W_1(\pi,\delta_{p_0})=\min_{p_0\in[0,1]} \frac{1}{2}+\gamma-2\gamma p_0 = \frac{1}{2}-\gamma \geq \frac{\gamma}{2}\period
\end{equation} Finally, by \eqref{eq:gamma_choice}, $\gamma < 1/4$ is satisfied for $n$ large enough. Consequently, we have that $
R_*(\epsilon) > \beta$ for $\epsilon \gtrsim \frac{1}{\sqrt{n}}$.
\end{proof}

\section{Code to reproduce simulations and applications}

The code to reproduce the simulations and experiments can be accessed at \begin{center}
\url{https://github.com/lkania/Testing-Random-Effects}
\end{center} 

\section{Simulations}\label{sec:simulations}

\subsection{Homogeneity testing with a reference effect}\label{sec:homogeneity_testing_simulations}

Due to symmetry, we consider only $\pi_0 = \delta_{p_0}$ for $0 \leq p_0 \leq 1/2$. We assess the performance of various tests under the distributions used to derive the lower bounds in \zcref{lemma:RandomEffectsHomogeneityTesting}. The first family of distributions perturbs the null distribution's mean \begin{equation}\label{eq:family_perturb_first_moment}
\pi=\delta_{p_0+\epsilon} \quad \for 0 \leq \epsilon \leq 1-p_0 \comma
\end{equation} the second family matches the mean \begin{equation}\label{eq:family_match_first_moment}
\pi = \frac{1}{2} \cdot \left(\delta_{p_0+\epsilon} + \delta_{p_0-\epsilon} \right) \quad \for 0 \leq \epsilon \leq p_0 \comma
\end{equation} while the third family perturbs the probability assigned to the null hypothesis's point mass \begin{equation}\label{eq:family_perturb_probability}
\pi = (1-\epsilon)\cdot \delta_{p_0} + \epsilon \cdot \delta_1 \quad \for 0 \leq \epsilon \leq 1 \period
\end{equation} \zcref[S]{fig:homogeneity_power_per_lb} shows the power of the tests as a function of the distance to the null distributions $\pi_0=\delta_{0.5}$. In all cases, the local minimax test achieves good power.

\begin{figure}[!ht]
\centering
\includegraphics[width=0.85\linewidth]{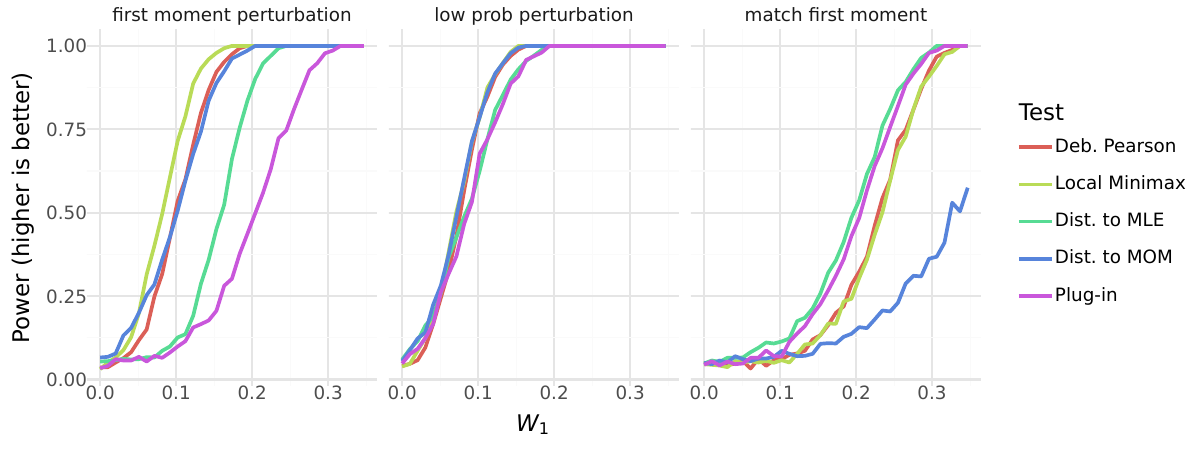}
\caption{Power of the test as a function of the distance $W_1(\pi,\pi_0)$, where the null distribution is $\pi_0=\delta_{0.5}$ and the alternative distribution $\pi$ belongs to the family of distributions \eqref{eq:family_perturb_first_moment} on the left panel, \eqref{eq:family_perturb_probability} on the center panel, and \eqref{eq:family_match_first_moment} on the right panel.}
\label{fig:homogeneity_power_per_lb}
\end{figure}

\zcref[S]{fig:homogeneity_w1_per_lb} displays the local critical separation for \eqref{eq:family_perturb_first_moment},\eqref{eq:family_match_first_moment}, and \eqref{eq:family_perturb_probability}. Each family plays a specific role in capturing the behavior predicted by \eqref{eq:local_critical_separation}. The family \eqref{eq:family_perturb_probability} limits the power for small null hypotheses close to the origin, \eqref{eq:family_perturb_first_moment} captures the linear growth with respect to $p_0$ for moderate $p_0$ values, while \eqref{eq:family_match_first_moment} captures the $\sqrt{p_0}$ dependence for null hypotheses that are close to $0.5$.

\begin{figure}[!ht]
\centering
\includegraphics[width=\linewidth]{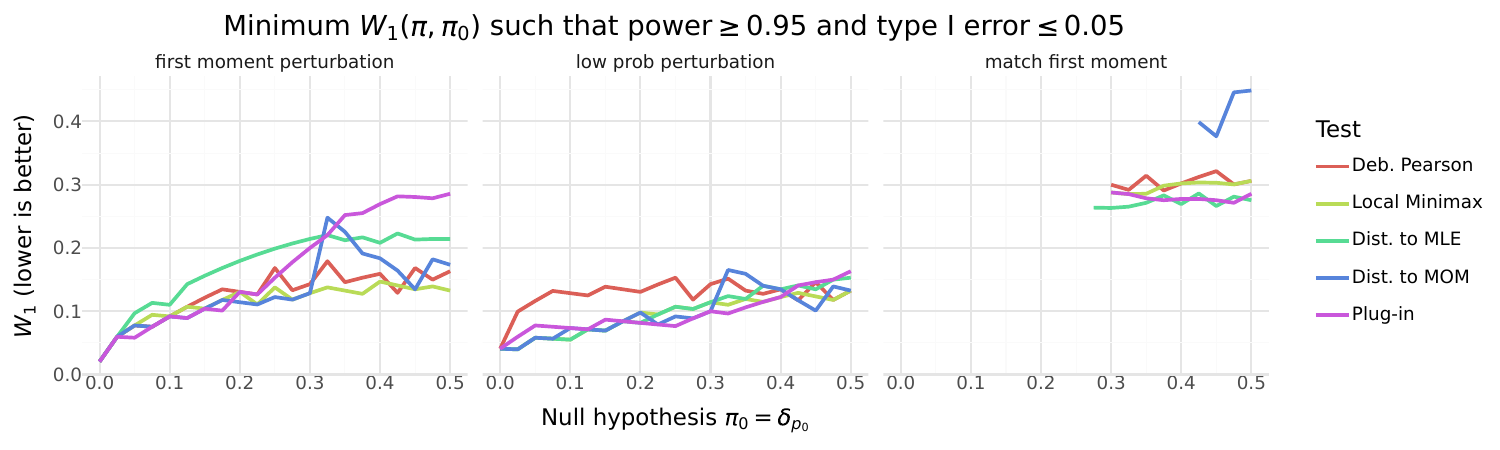}
\caption{Minimum $W_1$ required to obtain high power and low type I error for the families of distributions \eqref{eq:family_perturb_first_moment},\eqref{eq:family_match_first_moment}, and \eqref{eq:family_perturb_probability} as a function of the null distribution. For the right panel, if a line is not drawn for a given null distribution, it signifies that the corresponding test could not control the type I error.}
\label{fig:homogeneity_w1_per_lb}
\end{figure}

\zcref[S]{fig:homogeneity_w1} shows the empirical local critical separation when considering all families of distributions, approximating the local critical separation in \eqref{eq:small_t_rates}. All tests have a comparable performance when considering all three kinds of mixing distributions.
\begin{figure}[!ht]
\centering
\includegraphics[width=0.9\linewidth]{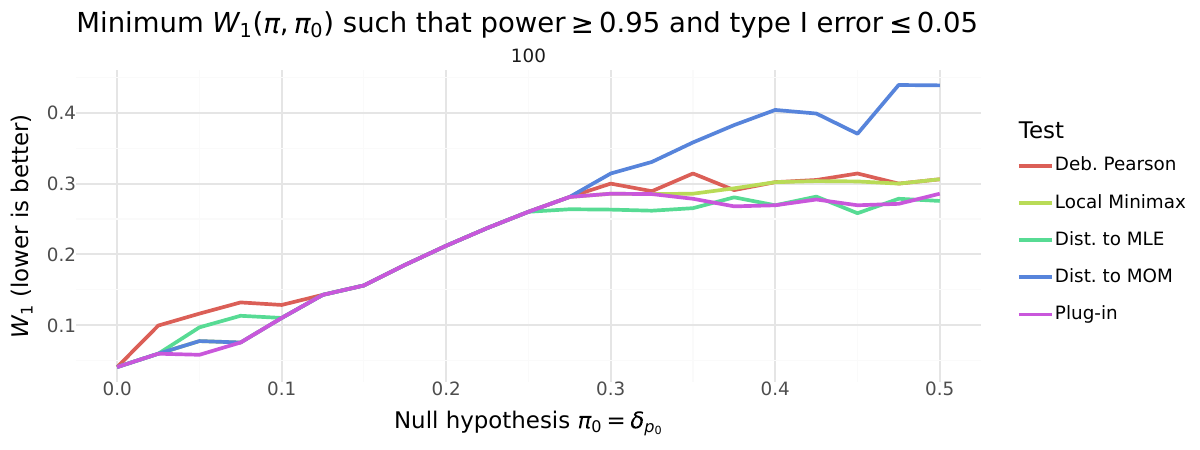}
\caption{Minimum $W_1$ required to obtain high power and low type I error as a function of the null distribution.}
\label{fig:homogeneity_w1}
\end{figure}

\subsection{Homogeneity testing without a reference effect}\label{sec:homogeneity_testing_unkown_simulations}


Three statistics are evaluated through simulation. The first is the Cochran's $\chi^2$ test statistic \citep{cochranMethodsStrengtheningCommon1954} $$
t\cdot \frac{\tilde{V}(X)}{\mu_{\hat{m}_1(X)}},$$ where $\tilde{V}(X)$ is the empirical variance $$\tilde{V}(X) = \frac{1}{n-1}\sum_{i=1}^n\left(\hat{m}_1(\pi)-\frac{X_i}{t}\right)^2.$$ The statistic is undefined when the numerator and the denominator are both zero, which happens with probability one when the mixing distribution is a point mass at the origin. To fix this issue, we threshold the denominator whenever it goes below $\gamma$: $$
C(X)=t\cdot \frac{\tilde{V}(X)}{ \max\left(\mu_{\hat{m}_1(X)} , \gamma\right)}.$$ Henceforth, we refer to the above as the modified Cochran's $\chi^2$ test statistic. We set $\gamma = 10^{-10}$ for all simulations in this section. We consider both the finite sample and asymptotically valid Cochran's $\chi^2$ tests \begin{equation}\label{eq:modified_cochran_chi2_test}
\psi_{C}(X) = I\left(C(X) \geq \sup_{\pi \in S}q_{\alpha}(P_{\pi},C)\right) \textand \psi_{AC}(X) = I\left(C(X) \geq \frac{\chi^2_{n-1,\alpha}}{n-1}\right) \comma
\end{equation} where $\chi^2_{n-1,\alpha}$ is the $1-\alpha$ quantile of the chi-squared distributions with $n-1$ degrees of freedom. The second statistic is based on \eqref{eq:truncated_test} but using the same normalization as the modified Cochran's $\chi^2$, we call the corresponding test the debiased Cochran's $\chi^2$ (V.1), \begin{equation}\label{eq:r1_test}
R_1(X) = t \cdot \frac{\hat{V}(X)}{\max\left(\mu_{\hat{m}_1(X)},\gamma\right)} \period \textand \psi_{R_1}(X) = I\left( R_1(X) \geq \sup_{\pi \in S}q_{\alpha}(P_{\pi},R_1)\right) \period
\end{equation} Finally, we consider \eqref{eq:truncated_test} and define the corresponding test analogously, called debiased Cochran's $\chi^2$ (V.2) \begin{equation}\label{eq:r2_test}
R_2(X) = t \cdot \frac{\hat{V}(X)}{\max\left(\hat{\mu}(X),\gamma\right)} \textand
\psi_{R_2}(X) = I\left( R_2(X) \geq \sup_{\pi \in S}q_{\alpha}(P_{\pi},R_2)\right) \period
\end{equation} We remark that $\psi_{R_1}$ and $\psi_{R_2}$ are expected to have similar performance since any denominator that estimates $\max(m_1(\pi),1-m_1(\pi))$ should work. \zcref[S]{fig:unknown_homogeneity_distribution} shows the distribution of the three statistics under the distribution $P_\pi$ where $\pi=\delta_{0.5}$. The bias is substantially reduced when using $R_{1}$ or $R_2$ rather than Cochran's statistic. The same can be observed for point masses closer to the origin; see \zcref[S]{sec:distribution_simulation_homogeneity_testing}. 
\begin{figure}[!ht]
\centering
\includegraphics[width=0.95\linewidth]{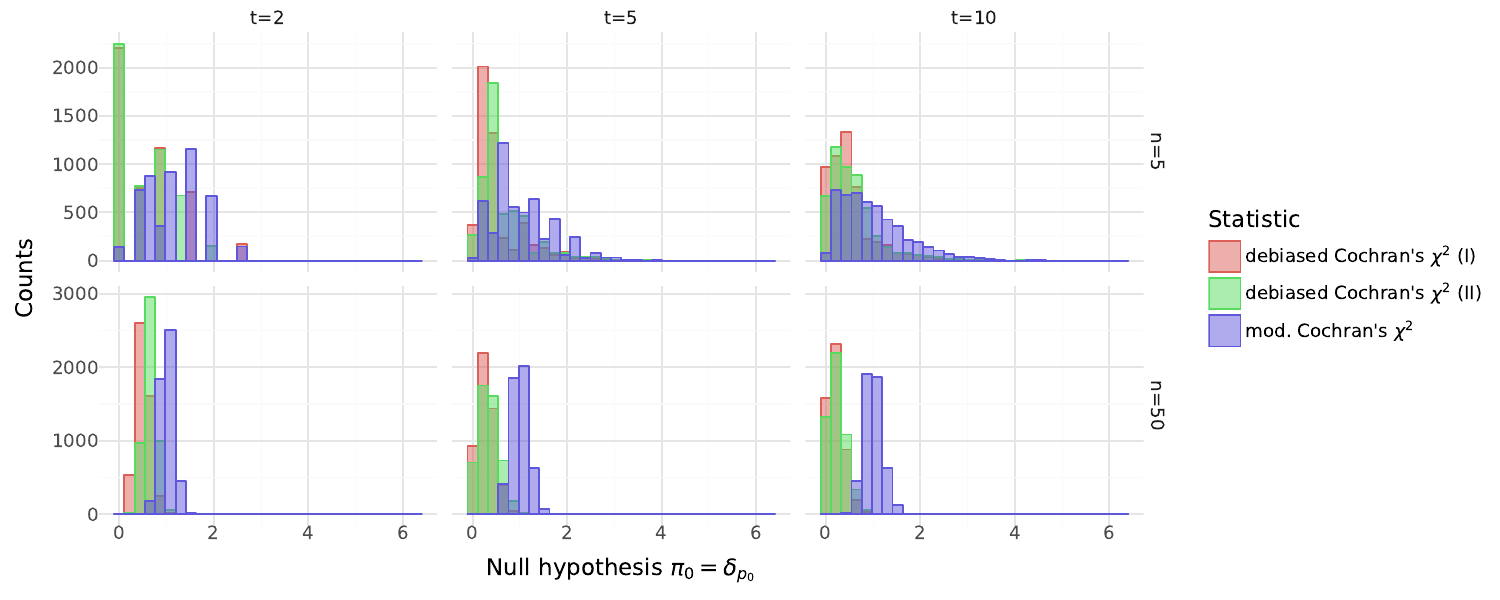}
\caption{Distribution of modified and debiased Cochran's $\chi^2$ test statistics under the mixing distribution $\pi=\delta_{0.5}$.}
\label{fig:unknown_homogeneity_distribution}
\end{figure}

\zcref[S]{fig:unkown_homogeneity_quantiles} displays the $1-\alpha$ quantile of the statistics for $\alpha=0.05$ for mixing distributions $\pi=\delta_{p_0}$ where $p_0 \in [0,0.5]$. Additionally, \zcref{fig:unknown_homogeneity_validity} illustrates that using the maximum quantile over $S$ leads to conservative thresholds for point masses close to the origin.

\begin{figure}[!ht]
\centering
\includegraphics[width=\linewidth]{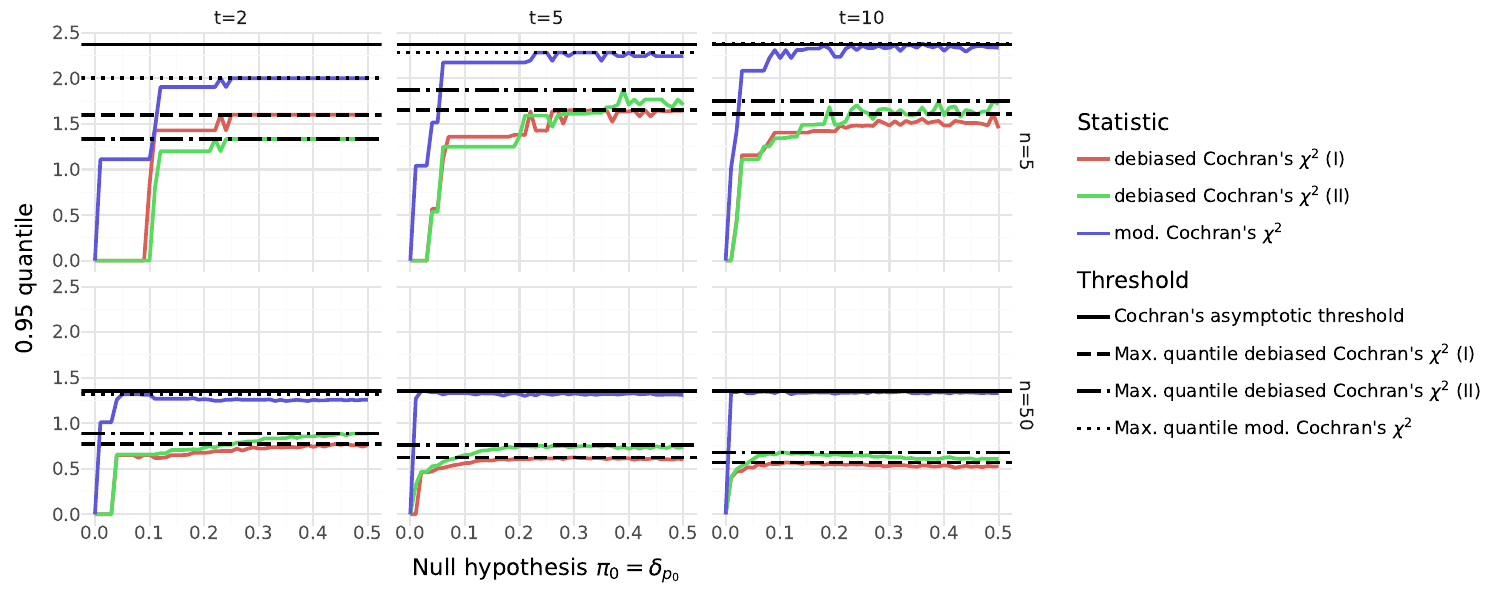}
\caption{0.95\% quantile as a function of the null hypothesis distribution for the studied statistics. Additionally, the asymptotic threshold used by Cochran's $\chi^2$ test \eqref{eq:modified_cochran_chi2_test} and the maximum quantile of each statistic are displayed.}
\label{fig:unkown_homogeneity_quantiles}
\end{figure}

\begin{figure}[!ht]
\centering
\includegraphics[width=\linewidth]{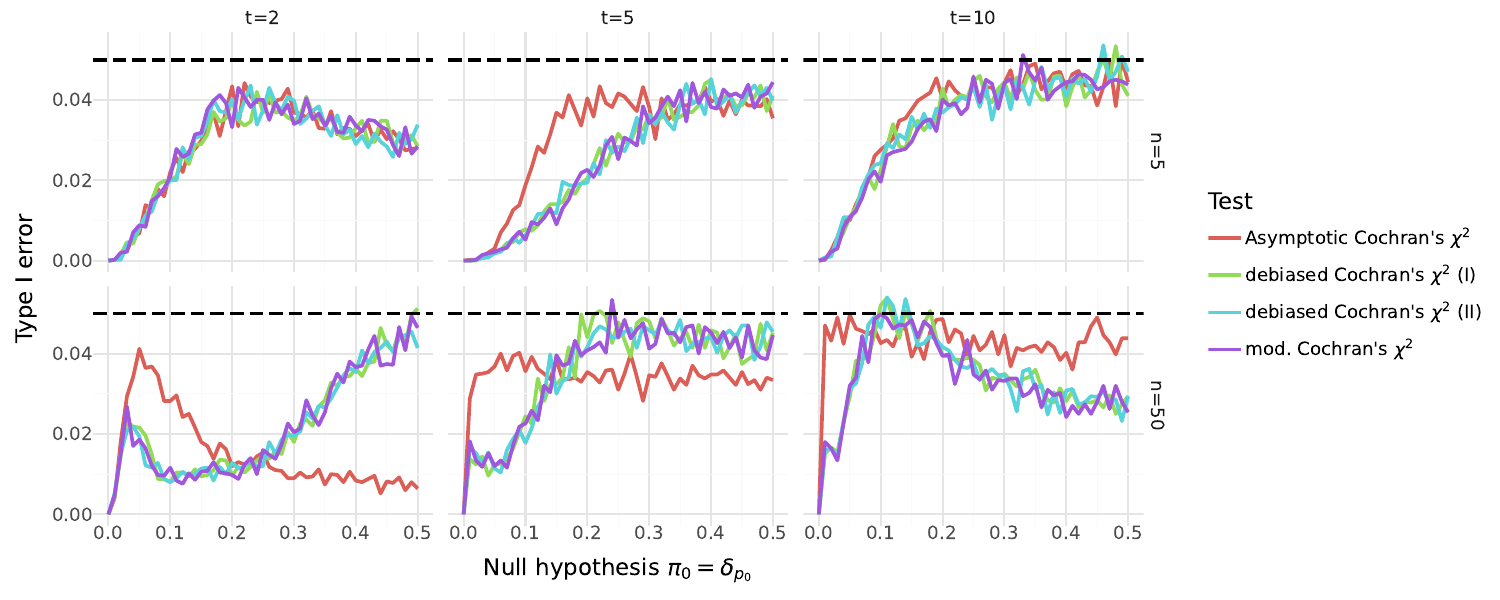}
\caption{Type I error of tests \eqref{eq:modified_cochran_chi2_test},\eqref{eq:r1_test} and \eqref{eq:r2_test} as a function of the null distribution. Note the conservative behavior whenever the number of trials $t$ is small.}
\label{fig:unknown_homogeneity_validity}
\end{figure}

To analyze the power of the proposed tests, we consider the families of distributions \eqref{eq:family_match_first_moment} and \eqref{eq:family_perturb_probability} for $p_0=0.5$, which were used for constructing the lower-bound \zcref{lemma:global_minimax_rates_homogeneity}. In our simulations, all tests agree except when the mixing distribution's moments cannot be reliably estimated, i.e., when $t$ is small relative to $n$. In that case, the tests valid for finite samples improve over the asymptotic Cochran's $\chi^2$ test as shown in \zcref[S]{fig:unknown_homogeneity_power}. \begin{figure}[ht]
\centering
\includegraphics[width=0.8\linewidth]{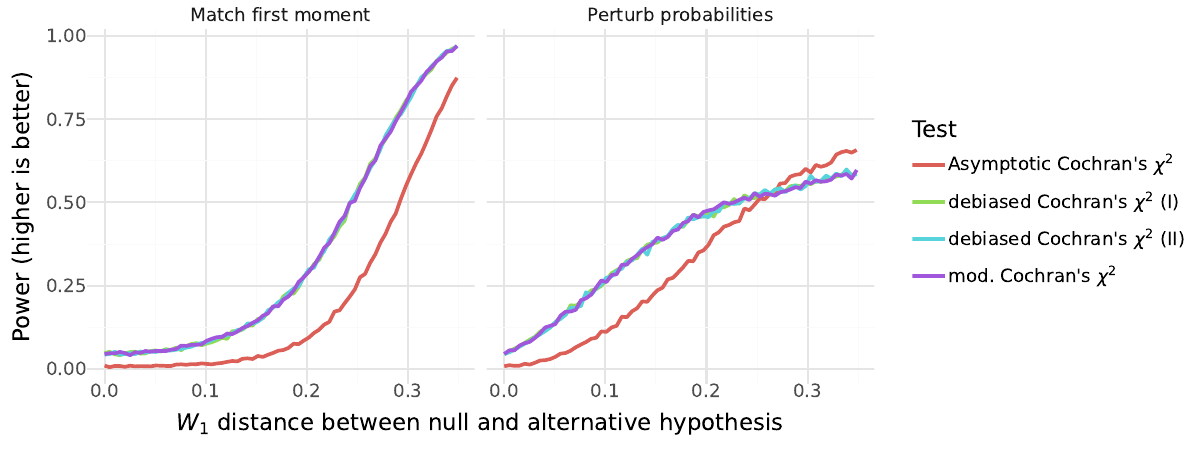}
\caption{Power of the test as a function of the distance $W_1(\pi,\pi_0)$ for $t=2$ and $n=50$. The null mixing distribution is $\pi_0=\delta_{0.5}$, and the alternative mixing distribution $\pi$ belongs to the family of distributions \eqref{eq:family_perturb_probability} on the left panel, and \eqref{eq:family_match_first_moment} on the right panel. }
\label{fig:unknown_homogeneity_power}
\end{figure} We conclude that any of the considered test statistics might be used in practice.

\subsection{Additional simulations for homogeneity testing without a reference effect}\label{sec:distribution_simulation_homogeneity_testing}

\begin{figure}[H]
\centering
\includegraphics[width=\linewidth]{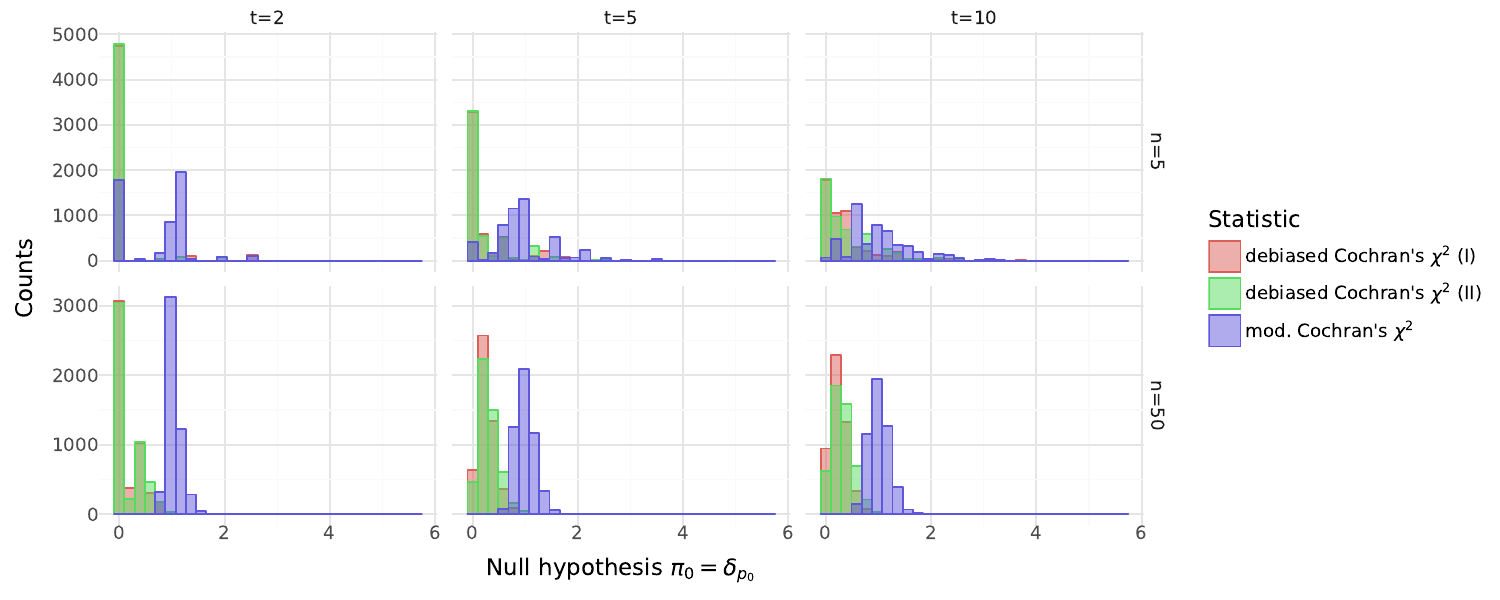}
\caption{Distribution of modified and debiased Cochran's $\chi^2$ test statistics under the mixing distribution $\pi=\delta_{0.1}$.}
\end{figure}

\begin{figure}[H]
\centering
\includegraphics[width=\linewidth]{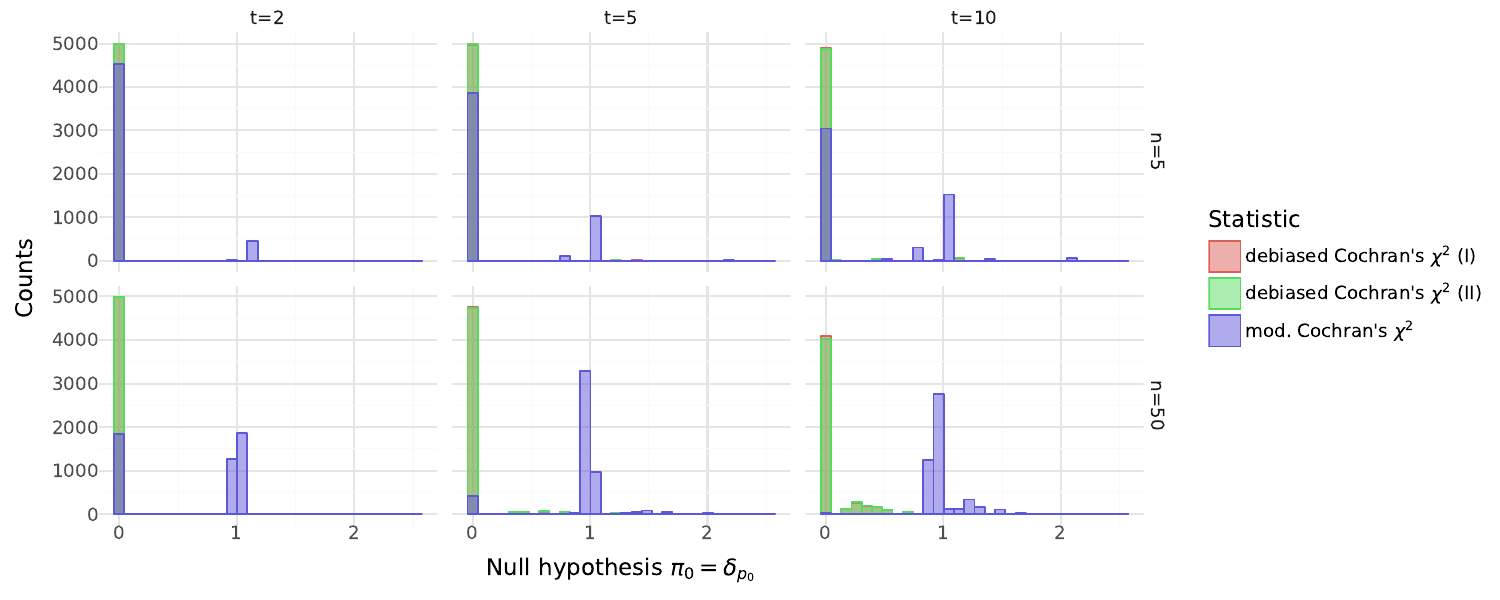}
\caption{Distribution of modified and debiased Cochran's $\chi^2$ test statistics under the mixing distribution $\pi=\delta_{0.01}$.}
\end{figure}

\section{Modeling political outcomes from 1976 to 2004}\label{sec:county}

\citet{tian2017} and \citet{vinayakMaximumLikelihoodEstimation2019} consider the problem of estimating the underlying mixing distribution of the counties' political leaning. They analyze data from 3,109 U.S. counties, recording the number of times the Republican Party won in each county during the eight presidential elections from 1976 to 2004. \citet{tian2017} used the method of moments (MOM) to estimate the mixing distribution, while \citet{vinayakMaximumLikelihoodEstimation2019} applied maximum likelihood estimation (MLE). \citet{vinayakMaximumLikelihoodEstimation2019} noted that the MLE and MOM estimators look qualitatively different, although they match their first 8 moments. Here, we extend their analysis by applying the goodness-of-fit tests from \zcref[S]{sec:general_testing} to assess which method better models the underlying distribution.

\begin{figure}[H]
\centering
\includegraphics[width=\linewidth]{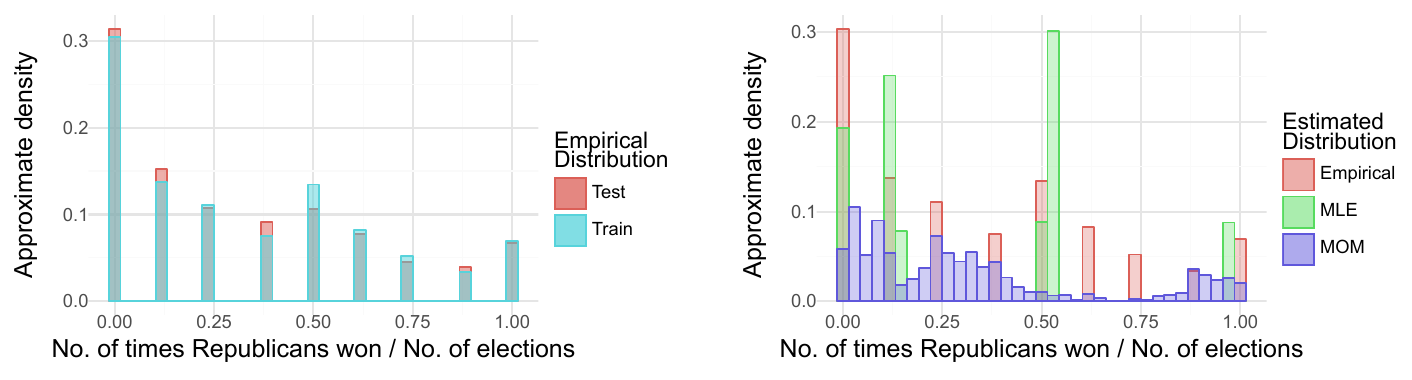}
\caption{(Left panel) Empirical distribution of the training and test dataset. (Right panel) Estimated distributions together with the empirical distribution of the training dataset.}
\label{fig:counties_data}
\end{figure}

We randomly split the counties into two halves, each containing approximately $n \approx 1,555$ counties with $t = 8$. We refer to these halves as the training and test datasets, shown in \zcref{fig:counties_data}. Using the training dataset, we compute the empirical, MLE, and MOM estimators of the mixing distribution, illustrated in the right panel of \zcref{fig:counties_data}. Although the MLE and MOM estimators match their first eight moments, they yield distinct qualitative interpretations of the data; see \zcref{fig:counties_moments}. The MLE estimator clusters around three points, corresponding to counties that consistently vote Democratic, consistently vote Republican, or are swing counties. In contrast, the MOM estimator suggests greater heterogeneity in political leanings.

\begin{figure}[H]
\centering
\includegraphics[width=0.55\linewidth]{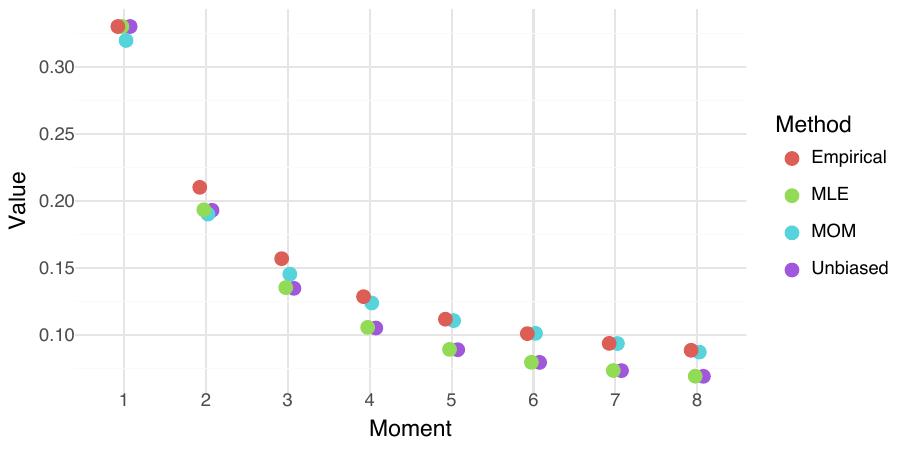}
\caption{First 8 moments of the empirical, MLE, and MOM estimators of the mixing distribution on the first half of the data. Additionally, the unbiased estimator of each moment is presented; see their definition in \eqref{eq:unbiased_moment_estimator} of \zcref[S]{sec:gof_test_definitions}.}
\label{fig:counties_moments}
\end{figure}

We use the second half of the data to evaluate whether the goodness-of-fit tests from \zcref[S]{sec:general_testing} can distinguish the estimated distributions from the observed data. The P-value represents the smallest significance level at which a test rejects the null hypothesis that the second half of the data follows the null mixing distribution. A higher P-value suggests the test finds it harder to differentiate the data from the null distribution. \zcref[S]{tab:counties_pvalues} compiles the computed P-values for each test and null distribution. Most tests reject the MOM and empirical estimates of the mixing distribution. Therefore, we may conclude that the MLE provides a more accurate fit. Additionally, the global minimax,
Debiased Pearson’s $\chi^2$, modified Pearson’s $\chi^2$, and modified likelihood ratio test agree in all cases.

\begin{table}[H]
    \centering
\begin{tabular}{lrrr}
\toprule
Test & MLE & MOM & Empirical \\
\midrule
Plug-in \eqref{eq:w1_plugin_test} & 0.07 & 0 & 0.99 \\
Global minimax \eqref{eq:global_minimax_test}  & 0.13 & 0 & 0\\
Debiased Pearson's $\chi^2$ \eqref{eq:fingerprint_test} & 0.14 & 0 & 0 \\
Modified Pearson's $\chi^2$ \eqref{eq:modified_pearson_chi2} & 0.17 & 0 & 0 \\
Modified likelihood ratio \eqref{eq:modified_LRT} & 0.16 & 0 & 0 \\
Distance to MOM estimator \eqref{eq:dist_to_MOM} & 0.37 & 0.94 & 0.26 \\
Distance to MLE estimator \eqref{eq:dist_to_MLE} & 0.65 & 0 & 0.10 \\
\bottomrule
\end{tabular}
\caption{P-values for goodness-of-fit tests of the MLE, MOM, and empirical mixing distributions on the second half of the county-level data. See \zcref[S]{sec:general_testing} and \zcref[S]{sec:gof_test_definitions} for their definitions.}
\label{tab:counties_pvalues}
\end{table}


\section{Definitions of testing procedures}

The following sections compile the definitions used across the simulations and applications for Goodness-of-fit testing (\zcref{sec:gof_test_definitions}), homogeneity testing with a reference effect (\zcref{appx:homogeneity_testing_simple_null}), and homogeneity testing without a reference effect (\zcref{appx:homogeneity_testing_composite_null}).

\subsection{Goodness-of-fit testing}\label{sec:gof_test_definitions}

Recall that $n_j$ is the observed fingerprint \eqref{eq:observed_fingerprint} and $b_{j,t}(\pi)$ is the expected fingerprint \eqref{eq:expected_fingerprint}. It holds that \begin{equation}
n_j \sim \Bin(n,b_{j,t}(\pi)) \comma E\left[\frac{n_j}{n}\right]=b_{j,t}(\pi) \textand V\left[\frac{n_j}{n}\right]=\frac{\mu_{b_{j,t}(\pi)}}{n}.
\end{equation} In the following, we define some test statistics that appear in the simulations but are not defined in the main text.

\paragraph{Modified Pearson's $\chi^2$ test} \begin{equation}\label{eq:modified_pearson_chi2}
T = \frac{1}{t+1}\sum_{j=0}^t \frac{\left(\frac{n_j}{n}-b_{j,t}(\pi_0)\right)^2}{\max\left(\mu_{b_{j,t}(\pi_0)},\gamma \right)} \textand \psi = I\left(T \geq q_{\alpha}\left(P_{\pi_0},T\right)\right).
\end{equation} The unmodified test uses $\gamma=0$, for our simulations, we use $\gamma =10^{-10}$.

\paragraph{Modified Likelihood Ratio test (LRT)}  \begin{equation}\label{eq:modified_LRT}
T = \left| \frac{1}{t+1}\sum_{j=0}^t \frac{n_j}{n}\cdot \log\left(\frac{\max\left(n_j/n,\gamma\right)}{\max\left(b_{j,t}(\pi_0),\gamma\right)}\right) \right| \textand \psi = I\left(T \geq q_{\alpha}\left(P_{\pi_0},T\right)\right).
\end{equation}The unmodified test uses $\gamma=0$, we use $\gamma =10^{-10}$.

\paragraph{Modified Maximum Likelihood Estimation (MLE) test}

The Maximum Likelihood Estimator (MLE) \citep{vinayakMaximumLikelihoodEstimation2019} is defined as any solution to the following optimization \begin{align}
\hat{\pi}_{\text{MLE}} \in &\argmax_{\pi \in D} \sum_{i=1}^n \log\left(E_{p\sim \pi}\left[\binom{t}{X_i}p^{X_i}(1-p)^{t-X_i}\right]\right)\\
&=\argmax_{\pi \in D} \sum_{j=0}^t \frac{n_j}{n}\cdot \log\left( E_{p\sim \pi}\left[\binom{t}{j}p^j(1-p)^{t-j}\right]\right).
\end{align} It minimizes the Kullback–Leibler divergence
between the observed and expected fingerprints \begin{equation}
\hat{\pi}_{\text{MLE}} \in \argmin_{\pi \in D} \text{KL}\left(\hat{b},b(\pi)\right)  \where \hat{b}_j = \frac{n_j}{n} \textand b_{j,t}(\pi)=E_{p\sim \pi}\left[\binom{t}{j}p^j(1-p)^{t-j}\right].
\end{equation} To avoid numerical issues during the simulation, we implement the following modified MLE optimization \begin{equation}
\hat{\pi}_{\text{MLE}} \in \argmax_{\pi \in D} \sum_{j=0}^t \frac{n_j}{n}\cdot \log\left( \max\left(\gamma\ ,\  E_{p\sim \pi}\left[\binom{t}{j}p^j(1-p)^{t-j}\right]\right)\right).
\end{equation} The original statistic can be recovered using $\gamma=0$, we use $\gamma =10^{-10}$. The corresponding test is defined as follows
\begin{equation}\label{eq:dist_to_MLE}
T = W_1(\hat{\pi}_{\text{MLE}},\pi_0) \textand \psi = I\left(T \geq q_{\alpha}\left(P_{\pi_0},T\right)\right).
\end{equation}

\paragraph{Method of Moments (MOM) test}

The Method of Moments (MOM) \citep{tian2017} minimizes the $L_1$ distance between the observed and expected moments. Let $\hat{m}_j(X)$ be the unbiased estimator of the $j$-th moment \begin{equation}\label{eq:unbiased_moment_estimator}
\hat{m}_j(X) = \frac{1}{n}\sum_{i=1}^n \frac{\binom{X_i}{j}}{\binom{t}{j}} \quad \for j \in \{1,\dots,t\}\comma
\end{equation} then the MOM estimator is defined as any solution to the following optimization \begin{equation}
\hat{\pi}_{\text{MOM}} \in \argmin_{\pi \in \mathcal{D}} \sum_{j=0}^t \left|E_{p \sim \pi}[p^j] - \hat{m}_j(X)\right|.
\end{equation} The corresponding test is defined as follows \begin{equation}\label{eq:dist_to_MOM}
T = W_1(\hat{\pi}_{\text{MOM}},\pi_0) \textand \psi = I\left(T \geq q_{\alpha}\left(P_{\pi_0},T\right)\right).
\end{equation}

\subsection{Homogeneity testing with a reference effect}\label{appx:homogeneity_testing_simple_null}

In the following, let $\pi_0=\delta_{p_0}$.

\paragraph{$\ell_2$ test}
The following test statistic is used
\begin{equation}
T = \frac{1}{n}\sum_{i=1}^n \left(\frac{X_i}{t_i}-p_0\right)^2.
\end{equation} The corresponding test is $
\psi = I\left(T \geq q_{\alpha}\left(P_{\pi_0},T\right)\right)$.

\paragraph{Modified Pearson's $\chi^2$ test}
The following test statistic is usually employed
\begin{equation}
T = \frac{1}{n}\sum_{i=1}^n t_i \cdot \frac{\left(\frac{X_i}{t_i}-p_0\right)^2}{\mu_{p_0}}
\end{equation} To avoid numerical instabilities in the application, we utilized the following modification \begin{equation}
T = \frac{1}{n}\sum_{i=1}^n t_i \cdot \left(\frac{X_i}{t_i}-p_0\right)^2\textand \psi = I\left(T \geq q_{\alpha}\left(P_{\pi_0},T\right)\right)
\end{equation}

\paragraph{Modified Likelihood Ratio test (LRT)} The test statistic is computed as follows  \begin{equation}
T = \left| \frac{1}{n}\sum_{i=1}^n X_i\cdot \log\left(\frac{\max\left(X_i/t_i,\gamma\right)}{\max\left(p_0,\gamma\right)}\right) + (t_i-X_i)\cdot \log\left(\frac{\max\left(1-X_i/t_i,\gamma\right)}{\max\left(1-p_0,\gamma\right)}\right)\right|.
\end{equation} The corresponding test is $
\psi = I\left(T \geq q_{\alpha}\left(P_{\pi_0},T\right)\right)$. The unmodified test uses $\gamma=0$, we use $\gamma =10^{-10}$.

\paragraph{Modified Maximum Likelihood Estimation (MLE) test} To avoid numerical issues during the simulation, we implement the following modified MLE optimization \begin{align}
\hat{\pi}_{\text{MLE}} \in \argmax_{\pi \in D} \sum_{i=1}^n \log\left(\max\left(\gamma\ ,\  E_{p\sim \pi}\left[\binom{t_i}{X_i}p^{X_i}(1-p)^{t_i-X_i}\right]\right)\right)\period
\end{align} The original statistic can be recovered using $\gamma=0$, we use $\gamma =10^{-10}$. The corresponding test is defined as follows
\begin{equation}
T = W_1(\hat{\pi}_{\text{MLE}},\pi_0) \textand \psi = I\left(T \geq q_{\alpha}\left(P_{\pi_0},T\right)\right)\period
\end{equation}

\paragraph{Method of Moments (MOM) test} Let $
t_{\max}=\max_{1\leq i\leq n} t_i$, and any MOM estimator
\begin{equation}
\hat{\pi}_{\text{MOM}} \in \argmin_{\pi \in \mathcal{D}} \sum_{j=0}^{t_{\max}}  \left|E_{p \sim \pi}[p^j] - \frac{\sum_{i=1}^n \frac{\binom{X_i}{j}}{\binom{t_i}{j}}\cdot I(j \leq t_i)}{\sum_{i=1}^n I(j \leq t_i)}\right|\period
\end{equation} The corresponding test is defined as follows \begin{equation}
T = W_1(\hat{\pi}_{\text{MOM}},\pi_0) \textand \psi = I\left(T \geq q_{\alpha}\left(P_{\pi_0},T\right)\right)\period
\end{equation}

\subsection{Homogeneity testing without a reference effect}\label{appx:homogeneity_testing_composite_null}

In the following definitions, let $\gamma = 10^{-10}$.

\paragraph{Modified Cochran's $\chi^2$ test} The $\chi^2$ statistic is defined as follows \begin{equation}
\hat{\chi}^2(X)= \frac{\tilde{V}(X)}{\max\left(\mu_{\hat{m}_1(X)}\ ,\ \gamma\right)}  \where \tilde{V}(X) = \frac{1}{n-1}\sum_{i=1}^nt_i\cdot \left(\frac{X_i}{t_i}-\hat{m}_1(X)\right)^2
\end{equation} and \begin{equation}\label{eq:m1_diff_t}
\hat{m}_1(X) = \frac{\sum_{i=1}^n X_i}{\sum_{i=1}^n t_i}.
\end{equation}

We define the modified Cochran's $\chi^2$ test as \begin{equation}
\psi_{\chi^2}(X) = I\left(\hat{\chi}^2(X) \geq \sup_{\pi \in S}q_{\alpha}(P_{\pi},\hat{\chi}^2) \right) \period
\end{equation}

We define the asymptotic modified Cochran's $\chi^2$ test as \begin{equation}
\psi_{\chi^2}(X) = I\left(\hat{\chi}^2(X) \geq \frac{\chi^2_{n-1,\alpha}}{n-1}\right) \period
\end{equation}

\paragraph{Debiased Cochran's $\chi^2$ test (V.1)} The $R_1$ statistic is defined as \begin{equation}
R_1(X) =  \frac{\hat{V}(X)}{\max\left(\mu_{\hat{m}_1(X)},\gamma\right)}
\end{equation} where \begin{equation}\label{eq:v_hat_diff_t}
\hat{V}(X) = \binom{n}{2}^{-1}\sum_{i<j}\frac{h(X_i,X_j)}{2} \st h(X_i,X_j)=\frac{\binom{X_i}{2}}{\binom{t_i}{2}}+\frac{\binom{Y_j}{2}}{\binom{t_j}{2}}-2\frac{\binom{X_i}{1}}{\binom{t_i}{1}}\frac{\binom{X_j}{1}}{\binom{t_j}{1}}
\end{equation} and $\hat{m}_1(X)$ is defined in \eqref{eq:m1_diff_t}. The corresponding test is \begin{equation}
\psi_{R_1}(X) = I\left( R_1(X) \geq \sup_{\pi \in S}q_{\alpha}(P_{\pi},R_1)\right) \period
\end{equation}

\paragraph{Debiased Cochran's $\chi^2$ test (V.2)} The $R_2$ statistic is \begin{equation}
R_2(X) =  \frac{\hat{V}(X)}{\max\left(\hat{\mu}(X),\gamma\right)}
\end{equation} where \begin{equation}
\hat{\mu}(X)= \left(\binom{n}{2}\right)^{-1}\sum_{i<j}\frac{\tilde{h}(X_i,x_j)}{2} \st \tilde{h}(X_i,X_j)= \frac{X_i}{t_i}+\frac{X_j}{t_j}-2\frac{X_i}{t_i}\frac{X_j}{t_j}
\end{equation} and $\hat{V}(X)$ is defined in \eqref{eq:v_hat_diff_t}. The corresponding test is \begin{equation}
\psi_{R_2}(X) = I\left( R_2(X) \geq \sup_{\pi \in S}q_{\alpha}(P_{\pi},R_2)\right) \period
\end{equation}

\etocdepthtag.toc{mtreferences}
\addcontentsline{toc}{section}{References}
\bibliography{TestingRandomEffects}

\end{document}